\documentclass[12pt]{amsart}
\usepackage{t1enc}
\usepackage[latin1]{inputenc}
\usepackage{graphicx}

\usepackage{amsthm, amssymb}
\usepackage{amsfonts}

\newtheorem{theorem}{Theorem}[section]
\newtheorem{lemma}[theorem]{Lemma}
\newtheorem{proposition}[theorem]{Proposition}
\newtheorem{corollary}[theorem]{Corollary}
\theoremstyle{definition}
\newtheorem{definition}[theorem]{Definition}
\newtheorem{example}[theorem]{Example}

\theoremstyle{remark}
\newtheorem{remark}[theorem]{Remark}
\numberwithin{equation}{section}

\newcommand\B{\mathbb{B}}
\newcommand\C{\mathbb{C}}

\newcommand\Z{\mathbb{Z}}

\newcommand\T{\mathbb{T}}
\newcommand\cA{\mathcal{A}}
\newcommand\cB{\mathcal{B}}

\newcommand\cH{\mathcal{H}}

\newcommand\cL{\mathcal{L}}
\newcommand\cM{\mathcal{M}}
\newcommand\cN{\mathcal{N}}
\newcommand\cO{\mathcal{O}}

\newcommand\fA{\mathfrak{A}}
\newcommand\fQ{\mathfrak{Q}}
\newcommand\fS{\mathfrak{S}}

\newcommand\inpr[2]{\langle{#1,#2}\rangle}
\newcommand\biota{\bar{\iota}}
\newcommand\bkappa{\bar{\kappa}}
\newcommand\hgamma{\hat{\gamma}}
\newcommand\hrho{\hat{\rho}}
\newcommand\hxi{\hat{\xi}}
\newcommand\heta{\hat{\eta}}
\newcommand\hzeta{\hat{\zeta}}
\newcommand\bpi{\bar{\pi}}
\newcommand\bnu{\bar{\nu}}
\newcommand\br{\bar{r}}
\newcommand\brho{\bar{\rho}}
\newcommand{\id}{\mathrm{id}}
\newcommand{\tr}{\mathrm{tr}}
\newcommand{\Ad}{\mathrm{Ad}}
\newcommand{\Mor}{\mathrm{Mor}}
\newcommand{\Sect}{\mathrm{Sect}}
\newcommand{\End}{\mathrm{End}}
\newcommand{\Aut}{\mathrm{Aut}}
\newcommand{\Ang}{\mathrm{Ang}}
\newcommand{\Gal}{\mathrm{Gal}}
\newcommand{\opp}{\mathrm{opp}}
\newcommand{\tA}{\tilde{A}}
\newcommand{\tB}{\tilde{B}}
\newcommand{\tC}{\tilde{C}}
\newcommand{\tD}{\tilde{D}}
\newcommand{\tN}{\tilde{N}}
\newcommand{\tM}{\tilde{M}}
\newcommand{\tP}{\tilde{P}}
\newcommand{\tQ}{\tilde{Q}}
\newcommand{\tR}{\tilde{R}}
\newcommand{\tfQ}{\tilde{\fQ}}
\newcommand{\tlambda}{\tilde{\lambda}}
\newcommand{\quadri}{\begin{array}{ccc}
P & \subset& M \\
\cup & & \cup \\
N &\subset &Q 
\end{array}}
\newcommand{\tquadri}{\begin{array}{ccc}
\tP & \subset& \tM \\
\cup & & \cup \\
\tN &\subset &\tQ 
\end{array}}

\newcommand{\hpic}[2]{\mbox{$\begin{array}[c]{l} 
\includegraphics*[height=#2]{figs/#1}
\end{array}$}}
 
\newcommand{\vpic}[2]{\mbox{$\begin{array}[c]{l} 
\includegraphics*[width=#2]{figs/#1}
\end{array}$}}

\begin{document}

\title{Classification of noncommuting quadrilaterals of factors}
\author{Pinhas Grossman}
\author{Masaki Izumi}
\thanks{Work supported by JSPS}
\maketitle

\begin{abstract} A quadrilateral of factors is an irreducible 
inclusion of factors $N \subset M$ with intermediate subfactors 
$P$ and $Q$ such that $P$ and $Q$ generate $M$ and the intersection 
of $P$ and $Q$ is $N$. 
We investigate the structure of a noncommuting quadrilateral of factors 
with all the elementary inclusions $P\subset M$, $Q\subset M$, $N\subset P$, 
and $N\subset Q$ 2-supertransitive. 
In particular we classify noncommuting quadrilaterals with the indices of the 
elementary subfactors less than or equal to 4. 
We also compute the angles between $P$ and $Q$ for quadrilaterals coming 
from $\alpha$-induction and asymptotic inclusions. 
\end{abstract}



\section{Introduction}

If $M$ is a factor with a finite group $G$ acting by outer automorphisms, 
then $N=M^G$, the algebra of fixed points of the action, is an irreducible 
subfactor of $M$ with index equal to the order of $G$. 
It is well known that in this case the intermediate subalgebras 
$N \subset P \subset M$ are precisely the fixed point algebras of 
the subgroups of $G$ \cite{NT}. 

In this spirit one thinks of the classification of intermediate subfactors 
as a ``noncommutative Galois theory''. 
Motivated by von Neumann's study of projection lattices in Hilbert space 
as a ``continuous geometry'', Watatani proposed studying lattices of 
intermediate subfactors, which are finite for irreducible finite-index 
inclusions, as a quantization \cite{Wat}. 
Sano and Watatani introduced the notion of angles between subfactors, 
a numerical invariant which measures the degree of noncommutativity of 
pairs of subfactors \cite{SW}.  

A major step towards the classification of intermediate subfactors was 
Bisch's characterization of intermediate subfactors as \textit{biprojections} 
in planar algebras- elements of the standard invariant which are projections 
in both of the dual algebraic structures \cite{Bs1}. 
Bisch and Jones then described a generic construction of intermediate 
subfactors by constructing a universal planar algebra generated by a single biprojection, 
the Fuss-Catalan algebra with parameters corresponding to the indices \cite{BJ1}. 
These intermediate subfactors are \textit{supertransitive} in the sense 
that the standard invariants are minimal. 

It was hoped that a similar generic planar algebra for multiple intermediate 
subfactors could be described in terms of indices and invariants such as angles. 
Indeed, a tensor product does yield a generic construction of commuting pairs of 
intermediate subfactors, but constructing pairs with nontrivial angles has proven 
more difficult. 
In \cite{GJ}, Jones and the first-named author showed that there are essentially only $2$ examples of noncommuting irreducible quadrilaterals 
of factors such that all of the elementary inclusions are supertransitive. One is a quadrilateral of subgroups of the symmetric group $\fS_3 $ and the other comes from the GHJ family of subfactors and all the elementary inclusions have index $2+\sqrt{2} $. 

It is known that the structure of an intermediate subfactor $N\subset P\subset M$ is 
determined by the relationship between the two systems of $P$-$P$ bimodules arising from 
$N\subset P$ and $P\subset M$ (see \cite{BH},\cite{IKo1}). 
After \cite{GJ} appeared, the second-named author gave a short proof of the above mentioned 
result using simple computation of $P-P$ bimodules (see Remark \ref{4remrak no-extra}). 
It turns out from the proof that the assumption in \cite{GJ} is too restrictive from the 
view point of bimodules. 
Therefore to capture more interesting and general structure, one should relax the assumption 
of supertransitivity. 
Indeed, the first named author \cite{G} recently obtained a complete classification 
result of irreducible non-commuting quadrilaterals with $N\subset P$ and 
$N\subset Q$ isomorphic to the Jones subfactor of index less than 4 (without posing any 
assumption on $P\subset M$ or $Q\subset M$) and showed that a series of such 
quadrilaterals exists. 

In this paper we investigate quadrilaterals whose elementary inclusions 
$P\subset M$, $Q\subset M$, $N\subset P$, and $N\subset Q$ are each 2-supertransitive (an inclusion $L \subset K$ is 2-supertransitive if the complement of $L$ in the $L-L$ bimodule decomposition of $K$ is irreducible.) In this case the systems of $P-P $ bimodules for $N \subset P  $  and for $P \subset M$ are each generated by a single irreducible $P-P$ bimodule, and the question is how these two bimodules are related.             

A key tool in this analysis is the notion of second cohomology for subfactors, introduced by the second named author and Kosaki in \cite{IKo2}, which counts the inner conjugacy classes of subfactors sharing the same basic extension (as a bimodule class). For a quadrilateral with $N \subset P$ and $N \subset Q $ 2-supertransitive, noncommutativity is equivalent to the existence of an $N-N$ bimodule isomorphism from $P$ to $Q$ \cite{GJ}. If $N \subset P$ (dual) has trivial second cohomology, this $N-N$ bimodule isomorphism can actually be realized as an algebra isomorhism. In particular they showed that any 3-supertransitive subfactor has trivial second cohomology.  

Ultimately this leads to a classification of noncommuting quadrilaterals whose elementary subfactors are 2-supertransitive with trivial second cohomology into two basic types: those which are not cocommuting (``cocommuting'' means that the quadrilateral of commutants is a commuting square) and have $[M:P]=[P:N] $, like the $(2+\sqrt{2})^2$ example, and those which are cocommuting and have $[M:P]=[P:N]-1 $, like the $\fS_3$ example. The latter type is further subdivided according to the Galois group of $N \subset M$, which must be a subgroup of $\fS_3$. In the case that the elementary subfactors have index less than or equal to $4$ we examine all the possibilities and arrive at the following result.

\begin{theorem}
There are exactly seven noncommuting irreducible quadrilaterals of hyperfinite II$_1$ factors whose elementary inclusions all have indices less than or equal to $4 $, up to conjugacy.

If $ (G_{N \subset P}, G_{P \subset M})$ are the principal graphs of the elementary subfactors, the possible configurations are:

$$(A_7,A_7), \quad (E_7^{(1)},E_7^{(1)})$$
$$(A_5,A_3), \quad (D_6,A_4), \quad (E_7^{(1)}, A_5), \quad (E_6^{(1)},D_4)$$
$$ (D_6^{(1)},A_3)$$
   
\end{theorem}

The fully supertransitive cases are $(A_7,A_7) $, which is the $(2+\sqrt{2})^2 $ example, and $(A_5,A_3 )$, which is the $\fS_3$ example, and all other cases have some extra structure. The new examples include a noncocommuting quadrilateral whose elementary subfactors have index $4$ $(E_7^{(1)},E_7^{(1)})$, a cocommuting quadrilateral whose indices are $(3+\sqrt{5})/2 $ and $(5+\sqrt{5})/2 $ $(D_6,A_4)$, and several cocommuting quadrilaterals arising from group actions. Note that any quadrilateral is either noncommuting, noncocommuting (and hence dual to a noncommuting quadrilateral), or else is both commuting and cocommuting, in which case there is no obstruction to the choices of $N \subset P $ and $N \subset Q$. 

It is also interesting to study such quadrilaterals with larger indices. In fact the exotic Haagerup subfactor appears as an elementary subfactor with maximal supertransitivity in both types of quadrilaterals and the Asaeda-Haagerup subfactor may appear as well. It is still unknown how many examples there are of the various types of quadrilaterals (although there is an infinite series of cocommuting quadrilaterals with 3-supertransitive elementary subfactors, coming from symmetric group actions.)

Second cohomolgy also turns out to be closely related to angles. For a noncommuting quadrilateral with  $N \subset P $ and $N \subset Q $ 2-supertransitive, there are at most two possibilities for the angle for each cohomolgy class of $N \subset P$ (dual). In particular, if $N \subset P $ is 3-supertransitive, the angle is always $\cos^{-1} 1/([P:N]-1) $. We also compute the angles in quadrilaterals arising from $\alpha $-induction (which are identified with the GHJ pairs of \cite{GJ}) and asymptotic inclusions, the former of which includes the forked Temperley-Lieb quadrilaterals of \cite{G}.  

The paper is organized as follows:

Section 2 sets forth the basic notation of quadrilaterals, sectors and Q-systems which are used throughout. We often deal with infinite factors since we are actually studying properties of the standard invariant.

Section 3 studies the relationship between second cohomology for subfactors and angles.

Section 4 proves that quadrilaterals whose elementary subfactors are 2-supertransitive with trivial second cohomology fall into the two basic types discussed above, and that for the cocommuting type the Galois group of the total inclusion is a subgroup of $\fS_3$.

Section 5 analyzes several classes of quadrilaterals according to the results of the previous section and for each class provides an example with maximal supertransitivity. 

Section 6 classifies all noncommuting irreducible quadrilaterals whose elementary subfactors have indices less than or equal to 4.

Section 7 computes the angles for quadrilaterals coming from $\alpha$-induction and asymptotic inclusions.

The Appendix constructs a Q-system from a sector in the Haagerup principal graph and shows that it is unique, filling in a gap for one of the examples in Section 5.

\section{Preliminaries}
We first fix notation used throughout this paper. 
We always assume that Hilbert spaces are separable and 
von Neumann algebras have separable preduals. 

\subsection{Quadrilaterals} 
For an inclusion of II$_1$ factors $N\subset M$, the Jones index, the Jones projection,  
and the trace preserving conditional expectation onto $N$ are denoted by $[M:N]$, 
$e_{N}$ 
and $E_{N}$ respectively \cite{J1}. 
We use the same symbols for an inclusion of properly infinite factors 
$N\subset M$ considering a unique minimum conditional expectation \cite{Ko}, 
\cite{Hi}. 
Although the Jones index $[M:N]$ does not necessarily coincides with the minimum 
index for a II$_1$ inclusion in general, no confusion arises  as we always consider 
extremal inclusions throughout this paper. 
For a II$_1$ factor, we denote by $\tr$ the unique normalize trace. 
For an inclusion of general factors $N\subset M$ of finite index, 
we use the same symbol $\tr$ for the restriction of $E_{N}$ to the 
relative commutant $M\cap N'$. 

Let $N \subset M$ be an inclusion of factors of finite index with associated tower 
$$N=M_{-1} \subset M=M_{0} \subset M_{1} \subset M_2\subset \cdots,$$ 
where $M_{k+1}$, $k \geq 0$ is the von Neumann algebra on the standard 
space of $M_k$ generated by $M_{k}$ and the Jones projection $e_{k+1}=e_{M_{k-1}}$. 
Each $e_{k}$ commutes with $N$, so $\{ 1,e_{1},..,e_{k} \}$ generates a *-subalgebra, 
which we will call $TL_{k+1}$, of the $k^{th}$ relative commutant $N' \cap M_{k}$. The following definition first appeared in \cite{J31}.

\begin{definition}
Call a finite-index subfactor $N \subset M$ 
 $k$-{\it supertransitive} (for $k>1$) if $N' \cap M_{k-1} = TL_k$. 
We will say $N\subset M$ is {\it supertransitive} if it is $k$-supertransitive
for all $k$.   
\end{definition}

Note that $N\subset M$ is $k$-supertransitive if and only if $M\subset M_1$ 
is $k$-supertransitive. 

Sano and Watatani \cite{SW} introduced the notion of angles between two subfactors. 

\begin{definition}
For two subfactors $P$ and $Q$ of a factor $M$, the {\it set of angles} 
$\Ang(P,Q)$ between $P$ and $Q$ is the spectrum of the angle 
operator of the two Jones projections $e_P$ and $e_Q$, that is, the 
spectrum of $\cos^{-1}\sqrt{e_Pe_Qe_P-e_P\wedge e_Q}$ where 
$e_Pe_Qe_P-e_P\wedge e_Q$ is regarded as the operator acting on its support. 
\end{definition}

Note that $\Ang(P,Q)=\Ang(Q,P)$ always holds. 

\begin{definition}
A {\it quadrilateral} of factors $\fQ=\quadri$ is 
an inclusion of factors $N\subset M$ of finite index with two intermediate subfactors 
$P\neq Q$ such that $P$ and $Q$ generate $M$ and $P\cap Q=N$. 
The subfactors $N\subset P$, $N\subset Q$, $P\subset M$, and $Q\subset M$ are said to be 
the elementary subfactors for $\fQ$. 
When $M\cap N'=\C$, the quadrilateral $\fQ$ is said to be {\it irreducible}. 
When $E_P$ commutes with $E_Q$, the quadrilateral $\fQ$ is said to be {\it commuting}. 
Let $M_1=J_MN'J_M$, $\hat{P}=J_MP'J_M$, and $\hat{Q}=J_MQ'J_M$ be the basic extensions of 
$M$ by $N$, $P$, and $Q$ respectively. 
Then $\hat{P}$ and $\hat{Q}$ are intermediate subfactors between 
$M\subset M_1$, which form the {\it dual quadrilateral} 
$\hat{\fQ}=
\begin{array}{ccc}
\hat{Q} & \subset& M_1 \\
\cup & & \cup \\
N &\subset &\hat{P} 
\end{array}$. 
The quadrilateral $\fQ$ is said to be {\it cocommuting} if $\hat{\fQ}$ is 
commuting. 
The quadrilateral $\fQ$ is said to be {\it $(p,q)$-supertransitive} if $P\subset M$ 
and $Q\subset M$ are $p$-supertransitive and $N\subset P$ and $N\subset Q$ are 
$q$-supertransitive. 
Two quadrilaterals $\fQ=\quadri$ and $\tfQ=\tquadri$ 
are said to be {\it conjugate} if there exists an isomorphism $\pi$ from $M$ onto $\tM$ 
such that $\pi(P)=\tP$ and $\pi(Q)=\tQ$. 
The two quadrilateral $\fQ$ and $\tfQ$ are said to be {\it flip conjugate} if either 
they are conjugate or there exists an isomorphism $\pi$ from $M$ onto $\tM$ such that 
$\pi(P)=\tQ$ and $\pi(Q)=\tP$.
\end{definition}

We denote by $\widehat{P\subset M}$ the inclusion $M\subset \hat{P}$ and 
by $\widehat{N\subset P}$ the inclusion $\check{P}\subset N$, 
where $\check{P}\subset N\subset P$ is the downward basic construction. 
Note that $\check{P}$ is uniquely determined up to inner conjugacy in $N$. 

When $\fQ=\quadri$ is a non-commuting quadrilateral of subfactors such that 
$N\subset P$ and $N\subset Q$ are 2-supertransitive, the two projections 
$e_P$ and $e_Q$ are of rank 2 in $M_1\cap N'$ and $e_P\wedge e_Q=e_N$. 
Thus ${\rm Ang}(P,Q)$ consists of at most one point, which will be denoted by 
$\Theta(P,Q)$ if $\fQ$ is non-commuting. 

The following lemma is essentially proved in \cite[Proposition 3.2.11, Theorem 3.3.4]{GJ} 
using Landau's result \cite{La2}: 

\begin{lemma} \label{2lemma angle} 
Let $\fQ=\quadri$ be an irreducible cocommuting quadrilaterals of 
subfactors such that $N\subset P$ and $N\subset Q$ are 2-supertransitive. 
If $[M:P]=[P:N]$, the quadrilateral $\fQ$ is commuting, and if 
$[M:P]\neq [P:N]$,  
$$\cos^2\Theta(P,Q)=\frac{[P:N]-[M:P]}{[M:P]([P:N]-1)}.$$
\end{lemma}

We show an easy example of a noncommuting quadrilateral coming 
from a finite group action. 
For an automorphism group $G$ of a factor $R$, we denote by $R^G$ the fixed point 
subalgebra of $R$ under $G$. 
The next lemma follows from Lemma \ref{2lemma angle}. 

\begin{lemma} \label{2lemma group} Let $G$ be a finite group and $H$ and $K$ be subgroups of 
$G$  such that the natural actions of $G$ on $G/H$ and $G/K$ 
are 2-transitive. 
Assume that $[G:H]=[G:K]$, $[G:H]\neq [H:H\cap K]$, and 
$G$ acts on a factor $R$ as an outer automorphism group. 
We set $M=R^{K\cap H}$, $P=R^H$, $Q=R^K$, and $N=R^G$. 
Then $\fQ=\quadri$ is an irreducible noncommuting and cocommuting quadrilateral of 
factors such that $N\subset P$ and $N\subset Q$ are 2-supertransitive. 
The angle between $P$ and $Q$ is given by
$$\cos^2\Theta(P,Q)=\frac{[G:H]-[H:H\cap K]}{[H:H\cap K]([G:H]-1)}.$$
\end{lemma}

\begin{example}\label{2example group}
For a finite set $F$, we denote by $\fS_F$ the set of permutations of $F$ and by 
$\fA_F$ the set of even permutations of $F$. 
Let $F_0=\{1,2,\cdots,n\}$, $F_1=\{1,2,\cdots,n-1\}$, and $F_2=\{2,3,\cdots,n\}$. 
Then $G=\fS_{F_0}$ ($G=\fA_{F_0}$), $H=\fS_{F_1}$ ($H=\fA_{F_1}$), and $K=\fS_{F_2}$ 
($K=\fA_{F_2}$) with $n\geq 3$ ($n\geq 4$) satisfies the assumption 
of Lemma \ref{2lemma group} and $\Theta(P,Q)=\cos^{-1}(1/(n-1))$. 
\end{example}

In \cite{GJ} Jones and the first-named author obtained the following theorem:

\begin{theorem}\label{2theorem no-extra} Let $\fQ=\quadri$ be an irreducible noncommuting 
quadrilateral of factors. 
If $\fQ$ is (6,6)-supertransitive, one of the following two occurs: 
\begin{itemize} 
\item [$(1)$] All the elementary subfactors of $\fQ$ are the $A_7$ subfactors and 
$\Theta(P,Q)=\cos^{-1}(\sqrt{2}-1)$. 
When $M$ is the hyperfinite II$_1$ factor, such a quadrilateral exists and is unique up to 
conjugacy. 
\item [$(2)$] $[M:P]=[M:Q]=2$, $[P:N]=[Q:N]=3$, $\Theta(P,Q)=\pi/3$, and 
the principal graphs of $N\subset P$ and $N\subset Q$ are $A_5$. 
In this case, there exists an outer action of the symmetric group $\fS_3$ of degree 3 
on $M$ such that $N$ is the fixed point algebra of the action. 
When $M$ is the hyperfinite II$_1$ factor, such a quadrilateral is unique up to 
conjugacy. 
\end{itemize}
\end{theorem}

In Section 4, we give a new proof of the above theorem, except for the existence, 
relaxing the assumption. 
The theorem still holds if we assume that $\fQ$ is (3,4)-supertransitive.

\subsection{Sectors} 
To compute the angle between $P$ and $Q$, it is more 
convenient to work with properly infinite factors using sectors as we will see below. 
Thus we recall a few basic facts about sectors here. 
The reader is referred to \cite{I4} for details. 
Note that the structure of the intermediate subfactor lattice 
of an inclusion $N\subset M$, including information of the angle, only depends on 
the standard invariant (or paragroup, plannar algebra) of $N\subset M$.  
Every standard invariant realized in the type II$_1$ case is also realized in the properly 
infinite case and vice versa. 

For the reason stated above, we always assume in the proofs that the subfactors involved 
are properly infinite though we state results for general quadrilaterals of factors. 
When we show uniqueness results, we need to deal with quadrilaterals of hyperfinite II$_1$ 
factors. 
Since the classification theory of subfactors for the hyperfinite II$_1$ factor and 
the hyperfinite II$_\infty$ factor is the same as far as strongly amenable subfactors are 
concerned \cite{P1}, we may assume in the proofs that the factors involved are isomorphic to 
the hyperfinite II$_\infty$ factors in this case.


\begin{remark} When one discusses sectors, it is customary to assume that 
every factor involved is of type III. 
However the whole theory also works for general properly infinite factors. 
This is based on the following facts: (1) for an inclusion of 
properly infinite factors $N\subset M$, every non-zero projection 
in the relative commutant $N'\cap M$ is an infinite projection in $M$, and 
(2) a properly infinite factor has a unique representation whose commutant 
is also properly infinite. 
\end{remark}

Let $M$ and $N$ be properly infinite factors. 
We denote by $\Mor(N,M)$ the set of normal unital homomorphisms from 
$N$ to $M$. 
For two morphisms $\rho,\sigma\in \Mor(N,M)$, we denote by $(\rho,\sigma)$ 
the set of intertwiners between $\rho$ and $\sigma$ 
$$(\rho,\sigma)=\{v\in M;\; v\rho(x)=\sigma(x)v,\;\forall x\in N\}.$$  
When $\rho$ is irreducible, that is $M\cap \rho(N)'=\C$, 
the space $(\rho,\sigma)$ is a Hilbert space with an inner product 
$\inpr{v_1}{v_2}=v_2^*v_1$ for $v_1,v_2\in (\rho,\sigma)$. 
The two morphisms $\rho$ and $\sigma$ are said to be equivalent 
if there exists a unitary $u\in (\rho,\sigma)$. 
We denote by $\Sect(M,N)$ the quotient of $\Mor(N,M)$ by this 
equivalence relation.  
We denote by $[\rho]$ the equivalence class of $\rho$, which we call a sector. 
However, we often omit the quare bracket when no confusion arises. 
When $M=N$, we use the notation $\Mor(M,M)=\End(M)$ and 
$\Sect(M,M)=\Sect(M)$. 

For $M-N$ bimodule $X$ and $\rho\in \Mor(L,N)$, we 
denote by $X_\rho$ the $M-L$ bimodule defined by 
$$x\cdot \xi\cdot y=x\xi\rho(y),\quad \xi\in X,\;x\in M,\;y\in L.$$
For $\sigma\in \Mor(L,M)$, we denote by ${}_\sigma X$ the $L-N$ bimodule 
similarly defined. 

Let $L^2(M)$ be the standard Hilbert space of $M$, 
which is a $M-M$ bimodule with $x\xi y=xJ_M y^*J_M\xi$, 
where $J_M$ is the modular conjugation. 
Then it is known that for every $M-N$ bimodule $X$, there 
exists a unique sector $[\rho] \in \Sect(M,N)$ such that 
$X$ is equivalent to $L^2(M)_\rho$. 
With this correspondence, the relative tensor product of two bimodules 
is transformed into the composition of two morphisms. 
A direct sum of bimodules is easily translated into the sector language too (see \cite{I4}). 

We denote by $\Mor_0(N,M)$ the set of $\rho\in \Mor(N,M)$ whose 
image has finite index. 
For $\Mor_0(N,M)$, we denote by $d(\rho)$ the square root of $[M:\rho(N)]$, 
called the (statistical) dimension of $\rho$, which is multiplicative under 
composition of two sectors and is additive under direct sum. 
We use the notation $\Sect_0(M,N)$ etc. in an obvious sense. 

Corresponding to the complex conjugate bimodule of $L^2(M)_\rho$, 
the conjugate sector of $\rho$ is defined. 
For $\rho\in \Mor_0(N,M)$, the (equivalence class of) conjugate morphism 
$\bar{\rho}\in \Mor_0(M,N)$ is characterized by existence of two isometries 
$r_\rho\in (\id_N,\bar{\rho}\rho)$ and $\bar{r}_\rho\in (\id_M,\rho\bar{\rho})$ 
satisfying the following equation 
$$r_\rho^*\bar{\rho}(\bar{r}_\rho)=\frac{1}{d(\rho)},\quad 
\bar{r}_\rho^*\rho(r_\rho)=\frac{1}{d(\rho)}.$$
Throughout this paper, we fix such $r_\rho$ and $\bar{r}_\rho$ for each $\rho$. 
When $\rho$ is not self-conjugate, we can and do assume 
$r_{\bar{\rho}}=\bar{r}_\rho$. 
This is not the case for a self-conjugate $\rho$. 
In this case $r_\rho^*\rho(r_\rho)=\pm1/d(\rho)$ holds and so 
$r_{\rho}=\pm\bar{r}_\rho$ \cite{L1}. 
A self-conjugate endomorphism $\rho$ is said to be real if $r_{\bar{\rho}}=\bar{r}_\rho$ 
and said to be pseudo-real if $r_{\bar{\rho}}=-\bar{r}_\rho$. 

We set $\phi_\rho(x)=r_\rho^*\bar{\rho}(x)r_\rho$ for $x\in N$, which is called 
the standard left inverse of $\rho$. 
We denote by $E_\rho$ the minimum conditional expectation from $M$ onto the image of 
$\rho$. 
Then we have $E_\rho=\rho\phi_\rho$. 

For an inclusion of properly infinite factors $N\subset M$ of finite index, 
we denote by $\iota_{M,N}$ the inclusion map $\iota_{M,N}:N\hookrightarrow M$. 
Then the $M-N$ bimodule $X={}_ML^2(M)_N$ is nothing but $L^2(M)\iota_{M,N}$. 
Since ${}_ML^2(M_1)_M$ is isomorphic to $X\otimes_NX^*$, 
we have ${}_ML^2(M_1)_M\cong L^2(M)_{\iota_{M,N}\overline{\iota_{M,N}}}$. 
The endomorphism $\iota_{M,N}\overline{\iota_{M,N}}$ of $M$ is called the canonical 
endomorphism for $N\subset M$. 
The endomorphism $\overline{\iota_{M,N}}\iota_{M,N}$ of $N$ is called the dual 
canonical endomorphism, which is the canonical endomorphism for the 
inclusion $\overline{\iota_{M,N}}(M)\subset N$. 

We often use diagrams for computation of intertwiners. 
An intertwiner $s\in (\rho,\sigma\tau)$ is expressed by the diagram
$$s=\vpic{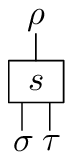}{0.3in}$$
For $t\in (\sigma,\mu\nu)$ and $u\in (\tau,\xi\eta)$, the products 
$ts\in (\rho,\mu\nu\tau)$ and $\sigma(u)s\in (\rho,\sigma\xi\eta)$ are 
expressed as
$$ts=\vpic{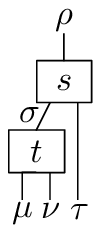}{0.5in},\quad \sigma(u)s=\vpic{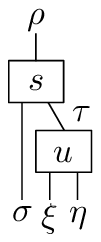}{0.5in}.$$
For special operators $1\in (\rho,\rho)$, 
$\sqrt{d(\rho)}r_\rho\in (\id,\bar{\rho}\rho)$, and 
$\sqrt{d(\rho)}\br_\rho^*\in (\id,\rho\bar{\rho})$, 
we use the following diagrammatic expressions: 
$$1=\vpic{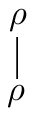}{0.1in},\quad \sqrt{d(\rho)}r_\rho=\vpic{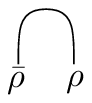}{0.4in},\quad 
\sqrt{d(\rho)}\br_\rho^*=\vpic{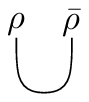}{0.4in}.$$
Then the equations $d(\rho)\br_\rho^*\rho(r_\rho)=1$ and 
$d(\rho)\bar{\rho}(\br_\rho^*)r_\rho=1$ are expressed as
$$\vpic{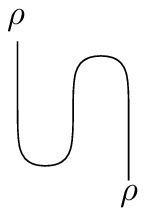}{0.6in}=\vpic{II4}{0.1in},\quad \vpic{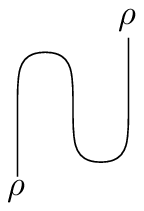}{0.6in}=\vpic{II4}{0.1in}$$

The following lemma holds (see \cite[Section 3]{ILP}):

\begin{theorem} \label{2theorem decomposition}
Let $N\subset M$ be an irreducible inclusion of properly infinite factors with finite index 
and let 
$$[\overline{\iota_{M,N}}\iota_{M,N}]=\bigoplus_{i=0}^m n_i[\rho_i]$$
be the irreducible decomposition, where $n_i$ is the multiplicity of 
$\rho_i$ in $\overline{\iota_{M,N}}\iota_{M,N}$. 
Then
\begin{itemize}
\item [$(1)$] $n_i\leq d(\rho_i)$. 
\item [$(2)$] Let $\cH_i=(\iota_{M,N},\iota_{M,N}\rho_i)$. 
For $s_1,s_2\in \cH_i$, 
$$E_N(s_1s_2^*)=\frac{\inpr{s_1}{s_2}}{d(\rho_i)}.$$
\item [$(3)$] $\dim \cH_i=n_i$. 
\item [$(4)$] Let $\{s(i)_j\}_{j=1}^{n_i}$ be an orthonormal basis of 
$\cH_i$. 
Then $x\in M$ has the following unique expansion 
$$x=\sum_{i=0}^m\sum_{j=1}^{n_i}s(i)_j^*x(i)_j$$
with coefficients $x(i)_j\in N$. 
The coefficients $x(i)_j$ is given by 
$$x(i)_j=d(\rho_i)E_N(s(i)_jx).$$ 
\end{itemize}
\end{theorem}

The fourth statement above says that $\{\sqrt{d(\rho_i)}s(i)_j^*\}_{ij}$ 
is a Pimsner-Popa basis \cite{PP}. 
The advantage of working on properly infinite factors is that we can choose each 
element of the basis from an intertwiner space, and so 
(4) is considered as a crossed product type decomposition. 

Let $A_i$ be the linear span of $\{s^*e_Nt;\;s,t\in \cH_i\}$. 
Then \cite[Section 3]{ILP} shows that $A_i$ is a simple component of 
$M\cap N'$ and 
$$M_1\cap N'=\bigoplus_{i=0}^mA_i.$$
We can introduce a right action of $A_i$ on $\cH_i$ by 
$$v\cdot (s^*e_Nt)=\frac{1}{d(\rho_i)}\inpr{v}{s}t,$$
which gives a $*$-isomorphism from the opposite algebra of $A_i$ 
onto $\B(\cH_i)$. 
Let $e(i)_{jk}=d(\rho_i)s(i)_j^*e_Ns(i)_k$. 
Then $\{e(i)_{jk}\}_{jk}$ is a system of matrix unit of $A_i$. 
Let $P$ be an intermediate subfactor and let 
$m_i=\dim (\iota_{P,N},\iota_{P,N}\rho_i)$. 
Then we may and do assume that $\{s(i)_j\}_{j=1}^{m_i}$ is an orthonormal basis 
for $\cH_i\cap P=(\iota_{P,N},\iota_{P,N}\rho_i)$.

\begin{lemma} \label{2lemma projection} Let the notation be as above. 
Then 
\begin{itemize}
\item [$(1)$] The restriction $E_P|_{\cH_i}$ of $E_P$ to $\cH_i$ is the projection 
onto $\cH_i\cap P=(\iota_{P,N},\iota_{P,N}\rho_i)$. 
\item [$(2)$] The Jones projection $e_P$ is expressed as 
$$e_P=\sum_{i=0}^m\sum_{j=1}^{m_i}e(i)_{jj}.$$
\item [$(3)$] Let $z_i$ be the unit of $A_i$. 
Then 
$$s\cdot z_ie_P=E_P(s),\quad \forall s\in \cH_i.$$
\end{itemize}
\end{lemma}

\begin{proof} (1) 
Note that $E_P(\cH_i)=\cH_i\cap P=(\iota_{P,N},\iota_{P,N}\rho_i)$ holds and 
$E_P|_{\cH_i}$ is an idempotent. 
Thus it suffices to show that $E_P|_{\cH_i}$ is self-adjoint. 
For $s,t\in \cH_i$, we have 
\begin{eqnarray*}
\inpr{E_P(s)}{t}&=&d(\rho_i)E_N(E_P(s)t^*)=d(\rho_i)E_N(E_P(E_P(s)t^*))\\
&=&d(\rho_i)E_N(E_P(s)E_P(t)^*))=\inpr{E_P(s)}{E_P(t)},
\end{eqnarray*}
which shows the statement. 

(2) Let $\varphi$ be a faithful normal state of $M$ satisfying 
$\varphi\cdot E_N=\varphi$.  
Then (1) and Theorem \ref{2theorem decomposition},(4) imply the following for 
$x\in M$ and the GNS cyclic and separating vector $\Omega$ for $\varphi$: 
\begin{eqnarray*}
e_Px\Omega&=&E_P(\sum_{i=0}^m\sum_{j=1}^{n_i}t(i)_j^*x(i)_j)\Omega
=\sum_{i=0}^m\sum_{j=1}^{n_i}E_P(t(i)_j)^*x(i)_j\Omega\\
&=&\sum_{i=0}^m\sum_{j=1}^{m_i}t(i)_j^*x(i)_j\Omega
=\sum_{i=0}^m d(\rho_i)\sum_{j=1}^{m_i}t(i)_j^*E_N(t(i)_jx)\Omega\\
&=&\sum_{i=0}^m d(\rho_i)\sum_{j=1}^{m_i}t(i)_j^*e_N t(i)_jx\Omega.
\end{eqnarray*}
This shows (2). 
(3) follows from (2). 
\end{proof}

\begin{remark}\label{2remark angle} 
The above lemma gives a powerful method to compute the angle between subfactors. 
Let $Q$ be another intermediate subfactor of $N\subset M$. 
To compute the angle between $P$ and $Q$, it suffices to compute 
the angle between two subspaces $(\iota_{P,N},\iota_{P,N}\rho_i)$ 
and $(\iota_{Q,N},\iota_{Q,N}\rho_i)$ of $(\iota_{M,N},\iota_{M,N}\rho_i)$ for those 
$\rho_i$ contained in 
$\overline{\iota_{P,N}}\iota_{P,N}$ and $\overline{\iota_{Q,N}}\iota_{Q,N}$. 
Equivalently, it suffices to compute the spectrum of the restriction of 
$E_P E_Q$ to $(\iota_{P,N},\iota_{P,N}\rho_i)$. 
\end{remark}

\begin{remark} \label{2remark angle2} 
Even when the inclusion $N\subset M$ is reducible, the angle $\Ang(P,Q)$ can be computed 
from $E_PE_Q|_{(\iota_{P,N},\iota_{P,N}\rho_i)}$ as long as 
$N\subset P$ and $N\subset Q$ are irreducible. 
Indeed we have $M=P\oplus\ker E_P$ and 
$$P=\bigoplus_{i=0}^m N(\iota_{P,N},\iota_{P,N}\rho_i).$$
Since $E_PE_QE_P$ preserves $(\iota_{P,N},\iota_{P,N}\rho_i)$, it suffice to compute 
$E_PE_Q|_{(\iota_{P,N},\iota_{P,N}\rho_i)}$ to obtain the eigenvalues 
of $E_PE_QE_P$. 
\end{remark}

\section{$Q$-systems and angles}
When $\quadri$ is an irreducible noncommuting quadrilateral of subfactors and 
$N\subset P$ and $N\subset Q$ are 2-supertransitive, 
it is observed in \cite[Lemma 4.2.1]{GJ} that $L^2(P)$ and $L^2(Q)$ are equivalent as 
$N-N$ bimodules, or in other words, $\overline{\iota_{P,N}}\iota_{P,N}$ is equivalent 
to $\overline{\iota_{Q,N}}\iota_{Q,N}$. 
Consequently, $[P:N]=[Q:N]$ and so $[M:P]=[M:Q]$. 
This follows from the fact that the two projections $e_P$ and $e_Q$ are 
equivalent in $M_1\cap N'$, which is identified with $\End_NL^2(M)_N$. 
Let $\check{P}$ and $\check{Q}$ be subfactors of $N$ such that 
$\check{P}\subset N\subset P$ and $\check{Q}\subset N\subset P$ are towers. 
It is known that under a certain situation, such as 3-supertransitivity 
of $N\subset P$ and $N\subset Q$, the equivalence 
${}_NL^2(P)_N\cong {}_NL^2(Q)_N$ implies inner conjugacy of $\check{P}$ and $\check{Q}$ 
in $N$ \cite{IKo2}, \cite{KL} though it is not necessarily the case in general. 

To characterize the canonical endomorphism, Longo \cite{L2} introduced 
the notion of $Q$-systems. 

\begin{definition} \label{3definition Q-system} Let $\cM$ be a properly infinite factor and 
let $\gamma\in \End_0(\cM)$. 
We say that $(\gamma,v,w)$ is a $Q$-system if the following three conditions are 
satisfied: 
\begin{itemize}
\item [$(1)$] $v\in (\id_\cM,\gamma)$ and $w\in (\gamma,\gamma^2)$ are isometries. 
\item [$(2)$] There exists a positive number $d$ such that 
$v^*w=w^*\gamma(v)=1/d$. 
\item [$(3)$] $ww=\gamma(w)w$.
\end{itemize}
Two $Q$-systems $(\gamma,v,w)$ and $(\gamma',v',w')$ 
are said to be equivalent if there exists a unitary $u\in (\gamma,\gamma')$ 
such that $v'=uv$ and $w'=u\gamma(u)wu^*$. 
\end{definition}

It is known that (3) is equivalent to $ww^*=\gamma(w^*)w$ under the  
assumption (1) and (2) \cite{LRo}, \cite{IKo2}. 
If $\gamma=\sigma\bar{\sigma}$ with $\sigma\in \Mor_0(\cN,\cM)$ and  
$v=\bar{r}_\sigma\in (\id_\cM,\gamma)$, $w=\sigma(r_\sigma)\in 
\sigma((\id_\cN,\bar{\sigma}\sigma))\subset (\gamma,\gamma^2)$, 
then $(\gamma,v,w)$ is a $Q$-system with $d=d(\sigma)$. 
The equivalence class of this $Q$-system only depends on the inner conjugacy class of 
$\sigma(\cN)$ in $\cM$. 
Thus we say that {\it $(\gamma,v,w)$ arises from $\sigma(\cN)\subset \cM$}. 
On the other hand, Longo \cite{L2} showed that any $Q$-system $(\gamma,v,w)$ 
arises from a subfactor $\cN\subset \cM$, determined by $E_{\cN}(x)=w^*\gamma(x)w$, 
and equivalence of two $Q$-systems exactly corresponds to inner conjugacy 
of the corresponding two subfactors. 

Let $\cN\subset \cM$ be an inclusion of properly infinite factors of finite 
index and let $(\gamma,v,w)$ be the $Q$-system arising from the inclusion. 
The second-named author and Kosaki \cite{IKo2} introduced the 
``second cohomology" $H^2(\cN\subset \cM)$, which is always a finite set, 
by the equivalence classes of $Q$-systems $(\gamma_1,v_1,w_2)$ such that $\gamma$ and $\gamma_1$ 
are equivalent. 
When $\cN$ is the fixed point algebra of $\cM$ by an outer action of 
a finite group $G$, then we can identify $H^2(\cN\subset \cM)$ with 
the second cohomology group $H^2(G,\T)$.

The following is \cite[Lemma 6.1]{IKo2}. 

\begin{lemma} \label{3lemma vanishing} 
Let $\cN\subset \cM$ be an irreducible inclusion of factors 
of finite index. 
If $\cN\subset \cM$ is 3-supertransitive, then $H^2(\cN\subset \cM)$ 
is trivial. 
\end{lemma}

When $\fQ=\quadri$ is a noncommuting quadrilateral of subfactors and 
$N\subset P$ and $N\subset Q$ are 2-supertransitive, 
the two inclusions $\check{P}\subset N$ and $\check{Q}\subset N$ give equivalent 
canonical endomorphisms as we observed above. 

\begin{definition} Let $\fQ=\quadri$ be an irreducible noncommuting quadrilateral  
of factors with 2-supertransitive $N\subset P$ and $N\subset Q$. 
We denote by $c(\fQ)$ the class of the $Q$-system arising from $\check{Q}\subset N$ 
in $H^2(\widehat{N\subset P})$. 
\end{definition}

The class $c(\fQ)$ is very much related to the angle $\Theta(P,Q)$ 
and there are at most two possibilities for $\Theta(P,Q)$ for 
a given class $c(\fQ)$ as we will see below. 

The following lemma is an easy consequence of the definition above: 

\begin{lemma} \label{3lemma trivial} 
For an irreducible noncommuting quadrilateral $\fQ=\quadri$ 
of factors with 2-supertransitive $N\subset P$ and $N\subset Q$, the following 
three conditions are equivalent: 
\begin{itemize}
\item [$(1)$] The class $c(\fQ)\in H^2(\widehat{N\subset P})$ is trivial. 
\item [$(2)$] The two subfactors $\check{P}$ and $\check{Q}$ are inner conjugate 
in $N$. 
\item [$(3)$] There exists an isomorphism $\pi:P\rightarrow Q$ of $P$ onto $Q$ 
such that the restriction $\pi|_N$ of $\pi$ to $N$ is the identity map. 
\end{itemize}
\end{lemma}

Since the $Q$-systems we encounter in this paper always have $\gamma$ decomposed into 
two irreducible components, we give a detailed description of the condition 
(3) of Definition \ref{3definition Q-system} in this case. 
Let $\sigma\in \End_0(\cM)$ be an irreducible endomorphism satisfying 
$[\id_\cM]\neq [\sigma]$ and let $\gamma=\id_{\cM}\oplus \sigma$. 
Note that for $\gamma$ to have a $Q$-system, the endomorphism $\sigma$ must be self-conjugate.
We fix isometries $v\in (\id_{\cM},\gamma)$ and $v_1\in (\sigma,\gamma)$, which means 
$\gamma(x)=vxv^*+v_1\sigma(x)v_1^*$. 
We give a description of a $Q$-system of the form $(\gamma,v,w)$ in terms of an isometry 
in $(\sigma,\sigma^2)$.  
It is easy to show that an isometry $w\in (\gamma,\gamma^2)$ satisfying (2) of 
Definition \ref{3definition Q-system} is of the following form and vice versa: 
$$w=\frac{v}{\sqrt{d(\sigma)+1}}+\frac{v_1\sigma(v)v_1^*}{\sqrt{d(\sigma)+1}}
+\sqrt{\frac{d(\sigma)}{d(\sigma)+1}}cv_1\sigma(v_1)r_\sigma v^*
+\sqrt{\frac{d(\sigma)-1}{d(\sigma)+1}}v_1\sigma(v_1)s v_1^*,$$
where $c$ is a complex number with $|c|=1$ and $s\in (\sigma,\sigma^2)$ 
is an isometry.  
Passing from $w$ to $u\gamma(u)wu^*$ with $u=vv^*+\bar{c}v_1v_1^*$ (or redefining $r_\sigma$) 
if necessary, we may always assume $c=1$ for a $Q$-system $(\gamma,v,w)$. 
We say that such a $Q$-system is normalized (with respect to $(v,v_1,r_\sigma)$). 
It is straightforward to show the following: 

\begin{lemma} \label{3lemma Q-system} 
Let $\sigma$, $\gamma$, $v$, and $v_1$ be as above. 
If $s\in (\sigma,\sigma^2)$ is an isometry satisfying 
\begin{equation}\sigma(s)r_\sigma=sr_\sigma,
\end{equation}
\begin{equation}\sqrt{d(\sigma)}r_\sigma+(d(\sigma)-1)s^2=\sqrt{d(\sigma)}\sigma(r_\sigma)
+(d(\sigma)-1)\sigma(s)s,
\end{equation}
and $w$ is given by 
\begin{equation}
w=\frac{v}{\sqrt{d(\sigma)+1}}+\frac{v_1\sigma(v)v_1^*}{\sqrt{d(\sigma)+1}}
+\sqrt{\frac{d(\sigma)}{d(\sigma)+1}}v_1\sigma(v_1)r_\sigma v^*
+\sqrt{\frac{d(\sigma)-1}{d(\sigma)+1}}v_1\sigma(v_1)s v_1^*,
\end{equation}
then $(\gamma,v,w)$ is a $Q$-system. 
Conversely, every normalized $Q$-system is of this form with an isometry 
$s\in (\sigma,\sigma^2)$ satisfying $(3.1)$ and $(3.2)$. 
If $s$ and $s'$ are isometries in $(\sigma,\sigma^2)$ satisfyingly $(3.1)$ and $(3.2)$, 
the corresponding $Q$-systems are equivalent if and only if $s=\pm s'$. 
\end{lemma}

\begin{remark} When the $Q$-system $(\gamma,v,w)$ as above arises from a 3-supertransitive 
subfactor, we have $\dim(\sigma,\sigma^2)=1$. 
If $s,s'\in (\sigma,\sigma^2)$ are isometries satisfying (3.1) and (3.2), 
there exists a complex number $\omega$ of modulus 1 such that $s'=\omega s$ and 
(3.2) implies $\omega^2=1$. 
This proves Lemma \ref{3lemma vanishing} again. 
\end{remark}

We keep using the same notation as above and consider the case with an inclusion 
$\cM\subset \cM_1$ such that $\gamma=\biota\iota$, $v=r_\iota$, $w=\biota(\br_\iota)$ 
where $\iota=\iota_{\cM_1,\cM}$. 
We set 
$$t=\frac{d(\iota)}{\sqrt{d(\sigma)}}\iota(v_1^*)\br_\iota,$$
which is an isometry in $(\iota,\iota\sigma)$ thanks to Frobenius reciprocity 
\cite[Proposition 2.2]{I4}. 
Then Theorem \ref{2theorem decomposition},(4) shows that for every $x\in \cM_1$, we have 
$$x=E_{\cM}(x)+d(\sigma)E_{\cM}(xt^*)t.$$
In particular, 
$$t^2=E_{\cM}(t^2)+d(\sigma)E_\cM(t^2t^*)t,$$
$$t^*=E_\cM(t^*)+d(\sigma)E_{\cM}({t^*}^2)t=d(\sigma)E_{\cM}({t^*}^2)t,$$
where we use the fact $E_\cM(t)\in (\id_\cM,\sigma)=\{0\}$. 
Therefore the $*$-algebra structure of $\cM_1$ is completely determined by 
$E_\cM(t^2)$ and $E_\cM(t^2t^*)$, 
which can be computed in terms of the $Q$-system as follows:
\begin{eqnarray*}
E_\cM(t^2) &=&\frac{d(\iota)^2}{d(\sigma)}
r_\iota^*\biota(\iota(v_1^*)\br_\iota\iota(v_1^*)\br_\iota)r_\iota
=\frac{d(\iota)}{d(\sigma)}v_1^*\biota(\iota(v_1^*)\br_\iota)r_\iota\\
&=&\frac{d(\iota)}{d(\sigma)}\sigma(v_1^*)v_1^*wv
=\frac{1}{\sqrt{d(\sigma)}}r_\sigma, 
\end{eqnarray*}
\begin{eqnarray*}
E_\cM(t^2t^*)&=&\frac{d(\iota)^3}{\sqrt{d(\sigma)}^3}
r_\iota^*\biota(\iota(v_1^*)\br_\iota\iota(v_1^*)\br_\iota \br_\iota^*\iota(v_1))r_\iota 
=\frac{d(\iota)}{\sqrt{d(\sigma)}^3}
v_1^*\biota(\iota(v_1^*)\br_\iota) v_1\\
&=&\frac{d(\iota)}{\sqrt{d(\sigma)}^3}
\sigma(v_1^*)v_1^*wv_1=\sqrt{\frac{d(\sigma)-1}{d(\sigma)^3}}s.
\end{eqnarray*}

\begin{lemma}\label{3lemma algebra} Let the notation be as above. 
Then the following hold: 
$$t^2=\frac{1}{\sqrt{d(\sigma)}}r_\sigma+\sqrt{\frac{d(\sigma)-1}{d(\sigma)}}st,$$
$$t^*=\sqrt{d(\sigma)}r_\sigma^*t.$$
\end{lemma}

\begin{corollary} Let $\cM$ be a properly infinite factor and let $\sigma\in \End_0(\cM)$ 
be an irreducible self-conjugate endomorphism. 
If there exists a $Q$-system for $\id_\cM\oplus \sigma$, then $[\sigma]$ is a real sector, 
that is,  $r_\sigma=\br_\sigma$. 
\end{corollary}

\begin{proof} Using one of the above formulae, we get 
$$t=(t^*)^*=\sqrt{d(\sigma)}(r_\sigma^*t)^*=\sqrt{d(\sigma)}t^*r_\sigma=d(\sigma)r_\sigma^*tr_\sigma
=d(\sigma)r_\sigma^*\sigma(r_\sigma) t,$$
which shows $r_\sigma^*\sigma(r_\sigma)=1/d(\sigma)$. 
\end{proof}

Now we show how the class $c(\fQ)$ determines the angle between $P$ and $Q$. 
Let $\fQ=\quadri$ be an irreducible quadrilateral of factors such that 
$N\subset P$ and $N\subset Q$ are 2-supertransitive. 
We assume that ${}_NL^2(P)_N$ and ${}_NL^2(Q)_N$ are equivalent. 
We apply the above argument to the case where $\cM_1=P$, $\cM=N$, $\iota=\iota_{P,N}$ and 
$\gamma=\biota\iota$. 
We set $\check{P}=\biota(P)$. 
We choose two isometries $v=r_\iota\in (\id,\gamma)$ and $v_1\in (\sigma,\gamma)$. 
Then we may assume that the $Q$-system $(\gamma,v,w_P=\biota(\br_\iota))$ 
is normalized with respect to $(v,v_1,r_\sigma)$. 
By choosing an appropriate representative of the conjugate sector of $[\iota_{Q,N}]$, 
we may also assume $\gamma=\overline{\iota_{Q,N}}\iota_{Q,N}$, $v=r_{\iota_{Q,N}}$, 
and that the $Q$-system $(\gamma,v,w_Q=\overline{\iota_{Q,N}}(\br_{\iota_{Q,N}}))$ 
is normalized with respect to $(v,v_1,r_\sigma)$. 
By definition, the class $c(\fQ)$ is given by $(\gamma,v,w_Q)$ in $H^2(\check{P}\subset N)$. 
Let $s_P$ and $s_Q$ be isometries in $(\sigma,\sigma^2)$ corresponding to 
$w_P$ and $w_Q$ through (3.3) respectively. 
We set 
$$t_P=\sqrt{\frac{d(\sigma)+1}{d(\sigma)}}\iota_{P,N}(v_1^*)\br_{\iota_{P,N}}
\in (\iota_{P,N},\iota_{P,N}\sigma),$$ 
$$t_Q=\sqrt{\frac{d(\sigma)+1}{d(\sigma)}}\iota_{Q,N}(v_1^*)\br_{\iota_{Q,N}}
\in (\iota_{Q,N},\iota_{Q,N}\sigma),$$
which are isometries. 

\begin{lemma}\label{3lemma quadratic} Let $\fQ=\quadri$ be an irreducible 
quadrilateral of factors such that $N\subset P$ and $N\subset Q$ are 2-supertransitive. 
If $L^2(P)$ and $L^2(Q)$ are equivalent as $N-N$ bimodules, then 
\begin{itemize}
\item [$(1)$] $\inpr{t_P}{t_Q}$ satisfies the following quadratic equation:
$$\inpr{t_P}{t_Q}^2-\frac{d(\sigma)-1}{d(\sigma)}\inpr{s_P}{s_Q}\inpr{t_P}{t_Q}
-\frac{1}{d(\sigma)}=0.$$
\item [$(2)$] $\inpr{s_P}{s_Q}$ and $\inpr{t_P}{t_Q}$ are real numbers. 
\end{itemize}
\end{lemma}

\begin{proof} (1) Lemma \ref{3lemma algebra} implies 
$$\inpr{t_P}{t_Q}^2={t_Q^2}^*t_P^2=\frac{1}{d(\sigma)}
+\frac{d(\sigma)-1}{d(\sigma)}\inpr{s_P}{s_Q}\inpr{t_P}{t_Q}.$$ 

(2) Thanks to Lemma \ref{2lemma projection}, we have 
$E_P(t_Q)=\inpr{t_Q}{t_P}t_P$ and $E_Q(t_P)=\inpr{t_P}{t_Q}t_Q$. 
Thus 
$$E_N(t_Pt_Q)=E_N(E_P(t_Pt_Q))=\inpr{t_Q}{t_P}E_N(t_P^2)
=\frac{\inpr{t_Q}{t_P}}{\sqrt{d(\sigma)}}r_\sigma,$$
$$E_N(t_Pt_Q)=E_N(E_Q(t_Pt_Q))=\inpr{t_P}{t_Q}E_N(t_Q^2)
=\frac{\inpr{t_P}{t_Q}}{\sqrt{d(\sigma)}}r_\sigma,$$
which shows that $\inpr{t_P}{t_Q}$ is real. 
This and (1) show that $\inpr{s_P}{s_Q}$ is also real. 
\end{proof}

Note that $\dim(\iota_{P,N},\iota_{P,N}\sigma)=\dim(\iota_{Q,N},\iota_{Q,N}\sigma)=1$ 
by Frobenius reciprocity. 
Thus Remark \ref{2remark angle} shows $\cos\Theta(P,Q)=|\inpr{t_P}{t_Q}|$. 

\begin{theorem} \label{3theorem angle} 
Let $\fQ=\quadri$ be an irreducible quadrilateral of factors such that 
$N\subset P$ and $N\subset Q$ are 2-supertransitive. 
If $L^2(P)$ and $L^2(Q)$ are equivalent as $N-N$ bimodules, then
\begin{itemize}
\item [$(1)$] $\fQ$ is not commuting. 
\item [$(2)$] The angle $\Theta(P,Q)$ is given by 
\begin{eqnarray*}
\cos\Theta(P,Q)&=&\frac{\sqrt{(d(\sigma)-1)^2\inpr{s_P}{s_Q}^2+4d(\sigma)}
+(d(\sigma)-1)|\inpr{s_P}{s_Q}|}{2d(\sigma)}\\
{\rm or}\;&=&\frac{\sqrt{(d(\sigma)-1)^2\inpr{s_P}{s_Q}^2+4d(\sigma)}
-(d(\sigma)-1)|\inpr{s_P}{s_Q}|}{2d(\sigma)}.
\end{eqnarray*}
\end{itemize}
\end{theorem}

\begin{proof} (1) If $\fQ$ were commuting, we would have $\inpr{t_P}{t_Q}=0$, which 
never occurs due to Lemma \ref{3lemma quadratic},(1). 
(2) follows from Lemma \ref{3lemma quadratic} too. 
\end{proof}

Note that $|\inpr{s_P}{s_Q}|$ is a numerical invariant of the class $c(\fQ)$. 
Therefore the above theorem says that there are only two possibilities of 
the angle $\Theta(P,Q)$ given $c(\fQ)$. 

\begin{corollary} \label{3corollary angle} 
Let $\fQ=\quadri$ be an irreducible noncommuting quadrilateral of factors 
such that $N\subset P$ and $N\subset Q$ are 2-supertransitive. 
Then
$$0<\Theta(P,Q)\leq \cos^{-1}\frac{1}{[P:N]-1}.$$
The equality holds if and only if the class $c(\fQ)\in H^2(\check{P}\subset N)$ 
is trivial. 
In particular, if $N\subset P$ is 3-supertransitive, 
$$\Theta(P,Q)= \cos^{-1}\frac{1}{[P:N]-1}.$$
\end{corollary}

\begin{proof} Note that $|\inpr{t_P}{t_Q}|=1$ never occurs since $|\inpr{t_P}{t_Q}|=1$ 
would imply $P=Q$. 
The statement follows from 
$$\frac{1}{d(\sigma)}\leq \frac{\sqrt{(d(\sigma)-1)^2\inpr{s_P}{s_Q}^2+4d(\sigma)}
-(d(\sigma)-1)|\inpr{s_P}{s_Q}|}{2d(\sigma)},$$
where the quality holds if and only if $\inpr{s_P}{s_Q}=\pm 1$. 
\end{proof}

\begin{corollary} \label{3corollary index} 
Let $\fQ=\quadri$ be an irreducible noncommuting quadrilateral of factors 
such that $N\subset P$ and $N\subset Q$ are 2-supertransitive. 
If $\fQ$ is cocommuting, then 
$$[M:P]\leq [P:N]-1.$$
The equality holds if and only if the class $c(\fQ)\in H^2(\check{P}\subset N)$ 
is trivial. 
In particular, if $\fQ$ is cocommuting and $N\subset P$ is 3-supertransitive, 
$$[M:P]= [P:N]-1.$$
\end{corollary}

\begin{proof} This follows from Lemma \ref{2lemma angle} and Lemma \ref{3corollary angle}. 
\end{proof}

\begin{remark}
There is an example of a quadrilateral $\fQ$ fulfilling the assumption of 
Corollary \ref{3corollary index} with non-trivial $c(\fQ)$. 
Let $q$ be a prime power and let $F$ be a finite field with $q$ elements. 
We set $G=PSL(n,q)$ with $n\geq 3$, which naturally acts on the projective space $PF^{n-1}$. 
It is known that this action is 2-transitive. 
We denote by ${F^n}^*$ the dual space of $F^n$, on which $G$ naturally acts too. 
Let $\{e_i\}_{i=1}^n$ be the canonical basis of $F^n$ and 
let $\{e_i^*\}_{i=1}^n$ be the dual basis of $\{e_i\}_{i=1}^n$. 
Let $H$ be the stabilizer of $F e_1$ and let $K$ be the stabilizer of 
$F e_1^*$. 
Then we have $[G:H]=[G:K]=(q^n-1)/(q-1)$ and $[H:H\cap K]=[K:H\cap K]=q^{n-1}$. 
Let $\fQ$ be the quadrilateral constructed by Lemma \ref{2lemma group} with these 
$G$, $H$, and $K$. 
Then  $[M:P]$ is strictly smaller than $[P:N]-1$. 
In this case we have $\cos\Theta(P,Q)=q^{-n/2}$. 

\end{remark}

\begin{lemma}
\label{3lemma uniqueness} 
Let $\fQ=\quadri$ be an irreducible noncommuting quadrilateral of factors such that 
$N\subset P$ and $N\subset Q$ are 2-supertransitive and $H^2(\widehat{N\subset P})$ 
is trivial. 
Let $[\overline{\iota_{P,N}}\iota_{P,N}]=[\id_N]\oplus [\sigma]$. 
If $\dim(\iota_{M,N},\iota_{M,N}\sigma)=2$ and there is an intermediate subfactor 
$N\subset R\subset M$ other than $P$ and $Q$ such that $L^2(R)$ is equivalent to $L^2(P)$ 
as $N-N$ bimodules, then $[P:N]=3$.  
\end{lemma}

\begin{proof} Let $t_P$ and $t_Q$ be as in Lemma \ref{3lemma quadratic}. 
We may and do assume $\inpr{t_P}{t_Q}=1/d(\sigma)$ by replacing $t_Q$ with $-t_Q$ 
if necessary. 
Assume that there exists an intermediate subfactor $N\subset R\subset M$ other than 
$P$ and $Q$ such that $L^2(R)$ is equivalent to $L^2(P)$ as $N-N$ bimodules. 
Then we may choose an isometry $t\in (\iota_{R,N},\iota_{R,N}\sigma)$ 
such that $\inpr{t}{t_P}=1/d(\sigma)$, 
$\inpr{t}{t_Q}=\epsilon/d(\sigma)$ with $\epsilon\in \{1,-1\}$, 
and $E_N(t^2)=1/\sqrt{d(\sigma)}r_\sigma$. 
Since $\dim(\iota_{M,N},\iota_{M,N}\sigma)=2$, the intertwiner space 
$(\iota_{M,N},\iota_{M,N}\sigma)$ is spanned by $t_P$ and $t_Q$. 
Therefore, the first two equalities imply
$$t=\frac{d(\sigma)-\epsilon}{d(\sigma)^2-1}(t_P+\epsilon t_Q).$$
Since $t$ is an isometry, we get 
\begin{eqnarray*}
1&=&\inpr{t}{t}= \frac{(d(\sigma)-\epsilon)^2}{(d(\sigma)^2-1)^2}
\inpr{t_P+\epsilon t_Q}{t_P+\epsilon t_Q}
=\frac{(d(\sigma)-\epsilon)^2}{(d(\sigma)^2-1)^2}\frac{2(d(\sigma)+\epsilon)}{d(\sigma)}\\
&=&\frac{2(d(\sigma)-\epsilon)}{d(\sigma)(d(\sigma)^2-1)}
\end{eqnarray*}
If $\epsilon=1$, we would have $d(\sigma)=1$, which contradicts 
$\dim(\iota_{M,N},\iota_{M,N}\sigma)=2$ and Theorem \ref{2theorem decomposition},(1). 
Therefore we get $\epsilon=-1$ and $d(\sigma)=2$, which means $[P:N]=3$. 
\end{proof}

\begin{remark} \label{3remark no-auto} Let $\fQ=\quadri$ be an 
irreducible noncommuting quadrilateral of factors. 
Then there exists $\sigma$ contained both in $\overline{\iota_{P,N}}\iota_{P,N}$ 
and in $\overline{\iota_{Q,N}}\iota_{Q,N}$ so that the multiplicity of 
$\sigma$ in $\overline{\iota_{M,N}}\iota_{M,N}$ is at least 2. 
Therefore $\sigma$ is not an automorphism due to Theorem \ref{2theorem decomposition},(1). 
In particular, the subfactor $N\subset P$ is not the crossed product by a finite group action.
\end{remark}

\section{(2,2)-supertransitive quadrilaterals}
In what follows, for a non-commuting quadrilateral $\quadri$ of properly infinite 
factors with 2-supertransitive elementary subfactors, 
we use the following notation: the symbols $\iota$, $\kappa$, $\iota_Q$, and $\kappa_Q$ 
denote the inclusion maps 
$$\iota:N\hookrightarrow P,\quad \kappa:P \hookrightarrow M,\quad 
\iota_Q:N\hookrightarrow Q,\quad \kappa_Q:Q \hookrightarrow M.$$
Since $\biota\iota$, $\iota\biota$, etc are decomposed into two irreducible sectors, 
one of which is the identity sector, 
we choose representatives $\xi,\eta\in \End_0(P)$, $\xi_Q,\eta_Q\in \End_0(Q)$, 
$\hxi,\hxi_Q\in \End_0(N)$, and $\heta,\heta_Q\in \End_0(M)$ satisfying 
$$[\iota\biota]=[\id_P]\oplus [\xi],\quad [\bkappa\kappa]=[\id_P]\oplus[\eta],$$
$$[\iota_Q\biota_Q]=[\id_Q]\oplus [\xi_Q],\quad [\bkappa_Q\kappa_Q]=[\id_Q]\oplus[\eta_Q],$$
$$[\biota\iota]=[\biota_Q\iota_Q]=[\id_N]\oplus [\hxi],$$
$$[\kappa\bkappa]=[\id_M]\oplus [\heta],\quad [\kappa_Q\bkappa_Q]=[\id_M]\oplus[\heta_Q] .$$
Note that neither $\xi$ nor  $\hxi$ is an automorphism (see Remark \ref{3remark no-auto}). 
When $N\subset P$ and $N\subset Q$ are 3-supertransitive, the sectors $\xi\iota$ and 
$\xi_Q\iota_Q$ are decomposed into two irreducible components and 
we use the following notation: 
$$[\xi\iota]=[\iota]\oplus [\iota'],\quad [\xi_Q\iota_Q]=[\iota_Q]\oplus [\iota_Q'].$$
In this case, we also have 
$$[\iota\hxi]=[\iota]\oplus [\iota'],\quad [\iota_Q\hxi_Q]=[\iota_Q]\oplus [\iota_Q'].$$
In the same way, when $P\subset M$ and $Q\subset M$ are 3-supertransitive, we use the 
following notation: 
$$[\kappa\eta]=[\kappa]\oplus [\kappa'],\quad [\kappa_Q\eta_Q]=[\kappa_Q]\oplus [\kappa_Q'],$$
$$[\heta\kappa]=[\kappa]\oplus[\kappa'],\quad [\heta_Q\kappa_Q]=[\kappa_Q]\oplus[\kappa_Q'].$$

In the sequel we often use the techniques developed in \cite{I1}. 
For example, let $\rho,\sigma$ be irreducible endomorphisms of $P$ associated with 
the inclusion $N\subset P$. 
Then $\rho\xi$ contains $\sigma$ if and only if either $[\rho]=[\sigma]$ or 
the distance between $\rho$ and $\sigma$ in the dual principal graph of $N\subset P$ is two. 
When $[\rho]\neq [\sigma]$, the multiplicity of $\sigma$ in $\rho\xi$ is the number 
of the paths from $\rho$ to $\sigma$. 
The multiplicity of $\rho$ in $\rho\xi$ is $\sum_{i=1}^mn_i^2-1$ where 
$[\rho\iota]=\oplus_i^m n_i[\tau_i],$ is the irreducible decomposition of $\rho\iota$. 
In particular, the inclusion $N\subset P$ is 3-supertransitive if and only if 
the multiplicity of $\xi$ in $\xi^2$ is one. 

\begin{lemma} \label{4lemma xi-eta} 
Let  $\fQ=\quadri$ be an irreducible noncommuting quadrilateral of factors. 
We assume that $\fQ$ is (2,2)-supertransitive. 
Then 
\begin{itemize} 
\item [(1)] $[\xi]\neq [\eta]$. 
\item [(2)] $\xi^2$ contains $\eta$. 
Consequently, the depth of $N\subset P$ is at least 4. 
\end{itemize}
\end{lemma}

\begin{proof} (1) Since $N\subset M$ is irreducible, we have 
$\dim(\kappa\iota,\kappa\iota)=1$. 
Thus Frobenius reciprocity implies 
$$1=\dim(\kappa\iota,\kappa\iota)=(\bkappa\kappa,\iota\biota)=(\id_P\oplus \eta,\id_P\oplus \xi),$$
which shows $[\xi]\neq [\eta]$. 

(2) Since $\fQ$ is noncommuting and 
$$L^2(P)\cong L^2(Q)\cong L^2(N)\oplus L^2(N)_{\hxi},$$
as $N-N$ bimodules, the $N-N$ bimodule $L^2(M)$ contains $L^2(N)_{\hxi}$ with multiplicity 
at least two. 
Thus 
\begin{eqnarray*}2&\leq& \dim(\biota\bkappa\kappa\iota,\hxi)=\dim (\bkappa\kappa,\iota\hxi\biota)
=\dim  (\bkappa\kappa,\iota\biota\iota\biota)-\dim(\bkappa\kappa,\iota\biota)
\\ &=&\dim(\id_P\oplus \eta,\xi\oplus\xi^2)=1+\dim(\eta,\xi^2),
\end{eqnarray*}
which shows the statement. 
\end{proof}

\begin{remark} \label{4remrak no-extra} Using the above lemma, we give a short proof of 
Theorem \ref{2theorem no-extra} only assuming (4,4)-supertransitivity of 
$\fQ$ (see Corollary \ref{4corollary no-extra} and Corollary \ref{5corollary 6-3} below for 
sharper results in this direction). 
Since $N\subset P$ is 4-supertransitive, Lemma \ref{4lemma xi-eta} shows 
$$[\xi^2]=[\id_P]\oplus [\xi]\oplus [\eta].$$ 
If the principal graph is $A_5$, $\eta$ is an automorphism and 
$M$ is the crossed product $P\rtimes_\eta \Z/2\Z$. 
Therefore the proof of \cite[Theorem 3.1]{I2} implies that there exists an 
outer action of $\fS_3$ on $M$ such that $N$ is the fixed point algebra of the action. 
Assume now that the principal graph of $N\subset M$ is not $A_5$. 
Then we have $\dim(\eta\iota,\eta\iota)\geq 2$. 
Frobenius reciprocity implies 
\begin{eqnarray*}
\dim(\eta^2,\xi)&=&\dim(\eta,\eta\xi)=\dim(\eta,\eta\iota\biota)-\dim(\eta,\eta)
=\dim(\eta\iota,\eta\iota)-1\\
&\geq&1. 
\end{eqnarray*}
Since $P\subset M$ is 4-supertransitive, we get $[\eta^2]=[\id_P]\oplus [\eta]\oplus [\xi]$. 
Therefore $d(\xi)=d(\eta)$ and $d(\xi)^2=1+2d(\xi)$, 
which shows that $[M:P]=[P:N]=2+\sqrt{2}=4\cos^2(\pi/8)$ and all the elementary subfactors are the $A_7$ subfactor. 
Uniqueness of such a quadrilateral follows from  Lemma \ref{3lemma vanishing} 
and Lemma \ref{3lemma uniqueness} as $\eta$ is uniquely determined by $N\subset P$ 
(see Theorem \ref{4theorem uniqueness1} below for details). 
\end{remark} 

The following is one of the main results in this paper, which says that  
irreducible noncommuting (2,2)-supertransitive quadrilaterals with trivial cohomological 
obstruction are classified into four classes and one exceptional case. 

\begin{theorem}\label{4theorem (2,2)-supertransitive} 
Let  $\fQ=\quadri$ be an irreducible noncommuting quadrilateral of factors. 
We assume that $\fQ$ is (2,2)-supertransitive and the class $c(\fQ)\in H^2(\widehat{N\subset P})$ 
is trivial. 
Then 
$$\Theta(M,N)=\cos^{-1}\frac{1}{[P:N]-1}.$$
Moreover, 
\begin{itemize}
\item [$(1)$] If $\fQ$ is not cocommuting and the class $c(\hat{\fQ})$ in $H^2(P\subset M)$ 
is trivial, then $[M:P]=[P:N]$.  
In this case, there exists an outer automorphism $\alpha$ of $P$ 
such that $[\eta]=[\alpha\xi]=[\xi \alpha^{-1}]$ and  $\xi^2$ contains $\eta$. 
$($In consequence, the automorphism $\alpha$ is contained in $\xi^3$.$)$ 
\item [$(2)$] If $\fQ$ is cocommuting, then $[M:P]=[P:N]-1$ and $[\xi]=[\bkappa\kappa_Q\pi]$, 
where $\pi$ is an isomorphism from $P$ to $Q$ such that the restriction of $\pi$ to 
$N$ is the identity map. 
In this case, the Galois group 
$$\mathrm{Gal}(M/N)=\{\theta\in \Aut(M);\;\theta(x)=x,\; \forall x\in N\}$$ 
is either trivial or isomorphic to $\Z/2\Z$,$\Z/3\Z$, or the symmetric group 
$\fS_3$ of degree 3. 

If $\mathrm{Gal}(M/N)=\{\id_M,\theta\}\cong \Z/2\Z$, then $\theta$ switches $P$ and $Q$ and 
$[\xi]=[\bkappa\theta\kappa]$. 

If $\mathrm{Gal}(M/N)=\{\id_M,\theta,\theta^2\}\cong \Z/3\Z$, then either $\theta(P)=Q$ 
or $\theta(Q)=P$ holds and $[\xi]=[\bkappa\theta\kappa]=[\bkappa\theta^2\kappa]$. 

If $\mathrm{Gal}(M/N)\cong \fS_3$, then $N$ is the fixed point algebra of 
an outer action of $\fS_3$ on $M$ and $P$ and $Q$ are the fixed point algebras of 
two different order two elements in $\fS_3$. 
\end{itemize}
\end{theorem}

\begin{proof} The angle between $P$ and $Q$ was already computed in 
Theorem \ref{3theorem angle}. 
Lemma \ref{3lemma trivial} shows that since $c(\fQ)$ is trivial there exists an 
isomorphism $\pi$ from $P$ onto $Q$ such that the restriction of $\pi$ to $N$ is trivial.  
This means $\iota_Q=\pi\iota$. 
Thus $\kappa\iota=\kappa_Q\iota_Q$ implies $\kappa\iota=\kappa_Q\pi\iota$ and 
\begin{eqnarray*}1&=&\dim(\kappa\iota,\kappa_Q\pi\iota)=\dim(\kappa\iota\biota,\kappa_Q\pi)
=\dim(\kappa,\kappa_Q\pi)+\dim(\kappa\xi,\kappa_Q\pi)\\
&=&\dim(\kappa,\kappa_Q\pi)+\dim(\xi,\bkappa\kappa_Q\pi).
\end{eqnarray*}
We claim $[\kappa]\neq [\kappa_Q\pi]$, and in consequence, claim that $\xi$ is contained in 
$\bkappa\kappa_Q\pi$. 
Indeed, if $[\kappa]=[\kappa_Q\pi]$ were the case, there would exist a unitary $u$ in 
$M$ such that $x=u\pi(x)u^*$ holds for all $x\in P$. 
In particular, the unitary $u$ commutes with all $x\in N$ and $u$ is a scalar. 
However, this means $P=Q$, which is a contradiction. 
Therefore the claim holds. 

(1) Assume that $\fQ$ is not cocommuting and the class $c(\hat{\fQ})$ in $H^2(P\subset M)$ 
is trivial. 
Then Lemma \ref{3lemma trivial} applied to the dual quadrilateral $\hat{\fQ}$ implies that 
$P$ and $Q$ are inner conjugate in $M$ and there exists a unitary $v\in M$ satisfying 
$vQv^*=P$. 
We introduce an automorphism $\alpha$ of $P$ by setting 
$\alpha(x)=v\pi(x)v^*$ for $x\in P$. 
This implies $[\kappa\alpha]=[\kappa_Q\pi]$ and so 
$$[\bkappa\kappa_Q\pi]=[\bkappa\kappa\alpha]=[\alpha]\oplus [\eta\alpha].$$
Since $\xi$ is contained in $\bkappa\kappa_Q\pi$, either $[\xi]=[\alpha]$ or 
$[\xi]=[\alpha\eta]$ holds. 
Note that $[\xi]=[\alpha]$ never occurs thanks to Remark \ref{3remark no-auto}. 
Thus we have $[\xi]=[\eta\alpha]$ and so $[\eta]=[\xi\alpha^{-1}]$, which 
is equal to $[\alpha\xi]$ for $\eta$ is self-conjugate. 
Since $\xi^2$ contains $\eta$ and $\id_P$, we conclude that $\xi^3$ 
contains $\alpha$. 

(2) Assume that $\fQ$ is cocommuting. 
Then Theorem \ref{3theorem angle},(1) applied to the dual quadrilateral $\hat{\fQ}$ implies that 
$[\kappa\bkappa]\neq [\kappa_Q\bkappa_Q]$ and so 
$$\dim(\bkappa\kappa_Q\pi,\bkappa\kappa_Q\pi)=\dim (\kappa\bkappa,\kappa_Q\bkappa_Q)=1.$$
This shows that $\bkappa\kappa_Q\pi$ is irreducible and $[\xi]=[\bkappa\kappa_Q\pi]$, 
which implies  
$$[P:N]-1=d(\xi)=d(\kappa)d(\kappa_Q)=[M:P].$$
(This is also shown in Corollary \ref{3corollary index} with help of 
Lemma \ref{2lemma angle} though we don't need to use it here.) 

Note that the Galois group of the inclusion $N\subset M$ is isomorphic to the 
set of 1-dimensional sectors $G$ contained in $\kappa\iota\biota\bkappa$. 
(Note that since $N\subset M$ is irreducible, the multiplicity of any automorphism in 
$\kappa\iota\biota\bkappa$ is at most one.) 
We assume that $G$ is non-trivial now. 
Using $[\xi]=[\bkappa\kappa_Q\pi]$, we get 
\begin{eqnarray*}[\kappa\iota\biota\bkappa]
 &=& [\kappa\bkappa]\oplus [\kappa\xi\bkappa]=[\id_M]\oplus [\heta]\oplus 
 [\kappa\bkappa\kappa_Q\pi\bkappa]\\
 &=&[\id_M]\oplus [\heta]\oplus [\kappa_Q\pi\bkappa]
 \oplus [\heta\kappa_Q\pi\bkappa].
\end{eqnarray*}

Assume that $\heta$ is an automorphism. 
Then $[M:P]=2$ and so $[P:N]=3$. 
Since $N\subset P$ is 2-supertransitive, it is the $A_5$ subfactor. 
Therefore, the proof of \cite[Theorem 3.1]{I2} shows that there exists a unique outer 
action of $\fS_3$ on $M$ such that $N$ is the fixed point algebra of the action. 

Assume that $\heta$ is not an automorphism. 
We claim that $\kappa_Q\pi\bkappa$ contains an automorphism. 
By the assumption of non-triviality of $\Gal(M/N)$, either 
$\kappa_Q\pi\bkappa$ or $\heta\kappa_Q\pi\bkappa$ contains an automorphism. 
Assume that the latter contains an automorphism, say $\alpha$. 
Then Frobenius reciprocity implies 
$$1=\dim (\heta\kappa_Q\pi\bkappa,\alpha)=\dim (\kappa_Q\pi\bkappa,\heta\alpha),$$
and $\kappa_Q\pi\bkappa$ contains $\heta\alpha$. 
Since $d(\kappa_Q\pi\bkappa)=d(\heta\alpha)+1$, we conclude that $\kappa_Q\pi\bkappa$ 
contains an automorphism.  

Let $\theta$ be an automorphism contained in $\kappa_Q\pi\bkappa$.  
Then Frobenius reciprocity implies $[\theta\kappa]=[\kappa_Q\pi]$. 
We may assume that $\theta\kappa=\kappa_Q\pi$ holds by choosing an appropriate 
representative of $[\theta]$. 
Then $\theta(P)=Q$ and $\theta\kappa\iota=\kappa_Q\pi\iota=\kappa_Q\iota_Q=\kappa\iota$, 
which shows $\theta\in \mathrm{Gal}(M,N)$. 
Using $[\theta\kappa]=[\kappa_Q\pi]$, we get 
\begin{eqnarray*}[\kappa\iota\biota\bkappa]
 &=&[\id_M]\oplus [\heta]\oplus [\theta\kappa\bkappa]
 \oplus [\heta\theta\kappa\bkappa]\\
 &=&[\id_M]\oplus [\theta]\oplus [\heta]\oplus [\theta\heta]\oplus [\heta\theta]
 \oplus [\heta\theta\heta].
\end{eqnarray*}
If $\heta\theta\heta$ does not contain any automorphism, then 
$\Gal(M,N)=\{\id_M,\theta\}\cong \Z/2\Z$. 

Assume that $\heta\theta\heta$ does contain an automorphism, say, $\beta$. 
Then Frobenius reciprocity again implies $[\heta\theta]=[\beta\heta]$ 
(and $[\theta\heta]=[\heta\beta]$), and  
$$[\kappa\iota\biota\bkappa]=[\id_M]\oplus [\theta]\oplus [\heta]\oplus [\theta\heta]
\oplus [\heta\theta] \oplus [\beta\heta^2].$$
Therefore 
$$G=\{[\id_M],[\theta]\}\cup\{[\beta\gamma]\}_{\gamma\in \Gal(M/\heta(M))}.$$ 
Since $[\beta]\in G$, the group $G_1=\{[\gamma]\}_{\gamma\in \Gal(M/\heta(M))}$ 
is a subgroup of $G$ such that $\# G=2+\# G_1$. 
Since $\#G_1$ divides $\#G$, either $G_1$ is trivial or $\#G_1=2$. 
If $G_1$ is trivial, we get $\Gal(M/N)\cong \{[\id_M],[\theta],[\beta]\} \cong \Z/3\Z$. 

Suppose that $G_1=\{[\id_M],[\gamma]\}$. 
Then $G=\{[\id_M],[\theta],[\beta],[\beta\gamma]\}$ and $[\theta]=[\gamma]$. 
This would imply $[\theta\heta]=[\gamma\heta]=[\heta]$ and 
$[\heta\theta\heta]=[\heta^2]$ would contain $[\id_M]$. 
However, this means that $\kappa\iota\biota\bkappa$ contain $\id_M$ with 
multiplicity two, which is a contradiction. 
Therefore $G_1$ is trivial. 
\end{proof}

\begin{corollary} \label{4corollary noncommuting-noncocommuting} 
Let $\fQ=\quadri$ be a (3,3)-supertransitive irreducible quadrilateral of 
factors such that $\fQ$ is neither commuting nor cocommuting. 
Then $\Theta(P,Q)=\cos^{-1}1/([P:N]-1)$, $[M:P]=[P:N]$ and 
there exists an outer automorphism $\alpha$ of $P$ such that 
$[\eta]=[\alpha\xi]=[\xi \alpha^{-1}]$.  
Moreover, 
\begin{itemize}
\item [$(1)$] The automorphism $\alpha$ satisfies $[\alpha\iota']=[\iota']$. 
\item [$(2)$] $\dim (\kappa\iota,\kappa\iota\hxi)=2$. 
\item [$(3)$] The Galois group $\Gal(M/N)$ is either trivial or isomorphic to $\Z/2\Z$. 
It is  isomorphic to $\Z/2\Z$ if and only if the order of $[\alpha]$ is two. 
When $\Gal(M/N)=\{\id,\theta\}$, the automorphism $\theta$ switches $P$ and $Q$. 
\end{itemize}
\end{corollary}

\begin{proof} The first part follows from Lemma \ref{3lemma vanishing} and 
Theorem \ref{4theorem (2,2)-supertransitive}. 

(1) Since $\alpha\xi$ is contained in $\iota'\biota$, by Frobenius reciprocity 
$\iota'$ is contained in $\alpha\xi\iota=\alpha\iota\oplus \alpha\iota'$. 
Since $d(\iota')=d(\iota)(d(\iota)^2-2)$, 
the equality $[\iota']=[\alpha\iota]$ would imply that $N\subset P$ is the $A_5$ subfactor, 
which does not allow $\xi^2$ to contain $\alpha\xi$, and so $[\alpha\iota']=[\iota']$. 

(2) Thanks to Lemma \ref{4lemma xi-eta}, we have 
\begin{eqnarray*}
[\biota\bkappa\kappa\iota]&=&[\biota\iota]\oplus [\biota\eta\iota]
=[\id_N]\oplus[\hxi]\oplus [\biota\alpha\xi\iota]
=[\id_N]\oplus[\hxi]\oplus [\biota\alpha\iota]\oplus [\biota\alpha\iota']\\
&=&[\id_N]\oplus[\hxi]\oplus [\biota\alpha\iota]\oplus [\biota\iota'].
\end{eqnarray*}
Frobenius reciprocity implies 
$$\dim (\biota\alpha\iota,\hxi)=\dim (\alpha\iota,\iota\hxi)=\dim (\alpha\iota,\iota)
+\dim (\alpha\iota,\iota')=\dim (\alpha\iota,\iota),$$
where we use $[\alpha\iota]\neq [\iota']$ again. 
Since the Galois group of $N\subset P$ is trivial, we have $\dim (\alpha\iota,\iota)=0$ 
and $\dim (\biota\alpha\iota,\hxi)=0$. 
Since $N\subset P$ is 3-supertransitive, the endomorphism $\biota\iota'$ contains 
$\hxi$ with multiplicity one and we get the statement. 

(3) The Galois group $\Gal(M/N)$ is isomorphic to the set of 1 dimensional sectors 
contained in $\kappa\iota\biota\bkappa$. 
By a similar computation as above, 
$$[\kappa\iota\biota\bkappa]=[\kappa\bkappa]\oplus [\kappa\xi\bkappa]
=[\id_M]\oplus [\heta]\oplus [\kappa\eta\alpha\bkappa]
=[\id_M]\oplus [\heta]\oplus [\kappa\alpha\bkappa]\oplus [\kappa'\alpha\bkappa].$$
Since the depth of $P\subset M$ is at least 4 due to Lemma \ref{4lemma xi-eta}, 
the endomorphism $\heta$ is not an automorphism. 
If $\kappa'\alpha\bkappa$ contained an automorphism $\theta$, Frobenius reciprocity 
would imply $[\kappa'\alpha]=[\theta\kappa]$ and $d(\kappa)=d(\kappa')$, which is a 
contradiction. 

Assume that $\Gal(M/N)$ is not trivial. 
Then $\kappa\alpha\bkappa$ contains an automorphism, say, $\theta$. 
Then Frobenius reciprocity implies $[\theta\kappa]=[\kappa\alpha]$ and 
$$[\id_P]\oplus [\alpha\xi]=[\bkappa\kappa]=[\alpha^{-1}\bkappa\kappa\alpha]
=[\id_P]\oplus [\alpha^{-1}\alpha\xi\alpha]=[\id_P]\oplus [\xi\alpha].$$
Since $[\alpha\xi]=[\xi\alpha^{-1}]$, this shows that $\xi^2$ contains $\alpha^2$. 
If $\alpha^2$ were outer, we would have $[\alpha^2\iota]=[\iota']$ as 
$\iota'\biota$
would contain $\alpha^2$. 
However this contradicts $d(\iota)\neq d(\iota')$. 
Thus the order of $[\alpha]$ is two. 
This implies $[\theta^2\kappa]=[\kappa\alpha^2]=[\kappa]$. 
Since the Galois group of $P\subset M$ is trivial, we get $[\theta^2]=[\id_M]$. 
Let $\beta$ be another automorphism contained in $\kappa\alpha\bkappa$. 
Then the above argument shows that $[\theta\kappa]=[\beta\kappa]$ and 
$[\beta^{-1}\theta\kappa]=[\kappa]$, which implies $[\beta^{-1}\theta]=[\id_M]$. 
Therefore $\Gal(M/N)\cong \Z/2\Z$. 

Assume now that the order of $[\alpha]$ is two. 
Then 
$$[\overline{(\kappa\alpha)}\kappa\alpha]=[\alpha^{-1}(\id_P\oplus \eta)\alpha]
=[\id_P]\oplus [\alpha^{-1}\eta\alpha]=[\id_P]\oplus [\eta]=[\bkappa\kappa].$$ 
Since $H^2(\widehat{P\subset M})$ is trivial, the subfactors 
$\alpha^{-1}\bkappa(M)$ and $\bkappa(M)$ are inner conjugate in $P$ and so 
there exists an automorphism $\theta$ of $M$ such that $[\theta\kappa]=[\kappa\alpha]$. 
Since the Galois group of $\bkappa(M)\subset P$ is trivial, we have 
$[\kappa\alpha]\neq [\kappa]$ and $\theta$ is outer. 
The same computation as above implies that $[\theta]$ is contained in 
$[\kappa\iota\biota\bkappa]$ and so $\Gal(M/N)$ is not trivial. 

The above argument shows that when $\Gal(M/N)$ is isomorphic to $\Z/2\Z$, we have 
$[\theta\kappa]=[\kappa\alpha]=[\kappa_Q\pi]$ and we can choose a representative 
$\theta$ of $[\theta]$ so that $\theta\kappa=\kappa_Q\pi$ holds.  
Then $\theta(P)=Q$ and $\theta\kappa\iota=\kappa_Q\pi\iota=\kappa_Q\iota_Q=\kappa\iota$, 
and so $\Gal(M/N)=\{\id,\theta\}$. 

\end{proof}

\begin{corollary} \label{4corollary Z/3Z} 
Let  $\fQ=\quadri$ be an irreducible noncommuting quadrilateral of factors 
such that $\fQ$ is (2,2)-supertransitive and the class $c(\fQ)$ is trivial. 
Assume that $\fQ$ is cocommuting and the Galois group $\Gal(M/N)$ is isomorphic to 
$\Z/3\Z$. 
Then $[M:P]\geq(5+\sqrt{13})/2$ and $\fQ$ is neither (4,2)-supertransitive nor 
(2,3)-supertransitive. 
If $\fQ$ is (3,2)-supertransitive, the dual principal graph (to be more precise, the 
induction-reduction graph for $M-M$ and $M-P$ bimodules) contains one of 
the   graphs of the Haagerup subfactor. 
\end{corollary}

\begin{proof} 
Let $\Gal(M/N)=\{\id_M,\theta,\theta^2=\theta^{-1}\}$.
The proof of Theorem \ref{4theorem (2,2)-supertransitive} shows that 
$[\theta\heta]=[\heta\theta^{-1}]$, the endomorphism $\heta^2$ does not contain any 
non-trivial automorphism and 
$$[\kappa\iota\biota\bkappa]=[\id_M]\oplus [\theta]\oplus [\heta]\oplus [\theta\heta]
\oplus [\theta^2\heta] \oplus [\theta^2\heta^2].$$
Note that $[\heta]$, $[\theta\heta]$, $[\theta^2\heta]$ are distinct sectors. 
Since $\heta^2$ contains $\heta$, the endomorphism $\kappa\iota\biota\bkappa$ 
contains $\theta^2\heta$ with multiplicity at least 2. 
Since $[\theta\kappa\iota]=[\kappa\iota]$, the multiplicities of $\heta$ and $\theta\heta$ 
in $\kappa\iota\biota\bkappa$ are also at least 2, which implies 
$\heta^2$ contains $\id_M\oplus \heta\oplus \theta\heta\oplus \theta^2\heta$. 
This shows that $P\subset M$ is not 4-supertransitive and 
$d(\heta)^2\geq 1+3d(\heta)$, which implies $[M:P]=1+d(\heta)\geq (5+\sqrt{13})/2$. 
Note that $N\subset P$ is 3-supertransitive if and only if $\dim(\xi,\xi^2)=1$. 
Using $[\xi]=[\bkappa\theta\kappa]$, we get 
\begin{eqnarray*}
\dim(\xi,\xi^2) &=&\dim(\bkappa\theta\kappa,\bkappa\theta\kappa\bkappa\theta\kappa)
=\dim(\kappa\bkappa\theta\kappa\bkappa,\theta\kappa\bkappa\theta)\\
&=&\dim((\id_M\oplus \heta)\theta(\id_M\oplus \heta),\theta(\id_M\oplus \heta)\theta)\\
&=&\dim(\theta\oplus \theta\heta\oplus \theta^2\heta\oplus \theta^2\heta^2, 
\theta^2\oplus\heta)\\
&=&\dim(\theta^2\heta^2, \theta^2\oplus\heta)\geq 2,
\end{eqnarray*}
which shows that $N\subset P$ is not 3-supertransitive. 

Assume that $\fQ$ is (3,2)-supertransitive. 
Then $\kappa'\bkappa$ contains $\heta\oplus\theta\heta\oplus\theta^2\heta$. 
and the dual principal graph of $P\subset M$ contains the following graph,
$$\hpic{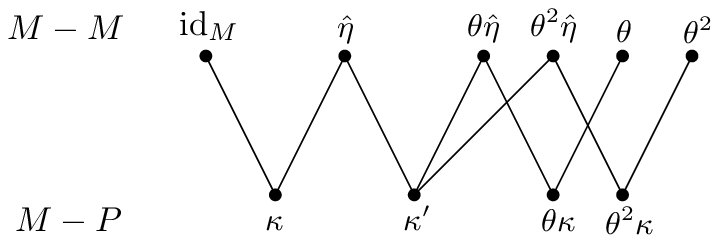}{1.0in}$$
which is one the principal graphs of the Haagerup subfactor \cite{AH}. 
\end{proof}

In Section 5 and in the Appendix, we will show that there exists a (3,2)-supertransitive 
quadrilateral of factors in the above class such that $P\subset M$ is the Haagerup 
subfactor and that such a quadrilateral is unique up to flip conjugacy. 

\begin{corollary} \label{4corollary etaiota=iota'}
Let  $\fQ=\quadri$ be an irreducible noncommuting quadrilateral of factors.  
Assume that $\fQ$ is cocommuting and (2,3)-supertransitive. 
Then $\Theta(P,Q)=\cos^{-1}1/([P:N]-1)$, $[M:P]=[P:N]-1$, 
and the Galois group $\Gal(M/N)$ is either trivial or isomorphic to $\Z/2\Z$ 
or $\fS_3$. 
Moreover, 
\begin{itemize}
\item [$(1)$] Two equalities $[\eta\iota]=[\iota']$ and 
$[\biota\bkappa\kappa\iota]=[\id_M]\oplus [\hxi]\oplus [\biota\iota']$ hold. 
In consequence, the endomorphism $\iota'\biota$ contains $\eta$ with multiplicity one 
and $[\iota'\biota]=[\eta]\oplus [\eta\xi].$
\item [$(2)$] $\dim (\kappa\iota,\kappa\iota\hxi)=2$. 
\item [$(3)$] Assume that $\Gal(M/N)$ is isomorphic to $\Z/2\Z$ and 
$\Gal(M/N)=\{\id_M,\theta\}$. 
Then 
$$[\kappa\iota\biota\bkappa]=[\id_M]\oplus [\theta]\oplus[\heta]\oplus[\theta\heta]\oplus 
[\heta\theta]\oplus [\heta\theta\heta],$$
and $\dim(\kappa\iota\biota\bkappa,\kappa\iota\biota\bkappa)\geq 7$. 
The sectors $[\heta]$, $[\theta\heta]$, $[\heta\theta]$, and $[\theta\heta\theta]$ are 
distinct. 
The endomorphism $\heta\theta\heta$ contains $\theta\heta\theta$ with multiplicity one 
and it does not contain either any automorphism or any of $\heta$, $\theta\heta$, $\heta\theta$.  
The endomorphism $\heta^2$ does not contain any of $\theta\heta$, $\heta\theta$, 
$\theta\heta\theta$. 
\item[$(4)$] The endomorphism $\xi^2$ contains $\bpi\eta_Q\pi$ and $[\bpi\eta_Q\pi\iota]=[\iota']$ holds. 
Assume that $H^2(\widehat{P\subset M})$ is trivial. 
Then the Galois group $\Gal(M/N)$ is trivial if and only if 
$[\eta]\neq [\bpi\eta_Q\pi]$. 
Assume further that  $\Gal(M/N)$ is trivial.  
Then the dual principal graph of $N\subset P$ contains $D_6$ in such a way that every endpoint 
of $D_6$ is also an endpoint of the dual principal graph. 
\item [$(5)$] Assume that $P\subset M$ is 3-supertransitive and $\Gal(M/N)$ is trivial. 
Then $Q\subset M$ is 3-supertransitive as well and 
$$[\kappa\iota\biota\bkappa]=[\id_M]\oplus [\heta]\oplus 2[\kappa_Q\pi\bkappa]
\oplus[\kappa_Q'\bkappa_Q]=[\id_M]\oplus [\heta_Q]\oplus 2[\kappa_Q\pi\bkappa]
\oplus[\kappa'\bkappa].$$  
The sector $[\kappa_Q\pi\bkappa]$ is irreducible and self-conjugate and 
$[\kappa']=[\kappa'_Q\pi]$ holds. 
Unless the principal graph of $P\subset M$ is $A_4$, the inequality 
$\dim(\kappa\iota\biota\bkappa,\kappa\iota\biota\bkappa)\geq 8$ holds.
\end{itemize}
\end{corollary}

\begin{proof} The first part follows from Lemma \ref{3lemma trivial}, 
Theorem \ref{4theorem (2,2)-supertransitive}, and 
Corollary \ref{4corollary Z/3Z}. 

(1) Let $d=d(\iota)=[P:N]-1$. 
Then we have $d(\xi)=d^2-1$ and $d(\iota')=d(d^2-2)$. 
On the other hand, we have $d(\eta)=[M:P]-1=[P:N]-2=d^2-2$. 
Since $\eta\iota$ contains $\iota'$ and $d(\eta\iota)=d(\iota')$, we 
get $[\eta\iota]=[\iota']$. 

(2) Using (1), we get  
$$[\biota\bkappa\kappa\iota]= [\biota\iota]\oplus[\biota\eta\iota]=
[\id_N]\oplus [\hxi]\oplus [\biota\iota'].$$
Since $N\subset P$ is 3-supertransitive, $\biota\iota'$ contains $\hxi$ with multiplicity one 
and $\dim(\kappa\iota,\kappa\iota\hxi)=\dim(\biota\bkappa\kappa\iota,\hxi)=2$. 

(3) The proof of Theorem \ref{4theorem (2,2)-supertransitive} shows that 
$\heta\theta\heta$ contain no automorphism, $[\theta\kappa]=[\kappa_Q\pi]$, 
$[\xi]=[\bkappa\theta\kappa]$, and 
$$[\kappa\iota\biota\bkappa]=
 [\id_M]\oplus [\theta]\oplus [\heta]\oplus [\theta\heta]\oplus [\heta\theta]
 \oplus [\heta\theta\heta].$$
If $[\theta\heta]=[\heta]$, we would have $[\heta\theta\heta]=[\heta^2]$, which 
contains $\id_M$. 
Thus $[\theta\heta]\neq [\heta]$. 
In the same way, we get $[\heta\theta]\neq [\heta]$ and 
$[\theta\heta]\neq[\theta\heta\theta]\neq [\heta\theta]$. 
If $[\heta]=[\theta\heta\theta]$, we would have 
$$[\kappa\bkappa]=[\id_M]\oplus[\heta]=[\id_M]\oplus[\theta\heta\theta]=[\theta\kappa\bkappa\theta]
=[\kappa_Q\pi\bpi\bkappa_Q]=[\kappa_Q\bkappa_Q],$$
which contradicts the assumption that $\fQ$ is cocommuting thanks to 
Theorem \ref{3theorem angle},(1). 
Thus $[\heta]\neq [\theta\heta\theta]$ and $[\theta\heta]\neq[\heta\theta]$. 
Since $N\subset P$ is 3-supertransitive, we have $\dim (\xi,\xi^2)=1$, and 
\begin{eqnarray*}
1&=&\dim(\xi,\xi^2)=\dim(\bkappa\theta\kappa,\bkappa\theta\kappa\bkappa\theta\kappa)
=\dim(\kappa\bkappa\theta\kappa\bkappa,\theta\kappa\bkappa\theta)\\
&=&\dim((\id_M\oplus \heta)\theta(\id_M\oplus \heta),\theta(\id_M\oplus \heta)\theta)\\
&=&\dim(\theta\oplus\theta\heta\oplus\heta\theta\oplus \heta\theta\heta,
\id_M\oplus \theta\heta\theta)\\
&=&\dim(\heta\theta\heta,\theta\heta\theta),
\end{eqnarray*}
which means that $\heta\theta\heta$ contains $\theta\heta\theta$ with multiplicity one. 
Since $[\theta\kappa\iota]=[\kappa\iota]$, the multiplicities of 
$\heta$, $\theta\heta$, $\heta\theta$ and $\theta\heta\theta$ in 
$\kappa\iota\biota\bkappa$ should be the same, which is the same as 
$\heta\theta\heta$ contains $\theta\heta\theta$ with multiplicity one. 
Thus $\heta\theta\heta$ does not contain any of $\heta$, $\theta\heta$, $\heta\theta$. 
Thanks to Frobenius reciprocity we conclude that $\heta^2$ does not contain any of 
$\theta\heta$, $\heta\theta$, $\theta\heta\theta$. 

(4) Since $[\xi]=[\bkappa\kappa_Q\pi]=[\bpi\bkappa_Q\kappa]$, we have 
$$[\xi^2]=[\bpi\bkappa_Q\kappa\bkappa\kappa_Q\pi]=
[\bpi\bkappa_Q\kappa_Q\pi]\oplus[\bpi\bkappa_Q\heta\kappa_Q\pi],$$
which contains $\bpi\eta_Q\pi$. 
Since $N\subset P$ is 3-supertransitive and $\bpi\eta_Q\pi$ (which 
is not equivalent to $\id$) is contained in $\xi^2$, 
the morphism $\iota'\biota$ contains $\bpi\eta_Q\pi$ and so 
$\bpi\eta_Q\pi\iota$ contains $\iota'$. 
Since $d(\iota')=d(\bpi\eta_Q\pi\iota)$, we conclude 
$[\bpi\eta_Q\pi\iota]=[\iota']$. 

The proof of Theorem \ref{4theorem (2,2)-supertransitive} shows that 
$\Gal(M/N)$ is non-trivial if and only if there exists and automorphism 
$\theta$ of $M$ such that $[\theta\kappa]=[\kappa_Q\pi]$. 
Assume first that such an automorphism $\theta$ exists. 
Then $[\bkappa\kappa]=[\bpi\bkappa_Q\kappa_Q\pi]$ and 
$[\eta]=[\bpi\eta_Q\pi]$ holds. 
Assume conversely that $[\eta]=[\bpi\eta_Q\pi]$ holds. 
Then 
$$[\bkappa\kappa]=[\id_P]\oplus [\eta]=[\bpi(\id_Q\oplus \eta_Q)\pi]
=[\bpi\bkappa_Q\kappa_Q\pi].$$
Since $H^2(\widehat{P\subset M})$ is trivial, the subfactors $\bkappa(M)$ and 
$\bpi\bkappa_Q(M)$ are inner conjugate in $P$ and there exists 
an automorphism $\theta$ of $M$ such that $[\theta\kappa]=[\kappa_Q\pi]$. 
Therefore the Galois group $\Gal(M/N)$ is trivial if and only if 
$[\eta]\neq [\bpi\eta_Q\pi]$. 

Now we assume that $\Gal(M/N)$ is trivial. 
Then the dual principal graph of $N\subset P$ is as follows: 
$$\hpic{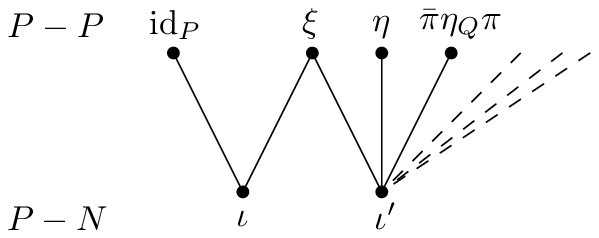}{1.0in}$$
which contains $D_6$. 

(5) 
Since $[\eta]\neq [\bpi\eta_Q\pi]$, 
$$\dim(\kappa_Q\pi\bkappa,\kappa_Q\pi\bkappa)
=\dim(\bpi\bkappa_Q\kappa_Q\pi,\bkappa\kappa)
=\dim(\id_P\oplus \bpi\eta_Q\pi,\id_P\oplus \eta)=1,$$
and $\kappa_Q\pi\bkappa$ is irreducible. 
Since $[\xi]=[\bkappa\kappa_Q\pi]$ is self-conjugate, 
$$1=\dim(\bkappa\kappa_Q\pi,\bpi\bkappa_Q\kappa)
=\dim(\kappa_Q\pi\bkappa,\kappa\bpi\bkappa_Q),$$
and $\kappa_Q\pi\bkappa$ is self-conjugate. 
This and the proof of Theorem \ref{4theorem (2,2)-supertransitive} imply  
\begin{eqnarray*}
[\kappa\iota\biota\bkappa]&=&[\id_M]\oplus[\heta]\oplus [\kappa_Q\pi\bkappa]\oplus 
[\heta\kappa_Q\pi\bkappa]
=[\id_M]\oplus[\heta]\oplus [\kappa_Q\pi\bkappa]\oplus 
[\heta\kappa\bpi\bkappa_Q]\\
&=&[\id_M]\oplus[\heta]\oplus 2[\kappa_Q\pi\bkappa]\oplus 
[\kappa'\bpi\bkappa_Q].
\end{eqnarray*}
On the other hand, replacing $P$ by $Q$, we get
$$[\kappa\iota\biota\bkappa]=[\id_M]\oplus[\heta_Q]\oplus [\kappa_Q\pi\bkappa]\oplus 
[\heta_Q\kappa_Q\pi\bkappa].$$
Since $[\heta]\neq [\heta_Q]$, the endomorphism $\heta_Q\kappa_Q\pi\bkappa$ contains 
$\heta$ and 
$$1\leq \dim (\heta_Q\kappa_Q\pi\bkappa,\heta)=\dim (\heta_Q\kappa_Q\pi,\heta\kappa)
=\dim (\heta_Q\kappa_Q\pi,\kappa\oplus\kappa').$$
If $\heta_Q\kappa_Q\pi$ contained $\kappa$, we would have 
$[\heta_Q]=[\kappa\bpi\bkappa_Q]$, which contradicts $d(\heta_Q)\neq d(\kappa\bpi\bkappa_Q)$. 
Therefore $\heta_Q\kappa_Q\pi$ contains $\kappa_Q\pi$ and $\kappa'$, and so 
$[\heta_Q\kappa_Q\pi]=[\kappa_Q\pi]\oplus [\kappa']$ as we have 
$d(\heta_Q\kappa_Q\pi)=d(\kappa_Q\pi)+d(\kappa')$. 
This shows that $Q\subset M$ is 3-supertransitive and $[\kappa_Q'\pi]=[\kappa']$, 
which finishes the proof. 
\end{proof}

About uniqueness, we have the following:

\begin{theorem}\label{4theorem uniqueness1} Let $\fQ=\quadri$ and 
$\tilde{\fQ}=\tquadri$ be irreducible noncommuting and noncocommuting 
quadrilaterals of factors such that 
$\fQ$ and $\tfQ$ are (3,3)-supertransitive. 
Assume that there exists an isomorphism $\Phi$ from $P$ onto 
$\tilde{P}$ such that $\Phi(N)=\tN$. 
If there exists no sector $\sigma$ contained in $\xi^2$ such that 
$[\xi]\neq [\sigma]$, $[\eta]\neq [\sigma]$, and $d(\sigma)=d(\eta)$, then 
$\fQ$ and $\tfQ$ are conjugate. 
\end{theorem}

\begin{proof} Since $P\subset M$ and $\tP\subset \tM$ are 3-supertransitive, 
$H^2(\widehat{P\subset M})$ and $H^2(\widehat{\tP\subset \tM})$ are trivial 
thanks to Lemma \ref{3lemma vanishing}. 
Note that we have $[\iota_{\tP,\tN}\overline{\iota_{\tP,\tN}}]
=[\Phi\iota_{P,N}\overline{\iota_{P,N}}\Phi^{-1}]$. 
Therefore Lemma \ref{4lemma xi-eta} implies $[\overline{\iota_{\tM,\tP}}\iota_{\tM,\tP}]
=[\Phi\overline{\iota_{M,P}}\iota_{M,P}\Phi^{-1}]$, which   
shows that $\Phi$ extends to an isomorphism $\Psi$ from $M$ onto $\tM$ 
as $H^2(\widehat{P\subset M})$ and $H^2(\widehat{\tP\subset \tM})$ are trivial. 
Lemma \ref{3lemma uniqueness} and Corollary \ref{4corollary noncommuting-noncocommuting} 
show $\Psi(Q)=\tQ$. 
\end{proof}

\begin{theorem}\label{4theorem uniqueness2} Let $\fQ=\quadri$ and 
$\tilde{\fQ}=\tquadri$ be irreducible noncommuting and cocommuting 
quadrilaterals of factors such that $H^2(\widehat{P\subset M})$ and 
$H^2(\widehat{\tP\subset \tM})$ are trivial and $\fQ$ and $\tfQ$ are (2,3)-supertransitive. 
Assume that there exists an isomorphism $\Phi$ from $P$ onto $\tilde{P}$ such that $\Phi(N)=\tN$. 
Then 
\begin{itemize} 
\item [$(1)$] If $\Gal(M/N)\cong \Z/2\Z$ and $\eta$ is the only sector $\sigma$ contained 
in $\xi^2$ such that $[\sigma\iota]=[\iota']$, 
then $\fQ$ and $\tfQ$ are conjugate. 
\item [$(2)$] If $\Gal(M/N)$ and $\Gal(\tM/\tN)$ are trivial and there exists only two sectors $\sigma$ contained 
in $\xi^2$ such that $[\sigma\iota]=[\iota']$,
then $\fQ$ and $\tfQ$ are flip conjugate. 
\end{itemize}
\end{theorem}

\begin{proof} The proof of (1) is the same as that of Theorem \ref{4theorem uniqueness1}. 
In (2), 
$$[\overline{\iota_{\tM,\tP}}\iota_{\tM,\tP}]
=[\Phi(\id_P\oplus \bpi\eta_Q\pi)\Phi^{-1}]
=[\Phi\pi^{-1}\overline{\iota_{M,Q}}\iota_{M,Q}\pi\Phi^{-1}]$$ 
may occur instead of $[\overline{\iota_{\tM,\tP}}\iota_{\tM,\tP}]
=[\Phi\overline{\iota_{M,P}}\iota_{M,P}\Phi^{-1}]$. 
In this case, the two quadrilaterals are flip conjugate. 
\end{proof}

We can improve the assumption of Theorem \ref{2theorem no-extra}. 

\begin{corollary} \label{4corollary no-extra}
Let  $\fQ=\quadri$ be an irreducible noncommuting quadrilateral of factors. 
Then 
\begin{itemize} 
\item [$(1)$] If $\fQ$ is not cocommuting and (3,4)-supertransitive, 
then all the elementary subfactors are isomorphic to $A_7$ subfactors. 
There exists a unique such quadrilateral up to isomorphism when the factors 
are hyperfinite II$_1$ factors. 
\item [$(2)$] If $\fQ$ is cocommuting and (2,4)-supertransitive, then there exists 
an outer action of $\fS_3$ on $M$ such that $N$ is the fixed point algebra of the action. 
There exists a unique such quadrilateral up to isomorphism when the factors 
are hyperfinite II$_1$ factors. 
\end{itemize}
\end{corollary}

\begin{proof} Since $N\subset P$ is 4-supertransitive, we have 
$[\xi^2]=[\id_P]\oplus [\xi]\oplus [\eta]$. 

(1) Since $\fQ$ is not cocommuting, there exists an automorphism $\theta$ of $P$ such that 
$[\eta]=[\theta\xi]$ and so we get $d(\xi)^2=1+2d(\xi)$. 
This shows that the principal graphs of $N\subset P$ and $P\subset Q$ are $A_7$. 
The uniqueness follows from Theorem \ref{4theorem uniqueness1}. 

(2) Since $\fQ$ is cocommuting, we have $[\eta\iota]=[\iota']$ and so the principal graph of 
$N\subset P$ is $A_5$. 
The rest of the proof has already been stated in Remark \ref{4remrak no-extra} except for 
uniqueness, which follows from the uniqueness of outer actions of finite groups by 
Jones \cite{J1} and the Galois correspondence. 
\end{proof}

\section{Classification I} 
According to our discussions in the last section, we consider the following 
four classes of quadrilaterals $\fQ$ of factors. 
We assume that every quadrilateral $\fQ=\quadri$ of factors appearing in this section 
is irreducible and noncommuting. \\
Class I: $\fQ$ is noncocommuting and (3,3)-supertransitive. \\
Class II: $\fQ$ is cocommuting and (2,3)-supertransitive. 
The Galois group $\Gal(M/N)$ is trivial. \\
Class III: $\fQ$ is cocommuting and (2,3)-supertransitive. 
The Galois group $\Gal(M/N)$ is isomorphic to $\Z/2\Z$. \\
Class IV: $\fQ$ is cocommuting and (3,2)-supertransitive. 
The Class $c(\fQ)$ is trivial. 
The Galois group  $\Gal(M/N)$ is isomorphic to $\Z/3\Z$.

For each class, we show that there exists an example of a quadrilateral and we 
seek an example with maximal supertransitivity. 

We keep using the notation of sectors such as $\iota$, $\kappa$, etc. as in the previous section, 
We often use explicit formulae of fusion rules of subfactors and the reader is 
refered to \cite{I1} for them. 

\subsection{Class I}
The structure of quadrilaterals in Class I is relatively easy to describe. 

\begin{theorem}\label{5theorem I1} Let $N\subset P$ be an irreducible 
3-supertransitive inclusion of factors. 
Let $\iota=\iota_{P,N}$ and let 
$[\iota_{P,N}\overline{\iota_{P,N}}]=[\id_P]\oplus [\xi]$ and $[\xi\iota]
=[\iota]\oplus [\iota']$ be the irreducible decomposition. 
Assume that there exists an outer automorphism $\alpha\in \Aut(P)$ such that 
$[\xi]\neq [\alpha\xi]=[\xi\alpha^{-1}]$ and $\xi^2$ contains $\alpha\xi$. 
If $\id_P\oplus \alpha\xi$ has a $Q$-system, then there exists a unique irreducible 
noncommuting quadrilateral $\fQ=\quadri$ in Class I such that $[\eta]=[\alpha\xi]$. 
\end{theorem}

\begin{proof} Assume that there exists a $Q$-system for $\id_P\oplus \alpha\xi$. 
We claim that a $Q$-system with $\id_P\oplus \alpha\xi$ is unique up to 
equivalence. 
For this, it suffices to show $\dim (\alpha\xi,(\alpha\xi)^2)=1$ thanks to Lemma \ref{3lemma Q-system}. 
As in the proof of Corollary \ref{4corollary Z/3Z}, we have 
$[\alpha\iota']=[\iota']$, $[\alpha\iota]\neq [\iota]$, and $d(\iota)\neq d(\iota')$. 
Thus 
\begin{eqnarray*}
\dim(\alpha\xi,(\alpha\xi)^2)&=&\dim(\alpha\xi,\xi^2)=\dim(\alpha\xi,\xi\iota\biota)-\dim(\alpha\xi,\xi)=
\dim(\alpha\xi\iota,\xi\iota)\\
&=&\dim(\alpha\iota\oplus \alpha\iota',\iota\oplus \iota')=1,
\end{eqnarray*}
which shows the claim. 
The claim implies that there exists a unique factor $M$ containing $N$ such that 
$[\bkappa\kappa]=[\id_P]\oplus [\alpha\xi]$ where $\kappa=\iota_{M,P}$. 
Since $[\xi]\neq [\alpha\xi]$, the inclusion $N\subset M$ is irreducible. 

We next claim that $\kappa\alpha\iota$ is irreducible and 
$[\kappa\alpha\iota]=[\kappa\iota]$. 
Indeed, we have
$$\dim(\kappa\alpha\iota,\kappa\alpha\iota)=\dim(\bkappa\kappa\alpha,\alpha\iota\biota)
=\dim(\alpha\oplus \alpha\xi\alpha,\alpha\oplus \alpha\xi)=1,$$
$$\dim(\kappa\alpha\iota,\kappa\iota)=\dim(\bkappa\kappa\alpha,\iota\biota)
=\dim(\alpha\oplus \alpha\xi\alpha,\id_P\oplus \xi)=1,$$
which show the claim. 
The claim shows that there exists a unitary $v\in M$ such that 
$v\alpha(x)v^*=x$ holds for every $x\in N$. 
We set $\pi(y)=v\alpha(y)v^*$ for $y\in P$ and set $Q=\pi(P).$ 
By construction $Q$ is an intermediate subfactor between $N$ and $M$ such that 
$\iota_{Q,N}=\pi\iota$, $[\overline{\iota_{Q,N}}\iota_{Q,N}]=[\biota\iota]$, 
and $[M:P]=[M:Q]$. 

Suppose $P=Q$. 
Then $\pi$ would be an automorphism of $P$, which satisfies $[\kappa\pi]=[\kappa\alpha]$. 
Frobenius reciprocity implies that  $\pi\alpha^{-1}$ is contained in 
$\bkappa\kappa=\id_P\oplus \alpha\xi$ and so $[\pi]=[\alpha]$.  
Thus there exists a unitary $u$ in $P$ such that $v\alpha(y)v^*=u\alpha(y)u^*$ for all 
$y\in P$ and so 
$u^*v\in P'\cap M=\C$. 
This shows that $v\in P$ and $[\alpha\iota]=[\iota]$, which is a contradiction.  
Therefore $P\neq Q$. 

Since $N\subset P$ is 3-supertransitive, it has no intermediate subfactor and 
$P\cap Q=N$. 
Since $\dim (\alpha\xi,(\alpha\xi)^2)=1$, the inclusion $P\subset M$ is 3-supertransitive 
too and $M$ is generated by $P$ and $Q$. 
Therefore $\fQ=\quadri$ is an irreducible quadrilateral of factors. 
Theorem \ref{3theorem angle},(1) shows that $\fQ$ is noncommuting. 
Since $P$ and $Q$ are inner conjugate in $M$, Theorem \ref{3theorem angle},(1) applied to 
the dual quadrilateral $\hat{\fQ}$ shows that $\fQ$ is noncocommuting.  
\end{proof}

The above theorem shows that quadrilaterals in Class I is completely determined by 
the subfactor $N\subset P$ and $\alpha$. 

Jones and the first-named author \cite{GJ} showed that there exists a unique quadrilateral 
of the hyperfinite II$_1$ factors such that all the elementary subfactors are the the $A_7$ 
subfactor. 
It is easy to show that this is the only quadrilateral in Class I satisfying $[P:N]<4$. 
Other than this example,  we know two subfactors satisfying the assumption 
of Theorem \ref{5theorem I1}, namely, the $E_7^{(1)}$ subfactor and 
the Haagerup subfactors. 

Let $N\subset P$ be the $E_7^{(1)}$ subfactor. 
Then the dual principal graph is as follows: 
$$\hpic{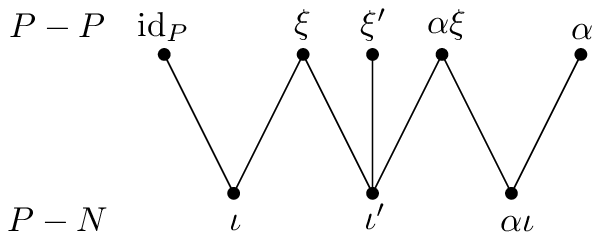}{1.0in}$$
Note that the category of $P-P$ bimodules for the  $E_7^{(1)}$ subfactor is isomorphic 
to the category $\hat{\fA_4}$ of the unitary representations of the alternating group 
$\fA_4$. 
It is observed in \cite[Corollary 4.2]{I8} that there exists an automorphism of 
the category $\hat{\fA_4}$ that flips the two representations corresponding to 
$\xi$ and $\alpha\xi$. 
Thus $\id_P\oplus \alpha\xi$ has a $Q$-system and gives rise to an irreducible 
noncommuting quadrilateral in Class I whose elementary subfactors are the $E_7^{(1)}$ subfactor. 
Note that this quadrilateral satisfies the assumption of Theorem \ref{4theorem uniqueness1} 
and such a quadrilateral is unique. 

Since it is easy to see that no other subfactors of index less than or equal to 4 fit into the statement of Theorem \ref{4theorem (2,2)-supertransitive},(1), 
we get the following: 

\begin{theorem} \label{5theorem I2}
Let $\fQ=\quadri$ be a quadrilateral of factors in Class I such 
that $[M:P]\leq 4$. 
Then one of the following two cases occurs: 
\begin{itemize}
\item [(1)] The principal graphs of all the elementary subfactors are $A_7$. 
\item [(2)] The principal graphs of all the elementary subfactors are $E_7^{(1)}$.  
\end{itemize}
In each case, such a quadrilateral of hyperfinite II$_1$ factors exists and is unique 
up to conjugacy.  
\end{theorem}

Let $N\subset P$ be the Haagerup subfactor \cite{AH}. 
Since the Haagerup subfactor is not self-dual, we specify the dual principal graph 
(the induction-reduction graph of the $P-P$ and $P-N$ bimodules) as below.  
$$\hpic{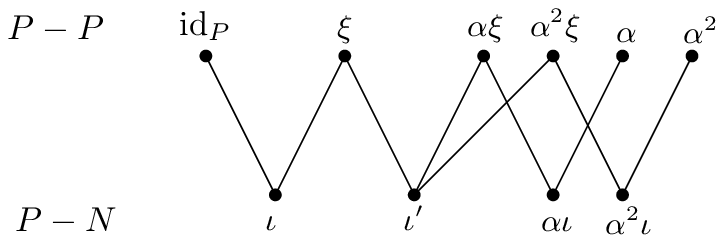}{1.0in}$$
It is known that $\alpha^3=\id_P$ and $\alpha\xi=\xi\alpha^{-1}$ holds. 
Since $\id_P\oplus \xi$ has a $Q$-system, so does 
$\alpha^{-1}(\id_P\oplus \xi)\alpha=\id_P\oplus \alpha\xi$. 
Thus there exists an irreducible noncommuting quadrilateral of factors in Class I  
whose elementary subfactors are the Haagerup subfactor. 
Although the quadrilateral arising from $N\subset P$ does not satisfy 
the assumption of Theorem \ref{4theorem uniqueness1}, we have 
$[\kappa\alpha]=[\kappa_Q\pi]$ and so $[\bpi\eta_Q\pi]
=[\alpha^{-1}\alpha\xi\alpha]=[\alpha^2\xi]$. 
Thus the same proof of Theorem \ref{4theorem uniqueness2},(2) works and 
we can show uniqueness of such a quadrilateral up to flip conjugacy. 

We don't know if there are infinitely many mutually non-conjugate 
quadrilaterals of the hyperfinite II$_1$ factors in Class I. 
\subsection{Class II} 
The following two are the main theorems of this subsection: 

\begin{theorem} \label{5theorem II1} 
Let $\fQ=\quadri$ be a quadrilateral of factors in Class II such that 
$\fQ$ is (5,3)-supertransitive. 
Then 
\begin{itemize}
\item [$(1)$] If $\fQ$ is (6,3)-supertransitive, the principal graphs of $P\subset M$ and 
$Q\subset M$ are  $A_4$ and those of $N\subset P$ and $N\subset Q$ are $D_6$. 
\item [$(2)$] If $\fQ$ is not (6,3)-supertransitive, the principal graphs of 
$P\subset M$ and $Q\subset M$ are $E_8^{(1)}$.  
\end{itemize}
In each case, such a quadrilateral of the hyperfinite II$_1$ factors exists and is unique up to conjugacy. 
\end{theorem}

\begin{theorem} \label{5theorem II2} 
Let $\fQ=\quadri$ be a quadrilateral of factors in Class II such that 
$[M:P]\leq 4$. 
Then one of the following holds:
\begin{itemize}
\item [$(1)$] The principal graphs of $P\subset M$ and $Q\subset M$ are $A_4$ and those of 
$N\subset P$ and $N\subset Q$ are $D_6$. 
Such a quadrilateral of the hyperfinite II$_1$ factors exists and is unique up to conjugacy. 
\item [$(2)$] The principal graphs of $P\subset M$ and $Q\subset M$ are $E_6$ and 
the dual principal graphs of $N\subset P$ and $N\subset Q$ are as below. 
$$\hpic{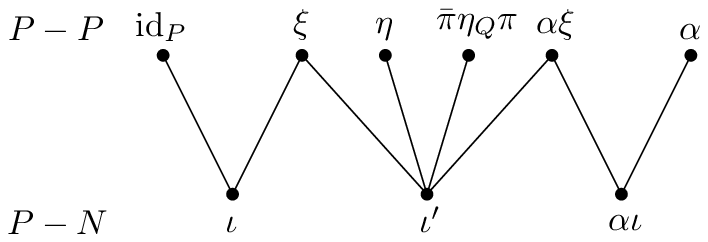}{1.0in}$$
Such a quadrilateral of the hyperfinite II$_1$ factors exists. 
\item [$(3)$] The principal graphs of $P\subset M$ and $Q\subset M$ are $E_8^{(1)}$. 
Such a quadrilateral of the hyperfinite II$_1$ factors exists and is unique up to conjugacy. 
\end{itemize}
\end{theorem}

In the sequel, the dimension of the space 
$(\kappa\iota\biota\bkappa,\kappa\iota\biota\bkappa)$ often gives important information 
about a quadrilateral and we give useful formulae about it first. 

\begin{lemma} \label{5lemma dim} 
Let $\fQ=\quadri$ be an irreducible noncommuting quadrilateral of factors 
such that $\fQ$ is cocommuting and (2,3)-supertransitive. 
Then $\iota'\biota$ contains $\xi$ and $\eta$ with multiplicity one, the equality 
$[\iota'\biota]=[\eta]\oplus [\eta\xi]$ holds, and 
$$ \dim(\kappa\iota\biota\bkappa,\kappa\iota\biota\bkappa)
=5+\dim(\eta\xi,\eta\xi).$$
If moreover $P\subset M$ is 3-supertransitive, then 
$[\eta\xi]=[\xi]\oplus [\bkappa'\kappa_Q\pi]$ and 
$$ \dim(\kappa\iota\biota\bkappa,\kappa\iota\biota\bkappa)
=6+\dim(\bkappa'\kappa_Q\pi,\bkappa'\kappa_Q\pi)
=6+\dim(\kappa'\bkappa',\kappa_Q\bkappa_Q).$$
\end{lemma}

\begin{proof}
By Corollary \ref{4corollary etaiota=iota'},(1) and Frobenius reciprocity, we have 
\begin{eqnarray*}\dim(\kappa\iota\biota\bkappa,\kappa\iota\biota\bkappa)
&=&\dim (\biota\bkappa\kappa\iota,\biota\bkappa\kappa\iota)
=\dim(\id_N\oplus \hxi\oplus \biota\iota',\id_N\oplus \hxi\oplus \biota\iota')\\
&=&2+2\dim(\iota\hxi,\iota')+\dim(\biota\iota',\biota\iota')=
4+\dim(\biota\iota',\biota\iota').
\end{eqnarray*}
Since $[\biota\iota\biota\iota]=2[\biota\iota]\oplus[\biota\iota']$, the sector 
$[\biota\iota']$ is self-conjugate, and so is $[\iota'\biota]$ for a similar reason. 
Thus we get 
$$\dim(\biota\iota',\biota\iota')=\dim(\biota\iota',\biota'\iota)=
\dim(\iota'\biota,\iota\biota')=\dim(\iota'\biota,\iota'\biota).$$
The statement follows from Corollary \ref{4corollary etaiota=iota'},(1) now. 
\end{proof}

In view of Corollary \ref{4corollary etaiota=iota'},(4), a canonical candidate 
of a quadrilateral in Class II with the smallest index is that with the $D_6$ subfactor 
for $N\subset P$. 

\begin{proposition} Let $\fQ=\quadri$ be a quadrilateral in Class II or Class III 
such that the principal graph of $P\subset M$ is $A_4$. 
Then $\fQ$ is in Class II and the principal graph of $N\subset P$ is $D_6$. 
Such a quadrilateral of the hyperfinite II$_1$ factors exists and is 
unique up to conjugacy. 
\end{proposition}

\begin{proof} 
From Theorem \ref{4theorem (2,2)-supertransitive}, we see 
$[P:N]=[M:P]+1=(5+\sqrt{5})/2=4\cos^2(\pi/10)$, which shows that the principal graph of 
$N\subset P$ is either $A_9$ or $D_6$.  
Thanks to Theorem \ref{2theorem no-extra}, the former never occurs. 
If $\fQ$ were in Class III, Corollary \ref{4corollary etaiota=iota'},(3) would imply that 
$\heta\theta\heta$ does not contain any automorphism and it contains $\theta\heta\theta$. 
However, this contradicts $d(\heta\theta\heta)-d(\theta\heta\theta)=d(\heta)^2-d(\heta)=1$ and $\fQ$ is 
in Class II. 

In the case of the hyperfinite II$_1$ factors, uniqueness up to flip conjugacy follows from 
Theorem \ref{4theorem uniqueness2},(2) and the uniqueness of the $D_6$ subfactor. 
Let $[\xi^2]=[\id_P]\oplus [\xi]\oplus [\xi']\oplus[\xi'']$ be the irreducible decomposition 
of $\xi^2$. 
To prove uniqueness up to conjugacy, it suffices to show that there exists an automorphism 
$\alpha$ of $P$ such that $\alpha(N)=N$ and $[\alpha\xi'\alpha^{-1}]=[\xi'']$ 
(see the proof of Theorem \ref{4theorem uniqueness2},(2)).   
It is shown in \cite{Ka} that there exists an automorphism $\beta$ of period two 
on the $A_9$ subfactor $\tN\subset \tP$ such that 
$$(N\subset P)\cong (\tN\rtimes_\beta\Z/2\Z\subset \tP\rtimes_\beta \Z/2\Z).$$
The dual action $\hat{\beta}$ of $\beta$ acts on the higher relative commutants of 
$N\subset P$ non-trivially. 
It is routine work to show that $\hat{\beta}$ does the right job. 

Finally we show existence. 
Let $\cN\subset \cM$ be an inclusion of factors whose principal graph is $D_6$. 
We use the following parameterization of sectors associated with $\cN\subset \cM$. 
$$\hpic{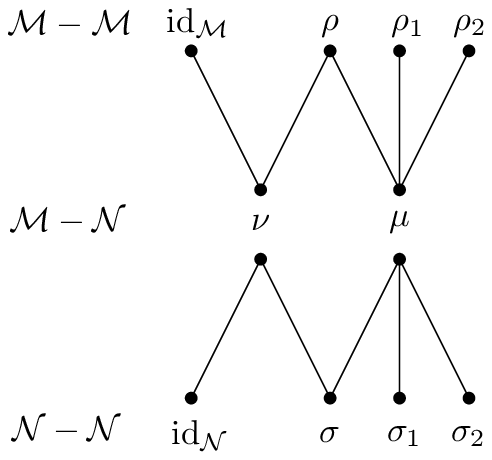}{2.0in}$$
Recall the fusion rule for the $D_6$ subfactors \cite[Proposition 3.10]{I1}: 
$$[\rho^2]=[\id_\cM]\oplus [\rho]\oplus [\rho_1]\oplus [\rho_2],\quad [\rho_1\rho_2]=[\rho_2\rho_1]=[\rho],$$
$$[\rho_1^2]=[\id_\cM]\oplus [\rho_1],\quad [\rho_2^2]=[\id_\cM]\oplus [\rho_2], $$
$$[\mu]=[\rho_1\nu]=[\rho_2\nu]=[\nu\sigma_1]=[\nu\sigma_2],$$
where $\nu=\iota_{\cM,\cN}$. 
We can take representatives $\rho_1$ and $\rho_2$ satisfying 
$\mu=\rho_1\nu=\rho_2\nu$. 
We set $M=\cM$, $P=\rho_1(\cM)$, $Q=\rho_2(\cM)$, and $N=\mu(\cN)=\rho_1(\cN)=\rho_2(\cN)$. 
Then $d(\rho_1)=d(\rho_2)=(1+\sqrt{5})/2$, which means $[M:P]=[M:Q]=4\cos^2\pi/5$, and so 
the principal graphs of $P\subset M$ and $Q\subset M$ are $A_4$. 
By construction, the principal graphs of $N\subset P$ and $N\subset Q$ are $D_6$. 
Since $\mu$ is irreducible, the inclusion $N\subset M$ is irreducible and 
$\fQ=\quadri$ is an irreducible quadrilateral. 
Since $[\rho_1^2]=[\id_\cM]\oplus [\rho_1]$ and $[\rho_2^2]=[\id_\cM]\oplus [\rho_2]$, we have 
$[\iota_{M,P}\overline{\iota_{M,P}}]=[\id_M]\oplus [\rho_1]$ and 
$[\iota_{M,Q}\overline{\iota_{M,Q}}]=[\id_M]\oplus [\rho_2]$, which implies that 
$\fQ$ is cocommuting. 
Therefore $[M:P]\neq [P:N]$ shows that the quadrilateral $\fQ$ is not commuting.  
\end{proof}

\begin{remark} 
The above construction shows that the principal graph and the dual principal graph of $N\subset M$ 
are the same, which can be easily computed from 
$$[\mu\overline{\mu}]=[\rho_1\nu\bar{\nu}\rho_1]
=[\rho_1(\id_M\oplus\rho)\rho_1]=[\id_M]\oplus [\rho_1]\oplus[\rho_2]\oplus 2[\rho],$$
$$[\rho_1\mu]=[\rho_1^2\nu]=[\nu]\oplus [\rho_1\nu]=[\nu]\oplus [\mu],$$
$$[\rho_2\mu]=[\rho_2^2\nu]=[\nu]\oplus [\rho_2\nu]=[\nu]\oplus [\mu],$$
$$[\rho\mu]=[\rho\rho_1\nu]=[\rho\nu]\oplus[\rho_2\nu]=[\nu]\oplus 2[\mu],$$
$$[\nu\bar{\mu}]=[\nu\bar{\nu}\rho_1]=[\rho_1]\oplus [\rho\rho_1]=[\rho_1]\oplus [\rho_2]\oplus [\rho].$$
By symmetry of $\mu$ and $\bar{\mu}$, there exist two more intermediate subfactors 
$L$ and $R$ between $N$ and $M$ such that the principal graphs of $L\subset M$ and 
$R\subset M$ are $D_6$ and those of $N\subset L$ and $N\subset R$ are $A_4$. 
Since $\iota_{M,L}\overline{\iota_{M,L}}$ and $\iota_{M,R}\overline{\iota_{M,R}}$ are contained 
in $\mu\bar{\mu}$, we see that they are equivalent to $\id_M\oplus \rho$. 
There is no intermediate subfactor other than $P,Q,L,R$. 
Indeed, if $S$ is an intermediate subfactor, the endomorphism $\iota_{M,S}\overline{\iota_{M,S}}$ 
is contained in $\mu\bar{\mu}$. 
Computation of the indices $[M:S]$ and $[S:N]$ shows that all the possibilities of  
$\iota_{M,S}\overline{\iota_{M,S}}$ are exhausted by $P,Q,L,R$. 
Thus Lemma \ref{3lemma uniqueness} applied to the dual quadrilateral implies the claim.  
\end{remark}

\begin{proposition} Let $\fQ=\quadri$ be a quadrilateral in Class II or Class III 
such that the principal graph of $P\subset M$ is $A_5$. 
Then $\fQ$ is in Class III and the principal graph of $N\subset P$ is $E_7^{(1)}$. 
When $M$ is the hyperfinite II$_1$ factor, such a quadrilateral exists and is unique 
up to conjugacy. 
\end{proposition}

\begin{proof} 
Since the principal graph of $P\subset M$ is $A_5$, 
there exists an outer automorphism of period two $\alpha\in \Aut(P)$ such that 
$[\kappa\eta]=[\kappa]\oplus [\kappa\alpha]$. 
Since $[\xi]=[\bkappa\kappa_Q\pi]$, we get 
$$[\eta\xi]=[\eta\bkappa\kappa_Q\pi]=[\bkappa\kappa_Q\pi]\oplus [\alpha\bkappa\kappa_Q\pi]
=[\xi]\oplus [\alpha\xi].$$
Therefore Corollary \ref{4corollary etaiota=iota'},(1) shows that the dual principal 
graph of $N\subset P$ is 
$$\hpic{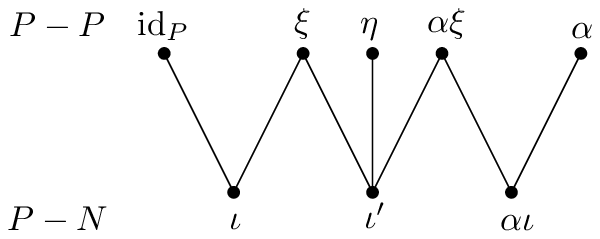}{1.0in}$$
which is $E_7^{(1)}$. 
Existence of such a quadrilateral follows from Example \ref{2example group} with 
$G=\fS_{\{1,2,3,4\}}$ and uniqueness follows from Theorem \ref{4theorem uniqueness2}. 
Corollary \ref{4corollary etaiota=iota'},(4) shows that $\fQ$ is in Class III.
\end{proof}

\begin{proposition} Let $\fQ=\quadri$ be a quadrilateral in Class II or Class III 
such that the principal graph of $P\subset M$ is $E_6$. 
Then $\fQ$ is in Class II and the dual principal graph of $N\subset P$ is 
as in Theorem \ref{5theorem II2}
and $[P:N]=3+\sqrt{3}$. 
There exists a quadrilateral of the hyperfinite II$_1$ factors satisfying the above property, 
which arises from the GHJ pair \cite{GJ} (see Example \ref{7example E6}) for 
$E_6$ with $*$ given by the vertex with the smallest entry of 
the Perron-Frobenius eigenvector.  
\end{proposition}

\begin{proof} From Theorem \ref{4theorem (2,2)-supertransitive}, we see 
$[P:N]=[M:P]+1=3+\sqrt{3}$. 
We recall the fusion rule for the $E_6$ subfactors \cite[p.968]{I1}: 
$$[\eta^2]=[\id_P]\oplus [\alpha]\oplus 2[\eta],\quad [\alpha\eta]=[\eta\alpha]=[\eta],
\quad [\alpha^2]=[\id_P],$$  
$$[\kappa\eta]=[\kappa]\oplus [\kappa']\oplus [\kappa\alpha],\quad [\bkappa\kappa']=[\eta],$$
$$[\heta^2]=[\id_M]\oplus [\beta]\oplus 2[\heta],\quad 
[\heta\kappa]=[\kappa]\oplus [\kappa']\oplus [\beta\kappa],$$
with $[\beta\kappa]=[\kappa\alpha]$, 
$d(\kappa)=\sqrt{2+\sqrt{3}}$, $d(\kappa')=\sqrt{2}$, $d(\alpha)=d(\beta)=1$, 
$d(\eta)=d(\heta)=1+\sqrt{3}$, $d(\xi)=2+\sqrt{3}$.

From $[\xi]=[\bkappa\kappa_Q\pi]$, we get 
$[\eta\xi]=[\eta\bkappa\kappa_Q\pi]=[\xi]\oplus[\bkappa'\kappa_Q\pi]\oplus [\alpha\xi].$ 
Since 
$$\dim(\bkappa'\kappa_Q\pi,\bkappa'\kappa_Q\pi)=\dim(\kappa'\bkappa',\kappa_Q\bkappa_Q)
=\dim(\id_M\oplus \beta,\id_M\oplus \heta_Q)=1,$$
the endomorphism $\bkappa'\kappa_Q\pi$ is irreducible, and 
Lemma \ref{5lemma dim} implies  
$\dim(\kappa\iota\biota\bkappa,\kappa\iota\biota\bkappa)=8.$

Suppose that $\fQ$ is in Class III. 
Then thanks to Corollary \ref{4corollary etaiota=iota'},(3), we have 
$$[\kappa\iota\biota\bkappa]=[\id_M]\oplus [\theta]\oplus [\heta]\oplus[\theta\heta]\oplus 
[\heta\theta]\oplus [\heta\theta\heta],$$
and $\heta\theta\heta$ does not contain any other irreducible components above.  
Therefore we must have $\dim(\heta\theta\heta,\heta\theta\heta)=3$. 
However, 
\begin{eqnarray*}
\dim(\heta\theta\heta,\heta\theta\heta)&=&
\dim(\theta\heta^2\theta,\heta^2)=
\dim(\id_M\oplus \theta\beta\theta\oplus 2\theta\heta\theta, 
\id_M\oplus\beta\oplus 2\heta)\\
&=&1+\dim(\theta\beta\theta,\beta),
\end{eqnarray*}
which is a contradiction. 
Therefore $\fQ$ is in Class II. 
By Corollary \ref{4corollary etaiota=iota'},(4), we get 
$[\bpi\eta_Q\pi]=[\bkappa'\kappa_Q\pi]$ and the principal graph of $N\subset P$ 
is as stated above.  
\end{proof}

We exclude $A_7$, $D_6$, and $E_8$ for the principal graph of $P\subset M$. 

\begin{lemma}
There exists no quadrilateral $\fQ=\quadri$ in Class II or Class III  such that the principal 
graph of $P\subset M$ is $A_7$. 
\end{lemma}

\begin{proof} Suppose there exists such a quadrilateral $\fQ$. 
Since $P\subset M$ is the $A_7$ subfactor, there exist outer automorphisms $\alpha$ 
of $P$ and $\beta$ of $M$ and such that 
$$[\heta^2]=[\id_M]\oplus [\heta]\oplus [\beta\heta],$$
$$[\bkappa'\kappa]=[\eta]\oplus [\alpha\eta].$$
Lemma \ref{5lemma dim} shows  $[\eta\xi]=[\xi]\oplus [\bkappa'\kappa_Q\pi].$ 

Suppose that $\bkappa'\kappa_Q\pi$ is irreducible. 
Then Lemma \ref{5lemma dim} shows that 
$\dim(\kappa\iota\bkappa\biota,\kappa\iota\bkappa\biota)=7$. 
Since $\kappa'\bkappa$ is reducible, 
Corollary \ref{4corollary etaiota=iota'},(3),(5) shows that the quadrilateral 
$\fQ$ should be in Class III and $\heta\theta\heta$ is decomposed into two mutually 
inequivalent irreducibles, which implies 
$$2=\dim(\heta\theta\heta,\heta\theta\heta)=\dim(\theta\heta^2\theta,\heta^2)
=(\id_M\oplus\theta\heta\theta\oplus \theta\beta\heta\theta,
\id_M\oplus \heta\oplus \beta\heta).$$
Thus either $[\heta]=[\theta\beta\heta\theta]$, $[\beta\heta]=[\theta\heta\theta]$, 
or $[\beta\heta]=[\theta\beta\heta\theta]$ should hold. 
However, if $[\heta]=[\theta\beta\heta\theta]$, we would get 
$[\heta^2]=[\overline{\heta}\heta]=[\theta\heta^2\theta],$
which is a contradiction. 
The same reasoning shows that the other two do not occur as well and we conclude that 
$\bkappa'\kappa_Q\pi$ is reducible. 

Suppose that $\bkappa'\kappa_Q\pi$ is reducible. 
Since any irreducible component $\zeta$ of $\bkappa'\kappa_Q\pi$ is contained in 
$\iota'\biota$, Frobenius reciprocity shows that $\zeta\iota$ contains $\iota'$. 
On the other hand, 
$$[\bkappa'\kappa_Q\pi\iota]=[\bkappa'\kappa_Q\iota_Q]=[\bkappa'\kappa\iota]
=[\eta\iota]\oplus[\alpha\eta\iota]=[\iota']\oplus[\alpha\iota'].$$
This implies $[\iota']=[\alpha\iota']$ and that $\bkappa'\kappa_Q\pi$ is decomposed 
into two irreducibles, say, $\zeta_1$ and $\zeta_2$ such that 
$[\zeta_1\iota]=[\zeta_2\iota]=[\iota']$. 
This shows $[\iota'\biota]=[\xi]\oplus [\eta]\oplus [\zeta_1]\oplus[\zeta_2]$ 
and the set of $P-P$ sectors appearing the dual principal graph of $N\subset P$ is 
$\{[\id_P],[\xi],[\eta],[\zeta_1],[\zeta_2]\}$. 
However, since $[\alpha\iota']=[\iota']$, the automorphism $\alpha$ appears in 
$\iota'\biota'$, which is a contradiction. 
Therefore the statement is shown.   
\end{proof}

\begin{lemma}
There exists no quadrilateral $\fQ=\quadri$ in Class II or Class III such that 
the principal graph of $P\subset M$ is $D_6$. 
\end{lemma}

\begin{proof} 
Suppose that $\fQ$ is such a quadrilateral. 
We first recall the fusion rule for the $D_6$ subfactor: 
$$[\kappa'\bkappa]=[\heta]\oplus [\heta_1]\oplus [\heta_2],\quad 
[\kappa'\bkappa']=[\id_M]\oplus 2[\heta]\oplus [\heta_1]\oplus [\heta_2],$$
$$[\heta^2]=[\id_M]\oplus[\heta]\oplus [\heta_1]\oplus [\heta_2],$$
$$[\heta_1^2]=[\id_M]\oplus [\heta_1],\quad [\heta_2]^2=[\id_M]\oplus [\heta_2],\quad 
[\heta_1\heta_2]=[\heta_2\heta_1]=[\heta],$$
$$[\heta_1\heta]=[\heta\heta_1]=[\heta]\oplus [\heta_2], \quad 
[\heta_2\heta]=[\heta\heta_2]=[\heta]\oplus [\heta_1],$$
$d(\heta)=(3+\sqrt{5})/2$, $d(\heta_1)=d(\heta_2)=(1+\sqrt{5})/2$. 

Lemma \ref{5lemma dim} shows $[\eta\xi]=[\xi]\oplus [\bkappa'\kappa_Q\pi].$
We claim that $\bkappa'\kappa_Q\pi$ is irreducible. 
Indeed, 
$$\dim(\bkappa'\kappa_Q\pi,\bkappa'\kappa_Q\pi)=\dim(\kappa'\bkappa',\kappa_Q\bkappa_Q)
=\dim(\id_M\oplus 2\heta\oplus \heta_1\oplus \heta_2,\id_M\oplus\heta_Q).$$
Since $\fQ$ is cocommuting, we have $[\heta]\neq [\heta_Q]$ and we get the claim. 
Lemma \ref{5lemma dim} and the claim implies that 
$\dim(\kappa\iota\biota\bkappa,\kappa\iota\biota\bkappa)=7$.  
Corollary \ref{4corollary etaiota=iota'},(3),(5) shows that $\fQ$ is in Class III and 
$\dim(\heta\theta\heta,\heta\theta\heta)=2$, and so 
\begin{eqnarray*}
2&=&\dim(\theta\heta^2\theta,\heta^2)
=\dim(\id_M\oplus \theta\heta\theta\oplus \theta\heta_1\theta\oplus \theta\heta_2\theta, 
\id_M\oplus \heta\oplus \heta_1\oplus \heta_2)\\
&=&1+\dim(\theta\heta_1\theta\oplus \theta\heta_2\theta,\heta_1\oplus \heta_2).
\end{eqnarray*}
If $[\theta\heta_1\theta]=[\heta_2]$, we would have $[\theta\heta_2\theta]=[\heta_1]$ too 
as the period of $\theta$ is two, which is a contradiction. 
Thus we may assume $[\theta\heta_1\theta]=[\heta_1]$ and 
$[\theta\heta_2\theta]\neq [\heta_1],[\heta_2]$. 

Since $\heta\theta\heta$ contains $\theta\heta\theta$ with multiplicity one, 
there exists an irreducible $\zeta\in \End_0(M)$ inequivalent to $\theta\heta\theta$ 
such that $[\heta\theta\heta]=[\theta\heta\theta]\oplus [\zeta]$. 
We claim that $[\zeta]=[\heta\theta\heta_2]$ holds. 
Indeed, we have 
$$\dim(\heta\theta\heta_2,\heta\theta\heta_2)=\dim(\heta^2,\theta\heta_2^2\theta)
=\dim(\id_M\oplus \heta\oplus\heta_1\oplus\heta_2,\id_M\oplus \theta\heta_2\theta)=1,$$
which shows that $\heta\theta\heta_2$ is irreducible. 
Since $[\theta\heta\theta]\neq [\heta\theta\heta_2]$ and 
$$\dim (\heta\theta\heta,\heta\theta\heta_2)=\dim(\heta^2,\theta\heta_2\heta\theta)
=\dim(\id_M\oplus\heta\oplus\heta_1\oplus\heta_2,
\theta\heta_1\theta\oplus \theta\heta\theta)=1,$$
we get the claim. 
Since 
$$[\kappa\iota\biota\bkappa]=[\id_M]\oplus [\theta]\oplus[\heta]\oplus [\theta\heta]
\oplus [\heta\theta]\oplus [\theta\heta\theta]\oplus [\zeta],$$
and $[\theta\kappa\iota]=[\kappa\iota]$, we get $[\theta\zeta]=[\zeta]$. 
However, 
$$\dim(\heta\theta\heta_2,\theta\heta\theta\heta_2)
=\dim(\heta\theta\heta,\theta\heta_2^2\theta)
=\dim(\theta\heta\theta\oplus \zeta,\id_M\oplus \theta\heta_2\theta)=0,$$
which is a contradiction. 
Thus the statement is proven. 
\end{proof}

\begin{lemma}
There exists no quadrilateral $\fQ=\quadri$ in Class II or Class III such that the principal graph of 
$P\subset M$ is $E_8$. 
\end{lemma}

\begin{proof} 
Suppose that $\fQ$ is such a quadrilateral. 
Since the $E_8$ factor is 4-supertransitive, there exists an irreducible 
$\heta'\in \End_0(M)$ such that $[\heta^2]=[\id_M]\oplus [\heta]\oplus [\heta']$. 
Note that we have $d(\heta)\neq d(\heta')$ and  
$\heta$ is the only irreducible sector $\sigma$ contained in 
$\heta'^2$ (also in $\kappa'\bkappa'$) such that $d(\heta)=d(\sigma)$ 
(see \cite[Section 3.3]{I1}). 

As before, we have $[\xi\eta]=[\xi]\oplus [\bkappa'\kappa_Q\pi]$ and 
$$\dim(\bkappa'\kappa_Q\pi,\bkappa'\kappa_Q\pi)=
\dim(\kappa'\bkappa',\kappa_Q\bkappa_Q)=1+\dim(\kappa'\bkappa',\heta_Q)=1.$$
Thus $\fQ$ is in Class III and 
$$2=\dim(\heta\theta\heta,\heta\theta\heta)=\dim(\theta\heta^2\theta,\heta^2)=\dim(\id_M\oplus 
\theta\heta\theta\oplus\theta\heta'\theta,\id_M\oplus \heta\oplus\heta'),$$
which shows $[\theta\heta'\theta]=[\heta']$. 
However, this implies $[\theta\heta'^2\theta]=[\heta'^2]$, 
which contradicts $[\theta\heta\theta]\neq [\heta]$ as $\heta'^2$ contains $\heta$. 
\end{proof}

We are ready to prove Theorem \ref{5theorem II1} now. 

\begin{proof}[Proof of Theorem \ref{5theorem II1}] Let $\fQ$ be in Class II and 
(5,3)-supertransitive such that the principal graph of $P\subset M$ is not $A_4$. 
Corollary \ref{4corollary etaiota=iota'},(5) shows 
$\dim(\kappa\iota\biota\bkappa,\kappa\iota\biota\bkappa)\geq 8$ and Lemma \ref{5lemma dim} 
shows $\dim(\bkappa'\kappa_Q\pi,\bkappa'\kappa_Q\pi)\geq 2$. 
Since the principal graph of  $P\subset M$ is not $A_5$, 
the depth of $P\subset M$ is at least 5. 
Thus we use the following parameterization of $M-M$, $M-P$, and $P-P$ sectors 
arising from $P\subset M$. 
$$\hpic{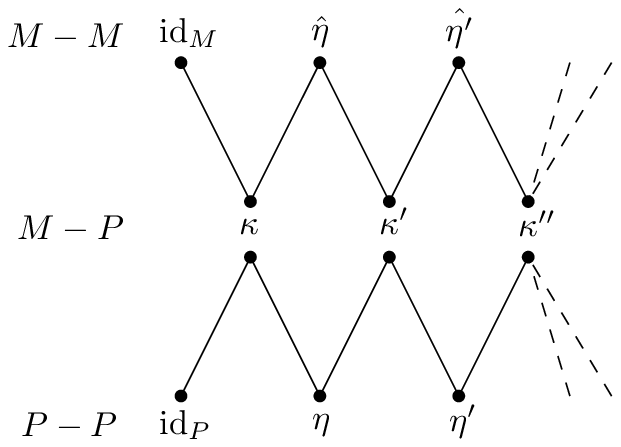}{2.0in}$$
Corollary \ref{4corollary etaiota=iota'},(5) shows 
$$[\kappa\iota\biota\bkappa]=[\id_M]\oplus [\heta]\oplus [\heta_Q]\oplus [\heta']\oplus 
2[\kappa_Q\pi\bkappa].$$
Since $\fQ$ is cocommuting, $[\heta]\neq [\heta_Q]$. 
Since the principal graph of $P\subset M$ is neither $A_5$ nor $A_7$, we have 
$d(\heta_Q)\neq d(\heta')$, $d(\kappa_Q\pi)\neq d(\kappa')$ 
and $d(\kappa_Q\pi)\neq d(\kappa'')$. 
Since 
$$\dim(\kappa_Q\pi\bkappa,\heta')=\dim(\kappa_Q\pi,\heta'\kappa)
=\dim(\kappa_Q\pi,\kappa'\oplus\kappa'')=0,$$
we get $[\kappa_Q\pi\bkappa]\neq [\heta']$. 
Thus $\dim(\kappa\iota\biota\bkappa,\kappa\iota\biota\bkappa)=8$ and 
$\dim(\bkappa'\kappa_Q\pi,\bkappa'\kappa_Q\pi)=2$. 
Thanks to Corollary \ref{4corollary etaiota=iota'},(4),(5), the endomorphism 
$\bkappa'\kappa_Q\pi$ is decomposed into $\bpi\eta_Q\pi$ and an irreducible endomorphism 
of $P$, say $\xi'$. 
Since 
$$[\bkappa'\kappa_Q\pi\iota]=[\bkappa'\kappa\iota]=[\eta\iota]\oplus[\eta'\iota]
=[\iota']\oplus [\eta'\iota],$$
and $[\bpi\eta_Q\pi\iota]=[\iota']$, we get $[\xi'\iota]=[\eta'\iota]$, 
which contains $\iota'$ as $\xi'$ is contained in $\iota'\biota$. 
Thus Frobenius reciprocity implies 
$$1=\dim(\iota',\eta'\iota)=\dim(\iota'\biota,\eta')=\dim(\xi\oplus\eta\oplus 
\bpi\eta_Q\pi\oplus \xi',\eta').$$
If $[\xi]=[\eta']$, we would have $[\xi'\iota]=[\iota]\oplus [\iota']$ and 
by Frobenius reciprocity $[\xi']=[\xi]$, a contradiction. 
Thus $[\xi']=[\eta']$.  
$$\hpic{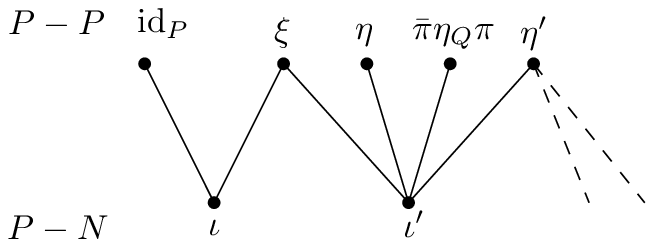}{1.0in}$$

We claim $[\eta\bar{\pi}\eta_Q\pi]=[\xi]\oplus [\eta']$. 
Thanks to Corollary \ref{4corollary etaiota=iota'}(5), we know that 
$\kappa_Q\pi\bkappa$ is self-conjugate. 
Thus 
\begin{eqnarray*}[\xi^2]&=&
[\bkappa\kappa_Q\pi\bkappa\kappa_Q\pi]=[\bkappa\kappa\bpi\bkappa_Q\kappa_Q\pi]\\
&=&[(\id_P\oplus \eta)(\id_P\oplus \bar{\pi}\eta_Q\pi)]=[\id_P]\oplus [\eta]\oplus 
[\bar{\pi}\eta_Q\pi]\oplus [\eta\bar{\pi}\eta_Q\pi],
\end{eqnarray*}
which shows the claim. 

Let 
$[\eta'\iota]=[\iota']\oplus \bigoplus_{i=1}^km_i[\iota_i],$ and 
$[\iota_i\biota]=m_i[\eta']\oplus\bigoplus_{j=1}^{l} m_{ij}[\xi_j]$ 
be the irreducible decompositions. 
We compute $[\eta\xi]$ in two ways:
\begin{eqnarray*}
[\eta\xi][\xi]&=&[\xi^2]\oplus [\bar{\pi}\eta_Q\pi\xi]\oplus [\eta'\xi]\\
&=&[\id_P]\oplus [\xi]\oplus [\eta]\oplus [\bar{\pi}\eta_Q\pi]\oplus [\eta'] 
\oplus [\xi]\oplus [\eta]\oplus [\eta']\\
&\oplus& [\xi]\oplus [\eta]\oplus[\bar{\pi}\eta_Q\pi]\oplus \sum_{i=1}^km_i^2[\eta']
\oplus \bigoplus_{j=1}^l\sum_{i=1}^km_im_{ij}[\xi_j]\\
&=&[\id_P]\oplus 3[\xi]\oplus 3[\eta]\oplus 2[\bar{\pi}\eta_Q\pi] \oplus 
(2+ \sum_{i=1}^km_i^2)[\eta']\oplus \bigoplus_{j=1}^l\sum_{i=1}^km_im_{ij}[\xi_j],
\end{eqnarray*}
\begin{eqnarray*}
[\eta][\xi^2] &=&[\eta]\oplus [\eta\xi]\oplus [\eta^2]\oplus[\eta\bar{\pi}\eta_Q\pi]\oplus 
[\eta\eta'] \\
 &=& [\eta]\oplus [\xi]\oplus [\bar{\pi}\eta_Q\pi]\oplus [\eta']
 \oplus [\id_P]\oplus [\eta]\oplus [\eta']
 \oplus [\xi]\oplus [\eta']\oplus [\eta\eta']\\
&=&[\id_P]\oplus 2[\xi]\oplus 2[\eta]\oplus [\bar{\pi}\eta_Q\pi]\oplus 3[\eta']
\oplus [\eta\eta'].
\end{eqnarray*}
This shows 
$$[\eta\eta']=[\xi]\oplus [\eta]\oplus [\bar{\pi}\eta_Q\pi] \oplus 
(\sum_{i=1}^km_i^2-1)[\eta']\oplus \bigoplus_{j=1}^l\sum_{i=1}^km_im_{ij}[\xi_j].$$
Since $P\subset M$ is 5-supertransitive, the multiplicity of $\eta'$ in $\eta\eta'$ 
is 1, which implies $k=2$ and $m_1=m_2=1$. 
This shows 
$$[\bkappa\kappa'']=[\eta']\oplus [\bar{\pi}\eta_Q\pi]\oplus [\xi]\oplus 
\bigoplus_{j=1}^l(m_{1j}+m_{2j})[\xi_j]$$
and $P\subset M$ cannot be 6-supertransitive. 
Frobenius reciprocity implies that $\kappa\bar{\pi}\eta_Q\pi$ and $\kappa\xi$ 
contain $\kappa''$ with multiplicity one and  
$$[\kappa\xi]=[\kappa\bkappa\kappa_Q\pi]
=[\kappa_Q\pi]\oplus[\heta\kappa_Q\pi],$$ 
On the other hand, the endomorphisms $\kappa\bpi\eta_Q\pi$ and $\heta\kappa_Q\pi$ are irreducible as 
we have 
\begin{eqnarray*}\
\dim(\kappa\bpi\eta_Q\pi,\kappa\bpi\eta_Q\pi)&=&\dim(\bkappa\kappa\bpi\eta_Q\pi,\bpi\eta_Q\pi)
=\dim(\bpi\eta_Q\pi\oplus\eta\bpi\eta_Q\pi ,\bpi\eta_Q\pi)\\
&=&\dim(\bpi\eta_Q\pi\oplus \xi\oplus \eta',\bpi\eta_Q\pi)=1,
\end{eqnarray*}
$$\dim(\heta\kappa_Q\pi,\heta\kappa_Q\pi)=\dim(\heta^2,\kappa_Q\bkappa_Q)=
\dim (\id_M\oplus \heta\oplus\heta',\id_M\oplus\heta_Q)=1,$$
and so $[\kappa'']=[\kappa\bpi\heta_Q\pi]=[\heta\kappa_Q\pi]$. 
By induction, we can get $d(\eta)=d(\kappa)^2-1$, $d(\kappa')=d(\kappa)(d(\kappa)^2-2)$, 
$d(\eta')=d(\kappa)^4-3d(\kappa^2)+1$, $d(\kappa'')=d(\kappa)(d(\kappa)^4-4d(\kappa)^2+3)$. 
Therefore $d(\kappa'')=d(\kappa\bpi\heta_Q\pi)$ implies $d(\kappa)^4-5d(\kappa)^2+4=0$ and 
$[M:P]=d(\kappa^2)=4.$
This shows that the principal graph of $P\subset M$ is $E_8^{(1)}$. 
$$\hpic{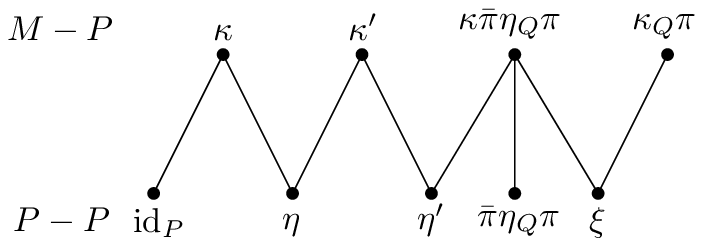}{1.0in}$$

The above computation shows that $\xi$ is determined solely by the inclusion 
$P\subset M$. 
Since $N\subset P$ is 3-supertransitive, a $Q$-system for $\id_P\oplus \xi$ is unique 
up to equivalence and so $M$ is uniquely determined by $P\subset M$ up to inner conjugacy 
in $P$. 
Therefore uniqueness of such a quadrilateral of the hyperfinite II$_1$ factors up to 
conjugacy follows from Theorem \ref{3lemma uniqueness}. 

We show existence now. 
Let $R$ be a factor and $\alpha$ be an outer action of the alternating group $\fA_5$. 
We regard $\fA_4$ as a subgroup of $\fA_5$ and fix it. 
We choose two mutually inequivalent 2-dimensional irreducible projective unitary 
representations $\sigma_1$ and $\sigma_2$. 
Since $\sigma_1$ and $\sigma_2$ carry a unique non-trivial element in $H^2(\fA_5,\T)$, 
we may assume that they have the same cocycle $\omega$. 
Since $\fA_4$ has a unique 2-dimensional irreducible projective unitary representation 
up to equivalence, we may assume that the restriction of $\sigma_1$ and that of $\sigma_2$ to 
$\fA_4$ coincide. 
Let $M_2(\C)$ be the 2 by 2 matrix algebra and 
let $\beta_g=\alpha_g\otimes \Ad \sigma_1(g)$ for $g\in \fA_5$, which is an action 
of $\fA_5$ on $R\otimes M_2(\C)$. 
We set 
$$N=(R\otimes \C)\rtimes_\beta\fA_4\subset P=(R\otimes \C)\rtimes_\beta \fA_5
\subset M=(R\otimes M_2(\C))\rtimes_\beta \fA_5.$$ 
Then the principal graph of $P\subset M$ is $E_8^{(1)}$ \cite[p.224]{GHJ}, 
the inclusion $N\subset P$ is 3-supertransitive \cite{KoY}, and it is routine work 
to show that $N\subset M$ is irreducible. 
Let $u_g=1\otimes \sigma_2(g)\sigma_1(g)^*$, which is a $\beta$-cocycle, and 
let $\{\lambda_g\}_{g\in \fA_5}$ be the implementing unitary in 
$M=(R\otimes M_2(\C))\rtimes_\beta \fA_5$. 
Since $u_g$ commutes with $R\otimes \C$, we can introduce a homomorphism 
$\pi:P\rightarrow M$ by setting $\pi(x)=x$ for $x\in R\otimes \C$ and 
$\pi(\lambda_g)=u_g\lambda_g$ for $g\in \fA_5$. 
We set $Q=\pi(P)$. 
The restriction of $\pi$ to $N$ is the identity map because $u_g=1$ for $g\in \fA_4$. 
Since $\sigma_1$ and $\sigma_2$ are not equivalent, there exists 
$g\in \fA_5\setminus \fA_4$ such that $u_g$ is not a scalar, and so $P\neq Q$. 
Since $\fA_4$ is maximal in $\fA_5$ there is no intermediate subfactor between $N\subset P$. 
Since the principal graph of $P\subset M$ is $E_8^{(1)}$, there is no intermediate subfactor between
$P\subset M$. 
Therefore we conclude that $\fQ=\quadri$ is a quadrilateral, which is noncommuting thanks 
to Theorem \ref{3theorem angle},(1). 
Theorem \ref{4theorem (2,2)-supertransitive} shows that $\fQ$ is cocommuting as we have 
$[M:P]=[P:N]-1$. 

To finish the proof, we show that the principal graph of $Q\subset M$ is also $E_8^{(1)}$. 
Since $(N\subset Q)\cong (R\rtimes_\alpha\fA_4\subset R\rtimes_\alpha \fA_5)$, 
the category of $Q-Q$ bimodules arising from $N\subset Q$ is equivalent to the unitary 
dual of $\fA_5$. 
Lemma \ref{4lemma xi-eta} shows that this category contains the category of 
$Q-Q$ bimodules arising from $Q\subset M$. 
Since $[M:Q]=4$, we conclude that the principal graph of $Q\subset M$ is $E_8^{(1)}$. 
\end{proof}

To prove Theorem \ref{5theorem II2}, 
we have to exclude two more cases. 

\begin{proposition} \label{5proposition E7(1)}
Let $\fQ=\quadri$ be a quadrilateral in Class II or III 
such that the principal graph of $P\subset M$ is $E_7^{(1)}$. 
Then $\fQ$ is in Class III. 
There exists a factor $R$ and an outer action of $G=\fS_5$ on $R$ with two maximal subgroups 
$H\neq K$ of $G$ such that 
$M=R^{H\cap K}$, $P=R^H$, $Q=R^K$, and $N=R^G$. 
Such a quadrilateral of the hyperfinite II$_1$ factors exists and is unique up to conjugacy. 
\end{proposition}

\begin{proof} We use the following parameterization of $M-M$, $M-P$, and $P-P$ sectors 
arising from $P\subset M$: 
$$\hpic{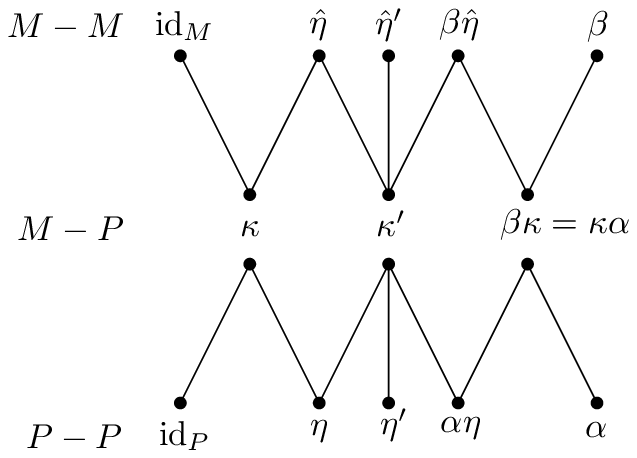}{2.0in}$$
The category of $M-M$ bimodules and $P-P$ bimodules arising from $P\subset M$ are isomorphic to 
the unitary dual of $\fS_4$, which gives the fusion rule of the above sectors.  
We have 
$$[\eta^2]=[\id_P]\oplus [\eta]\oplus[\eta']\oplus [\alpha\eta],\quad 
[\eta'^2]=[\id_P]\oplus [\alpha]\oplus [\eta'],$$
$$[\alpha\eta]=[\eta\alpha],\quad [\alpha\eta']=[\eta'\alpha]=[\eta'],\quad [\alpha^2]=[\id_P],$$
$$[\heta^2]=[\id_M]\oplus [\heta]\oplus[\heta']\oplus [\beta\heta],\quad 
[\heta'^2]=[\id_M]\oplus [\beta]\oplus [\heta'],$$
$$[\beta\heta]=[\heta\beta],\quad 
[\beta\heta']=[\heta\beta']=[\heta'],\quad [\beta^2]=[\id_M],$$
$d(\eta)=d(\heta)=3$, $d(\eta')=d(\heta')=2$, $d(\alpha)=d(\beta)=1$. 

First we claim that $\bkappa'\kappa_Q\pi$ is irreducible, which shows that $\fQ$ is in Class III thanks to 
Corollary \ref{4corollary etaiota=iota'},(4) and Lemma \ref{5lemma dim}. 
Indeed, we have 
\begin{eqnarray*}
\dim(\bkappa'\kappa_Q\pi,\bkappa'\kappa_Q\pi) &=&\dim(\kappa'\bkappa',\kappa_Q\bkappa_Q) \\
 &=&\dim(\id_M\oplus\beta\oplus 2\heta\oplus 2\beta\heta\oplus \heta',\id_M\oplus\heta_Q)\\
 &=&1+2\dim(\beta\heta,\heta_Q),
\end{eqnarray*}
which shows that either $[\beta\heta]\neq [\heta_Q]$ and $\bkappa'\kappa_Q\pi$ is irreducible or 
$[\beta\heta]=[\heta_Q]$ and $\bkappa'\kappa_Q\pi$ is a direct sum of three irreducible 
distinct sectors. 
Suppose that the latter holds and $[\bkappa'\kappa_Q\pi]=[\xi_1]\oplus [\xi_2]\oplus [\xi_3]$. 
Then $\xi_i\iota$ would contain $\iota'$ for each $i=1,2,3$. 
On the other hand, we have  
$$[\bkappa'\kappa_Q\pi\iota]=[\bkappa'\kappa\iota]=[\eta\iota]\oplus[\eta'\iota]\oplus [\alpha\eta\iota]
=[\iota']\oplus [\eta'\iota]\oplus [\alpha\iota'].$$
Since $d(\eta')=2$, $d(\iota)=\sqrt{5}$, and $d(\iota')=3\sqrt{5}$, this is impossible and we get the claim. 

Lemma \ref{5lemma dim} and Corollary \ref{4corollary etaiota=iota'},(3),(5) imply 
\begin{eqnarray*}
2&=&\dim(\heta\theta\heta,\heta\theta\heta)=\dim (\heta^2,\theta\heta^2\theta)\\
&=&\dim(\id_M\oplus \heta\oplus \heta' \oplus \beta\heta,
\id_M\oplus \theta\heta\theta\oplus \theta\heta'\theta \oplus \theta\beta\heta\theta)\\
&=&1+\dim(\heta,\theta\beta\heta\theta)+\dim(\beta\heta,\theta\heta\theta)
+\dim(\beta\heta,\theta\beta\heta\theta)+\dim(\heta',\theta\heta'\theta).
\end{eqnarray*}
If $[\heta]=[\theta\beta\heta\theta]$, we would have $[\heta^2]=[\overline{\heta}\heta]
=[\theta\heta^2\theta]$, which is impossible. 
For the same reason, we get $[\beta\heta]\neq[\theta\heta\theta]$, 
$[\beta\heta]\neq [\theta\beta\heta\theta]$ and so $[\heta']=[\theta\heta'\theta]$. 

In \cite[Corollary 4.8]{I5}, the second-named author showed that there exists a factor $R$ and 
an outer action $\mu$ of $\fS_4$ on $R$ such that $M=R^{\fS_3}$ and $P=R^{\fS4}$. 
Let $\nu=\iota_{R,M}$. 
Then we have $[\bar{\nu}\nu]=[\id_M]\oplus 2[\sigma]\oplus [\varphi]$ 
where $d(\sigma)=2$, $d(\varphi)=1$ and 
$$[\sigma^2]=[\id_M]\oplus [\varphi]\oplus[\sigma],\quad [\varphi\sigma]=[\sigma\varphi]=[\sigma],
\quad [\varphi^2]=[\id_M].$$ 
We claim that $[\heta']=[\sigma]$ and $[\beta]=[\varphi]$ hold. 
Indeed, since $d(\heta)=3$ and $\heta^2$ does not contain any non-trivial automorphism, 
$[\id_M],[\varphi],[\sigma],[\heta]$, and $[\varphi\heta]$ are all distinct sectors. 
Since $P\subset R$ is of depth 2 and $[R:P]=24$, we get 
\begin{eqnarray*}
24 &=&\dim(\nu\kappa\bkappa\bar{\nu},\nu\kappa\bkappa\bar{\nu}) 
=\dim(\bar{\nu}\nu\kappa\bkappa,\kappa\bkappa\bar{\nu}\nu)\\
&=&\dim((\id_M\oplus \varphi\oplus 2\sigma)(\id_M\oplus \heta),
(\id_M\oplus \heta) (\id_M\oplus \varphi\oplus 2\sigma))\\
&=&\dim(\id_M\oplus \varphi\oplus 2\sigma\oplus \heta\oplus \varphi\heta,
\id_M\oplus \varphi\oplus 2\sigma\oplus \heta\oplus \heta\varphi) \\
&+&2\dim(\id_M\oplus \varphi\oplus 2\sigma\oplus \heta\oplus \varphi\heta,\heta\sigma) +
2\dim(\sigma\heta,\id_M\oplus \varphi\oplus 2\sigma\oplus \heta\oplus \heta\varphi)\\
&+&4\dim(\sigma\heta,\heta\sigma)\\
&=&7+\dim(\varphi\heta,\heta\varphi)+2\dim(\sigma\oplus\varphi\sigma\oplus 2\sigma^2,\heta)
+2\dim(\heta^2,\sigma)+2\dim(\varphi\heta,\heta\sigma)\\
&+&2\dim(\heta,\sigma\oplus\sigma\varphi\oplus 2\sigma^2)+2\dim(\sigma,\heta^2)
+2\dim(\sigma\heta,\heta\varphi)+4\dim(\sigma\heta,\heta\sigma),
\end{eqnarray*}
which shows $[\varphi\heta]=[\heta\varphi]$ and
\begin{eqnarray*}
4 &=&2\dim(\sigma^2,\heta)+\dim(\heta^2,\sigma)+\dim(\heta^2\varphi,\sigma)+\dim(\sigma\heta,\heta\sigma) \\
 &=&2\dim(\heta^2,\sigma)+\dim(\sigma\heta,\heta\sigma)
 =2\dim(\id_M\oplus \heta\oplus \heta'\oplus \beta\heta,\sigma)+\dim(\sigma\heta,\heta\sigma)\\
 &=&2\dim (\heta',\sigma)+\dim(\sigma\heta,\heta\sigma).
\end{eqnarray*}
If $[\heta']\neq[\sigma]$, we would have, 
$$\dim(\sigma\heta,\sigma\heta)=\dim(\sigma^2,\heta^2)
=\dim(\id\oplus\varphi\oplus \sigma,\id\oplus\heta\oplus\heta'\oplus\beta\heta)=1,$$
which means that $\sigma\heta$ is irreducible. 
A similar argument shows that $\heta\sigma$ is irreducible too and $\dim(\sigma\heta,\heta\sigma)$ 
is at most 1, which is a contradiction. 
Thus we get $[\sigma]=[\heta']$ and in consequence, the equality $[\varphi]=[\beta]$ holds. 

Since $\fQ$ is in Class III, the principal graph of $Q\subset M$ is $E_7^{(1)}$ too 
and we can apply the same argument to $Q\subset M$. 
Then there exists a factor $\tilde{R}$ and an outer action of $\fS_4$ on $\tilde{R}$ such that 
$M=\tilde{R}^{\fS_3}$ and $Q=\tilde{R}^{\fS_4}$. 
Note that we have $[\theta\kappa]=[\kappa_Q\pi]$, which implies $[\heta_Q]=[\theta\heta\theta]$ and 
$$[\heta_Q^2]=[\id_M]\oplus[\theta\heta\theta]\oplus [\theta\heta'\theta]\oplus [\theta\beta\heta\theta]
=[\id_M]\oplus[\heta_Q]\oplus [\heta']\oplus [\theta\beta\heta\theta].$$
Thus the above argument shows $[\overline{\iota_{\tilde{R},M}}\iota_{\tilde{R},M}]
=[\overline{\iota_{R,M}}\iota_{R,M}]$. 
The second-named author and Kosaki \cite[Theorem 3.3]{IKo1} showed that  
the cohomology $H^2(M\subset R)$ is identified with $H^2(\hat{\fS_3},\T)$ introduced by 
Wassermann \cite{Was}. 
Since the order of $\fS_3$ is square free, there is only one ergodic action of $\fS_3$ 
on a 6-dimensional $C^*$-algebra and the latter is trivial. 
Therefor we may assume $R=\tilde{R}$. 
We denote by $\mu^Q$ the action of $\fS_4$ on $R$ whose fixed point algebra is $Q$. 
We set $H=\mu_{\fS_4}$ and $K=\mu^Q_{\fS_4}$. 

We claim that $N\subset R$ is irreducible. 
Indeed, we known that $\heta\theta\heta$ is a direct sum of $\theta\heta\theta$ and an irreducible, say $\zeta$, 
with $d(\zeta)=6$. 
Thus we have
\begin{eqnarray*}
\dim(\nu\kappa\iota,\nu\kappa\iota)&=&\dim(\bar{\nu}\nu,\kappa\iota\biota\bkappa)\\
&=&
\dim(\id_M\oplus\beta\oplus 2\heta',\id_M\oplus \theta\oplus \heta\oplus\theta\heta\oplus\heta\theta 
\oplus\theta\heta\theta\oplus \zeta)\\
&=&1+\dim (\beta,\theta).
\end{eqnarray*}
Since $[\beta\heta]=[\heta\beta]$, if $[\beta]=[\theta]$, the endomorphism $\heta\theta\heta$ 
would contain an automorphism, which contradicts Corollary \ref{4corollary etaiota=iota'},(3). 
Therefore the claim is shown. 
 
Let $G$ be the group generated by $H$ and $K$ in $\Aut(R)$. 
Then the fixed point algebra for $R^G$ is $P\cap Q=N$ and the action of $G$ on $R$ is outer   
for $N\subset R$ is irreducible. 
Since $N\subset P$ is 3-supertransitive, the action of $G$ on $G/H$ is 3-transitive.  
Therefore from $[G:H]=[P:N]=[M:P]+1=5$, we conclude that $G$ is isomorphic to $\fS_5$ 
because any 3-transitive transforation group of the five point set is either 
$\fA_5$ or $\fS_5$. 
Uniqueness in the case of the hyperfinite II$_1$ factors follows from Theorem \ref{4theorem uniqueness2}. 
\end{proof}

\begin{proposition} Let $\fQ=\quadri$ be a quadrilateral in Class II or III 
such that the principal graph of $P\subset M$ is $E_6^{(1)}$. 
Then $\fQ$ is in Class III. 
There exists a factor $R$ and an outer action of $G=\fA_5$ on $R$ with two maximal subgroups 
$H\neq K$ of $G$ such that 
$M=R^{H\cap K}$, $P=R^H$, $Q=R^K$, and $N=R^G$. 
Such a quadrilateral of the hyperfinite II$_1$ factors exists and is unique up to conjugacy. 
\end{proposition}

\begin{proof} We employ the same strategy as in the proof of the previous proposition 
using $\fA_4$ instead of $\fS_4$. 
Since the principal graph of $P\subset M$ is $E_6^{(1)}$, 
there exists outer automorphisms of period 3 $\alpha\in \Aut(P)$ and $\beta\in \Aut(M)$ 
such that 
$$[\eta^2]=[\id_P]\oplus [\alpha]\oplus [\alpha^2]\oplus 2[\eta], \quad 
[\alpha\eta]=[\eta\alpha]=[\eta],$$
$$[\heta^2]=[\id_M]\oplus [\beta]\oplus [\beta^2]\oplus 2[\heta], \quad 
[\beta\heta]=[\heta\beta]=[\heta].$$
$$\hpic{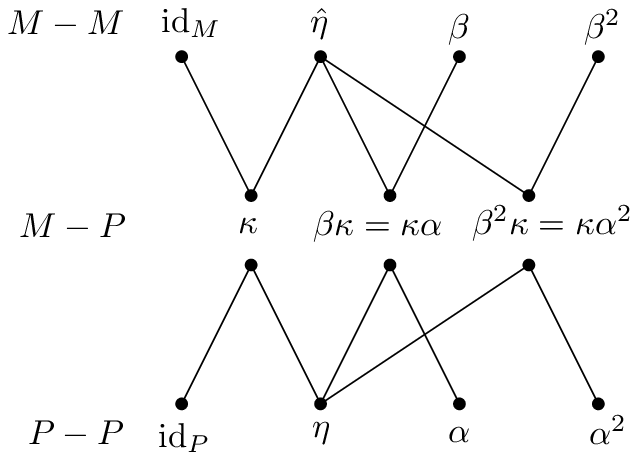}{2.0in}$$
Since $[M:Q]=4$, Theorem \ref{5theorem II1} and Proposition \ref{5proposition E7(1)} show that 
the principal graph of $Q\subset M$ is also $E_6^{(1)}$. 
Therefore there exists outer automorphism $\beta_Q\in \Aut(M)$ of period 3 such that 
$$[\heta_Q^2]=[\id_M]\oplus [\beta_Q]\oplus [\beta_Q^2]\oplus 2[\heta_Q], \quad 
[\beta_Q\heta_Q]=[\heta_Q\beta_Q]=[\heta_Q].$$
 
Since $[\kappa\eta]=[\kappa]\oplus[\kappa\alpha]\oplus [\kappa\alpha^2]$, we have 
$[\eta\xi]=[\eta\bkappa\kappa_Q\pi]=[\xi]\oplus[\alpha\xi]\oplus[\alpha^2\xi],$ 
and Lemma \ref{5lemma dim} shows $\dim(\kappa\iota\biota\bkappa, \kappa\iota\biota\bkappa)=8$. 
If $\fQ$ is in Class III, Corollary \ref{4corollary etaiota=iota'},(3) implies 
\begin{eqnarray*}
3&=&\dim(\heta\theta\heta,\heta\theta\heta)=\dim(\heta^2,\theta\heta^2\theta)\\
&=&\dim(\id_M\oplus \beta\oplus\beta^2\oplus 2\heta,
\id_M\oplus \theta\beta\theta\oplus\theta\beta^2\theta\oplus 2\theta\heta\theta),
\end{eqnarray*}
which shows either $[\beta]=[\beta_Q]$ or $[\beta]=[\beta_Q^2]$. 
Since the second cohomology of the cyclic group $\{\id_M,\beta,\beta^2\}\cong \fA_3$ is trivial, 
a similar proof, using either \cite{Ho} or \cite{I5}, as in the case of $E_7^{(1)}$ works . 

Suppose now that $\fQ$ is in Class II. 
(Note that since the cohomology of the $E_6^{(1)}$ subfactor is not trivial, we cannot 
apply Corollary \ref{4corollary etaiota=iota'},(4) as before.) 
Corollary \ref{4corollary etaiota=iota'},(4) shows that $\bpi\eta_Q\pi$ is contained in $\xi^2$, which is 
$$[\xi^2]=[\id_P]\oplus [\xi]\oplus [\alpha\xi]\oplus [\alpha^2\xi]\oplus [\eta],$$ 
and so $[\eta]=[\bpi\eta_Q\pi]$. 
The proof of Theorem \ref{4theorem (2,2)-supertransitive} shows 
$$[\kappa\iota\biota\bkappa]=[\id_M]\oplus [\heta]\oplus [\kappa_Q\pi\bkappa]
\oplus [\heta\kappa_Q\pi\bkappa].$$
Since $\dim(\kappa_Q\pi\bkappa,\kappa_Q\pi\bkappa)=\dim(\bpi\bkappa_Q\kappa_Q\pi,\bkappa\kappa)=2$, 
the endomorphism $\kappa\bpi\bkappa_Q$ is decomposed two irreducible endomorphism, say 
$\rho$ and $\sigma$, neither of which is an automorphism as we assumed that $\fQ$ is in Class II. 
By Frobenius reciprocity, the morphisms $\rho\kappa_Q\pi$ and $\sigma\kappa_Q\pi$ contain $\kappa$ 
with multiplicity one. 
Since 
$$[\kappa\bpi\bkappa_Q\kappa_Q\pi]=[\kappa(\id\oplus \eta)]
=2[\kappa]\oplus [\beta\kappa]\oplus [\beta^2\kappa],$$
we may assume $[\rho\kappa_Q\pi]=[\kappa]\oplus[\beta\kappa]$ and 
$[\sigma\kappa_Q\pi]=[\kappa]\oplus [\beta^2\kappa]$. 
By Frobenius reciprocity again, the morphism $\beta\kappa\bpi\bkappa_Q$ contains $\rho$ and 
the morphism $\beta^2\kappa\bpi\bkappa_Q$ contains $\sigma$, and so $\kappa\bpi\bkappa_Q$ contains 
$\beta^2\rho$ and $\beta\sigma$. 
If $[\beta\sigma]=[\sigma]$, we also have $[\beta^2\rho]=[\rho]$ and so 
$[\beta\kappa\bpi\bkappa_Q]=[\kappa\bpi\bkappa_Q]$. 
However, this would imply 
\begin{eqnarray*}
2&=&\dim(\beta\kappa\bpi\bkappa_Q,\kappa\bpi\bkappa_Q)
=\dim(\kappa\alpha\bpi\bkappa_Q,\kappa\bpi\bkappa_Q)
=\dim(\bkappa\kappa\alpha,\bpi\bkappa_Q\kappa_Q\pi)\\
&=&\dim(\alpha\oplus \eta,\id_P\oplus \eta),
\end{eqnarray*}
which is a contradiction. 
Thus we get $[\rho]=[\beta\sigma]$ and $d(\sigma)=2$ and 
\begin{eqnarray*}[\kappa\iota\biota\bkappa]&=&[\id_M]\oplus [\heta]\oplus [\sigma]\oplus [\beta\sigma]
\oplus [\heta\sigma]\oplus [\heta\beta\sigma]\\
&=&[\id_M]\oplus [\heta]\oplus [\sigma]\oplus [\beta\sigma]
\oplus 2[\heta\sigma]. 
\end{eqnarray*}
Since $d(\sigma)=2$, the endomorphism $\heta^2$ does not contain $\sigma$ and by Frobenius reciprocity 
$\heta\sigma$ does not contain $\heta$. 
Therefore the multiplicity of $\heta$ in $\kappa\iota\biota\kappa$ is one, and so the multiplicity of 
$\heta_Q$ is also one by symmetry. 
However, since $d(\heta_Q)=3$ and $[\heta]\neq [\heta_Q]$, this is impossible. 
Therefore $\fQ$ is in Class III. 
\end{proof}

Now the proof of Theorem \ref{5theorem II2} follows from the above propositions, lemmas, and 
Theorem \ref{5theorem II1}. 

We don't know if there are infinitely many mutually non-conjugate 
quadrilaterals of the hyperfinite II$_1$ factors in Class II. 
\subsection{Class III}

\begin{theorem} \label{5theorem III1} 
Let $\fQ=\quadri$ be a quadrilateral of factors in Class III such that 
$\fQ$ is (5,3)-supertransitive. 
Then 
\begin{itemize}
\item [$(1)$] If $\fQ$ is (6,3)-supertransitive, the principal graphs of $P\subset M$ and 
$Q\subset M$ are $A_5$ and those of $N\subset P$ and $N\subset Q$ are $E_7^{(1)}$. 
\item [$(2)$] If $\fQ$ is not (6,3)-supertransitive, 
the dual principal graph of $P\subset M$ contains one of the principal graphs of the Asaeda-Haagerup subfactor 
\cite{AH}, and in consequence $[M:P]\geq (5+\sqrt{17})/2$. 
\end{itemize}
\end{theorem}

\begin{proof} Assume that the principal graph of $P\subset M$ is not $A_5$. 
Since it is neither $A_3$ nor $A_4$, the depth of $P\subset M$ is at least 5. 
We use the same symbols $\heta',\kappa',\kappa''$, and $\eta'$ for $M-M$, $M-P$, and $P-P$ sectors 
arising from $P\subset M$ as in the proof of Theorem \ref{5theorem II1}. 
Corollary \ref{4corollary etaiota=iota'},(3) implies 
$$2\leq \dim (\heta\theta\heta,\heta\theta\heta)
=\dim(\id_M\oplus \heta\oplus \heta',\id_M\oplus \theta\heta\theta\oplus \theta\heta'\theta)
=1+\dim (\heta',\theta\heta'\theta),$$
which shows that $[\heta']=[\theta\heta'\theta]$ and the endomorphism $\heta\theta\heta$ is a direct sum of 
$\theta\heta\theta$ and an irreducible endomorphism, which we call $\hzeta$. 
Therefore we have 
$$[\kappa\iota\biota\bkappa]=[\id_M]\oplus [\theta]\oplus [\heta]\oplus [\theta\heta]\oplus 
[\heta\theta]\oplus [\theta\heta\theta]\oplus [\hzeta].$$
Since $\theta\kappa\iota=\kappa\iota$, we have $[\theta\hzeta]=[\hzeta\theta]=[\hzeta]$, 
which is also self-conjugate. 

We claim $[\theta\heta']\neq[\heta']$. 
Suppose, on the contrary, that $[\theta\heta']=[\heta']$ holds. 
Then
$$[\kappa']\oplus[\kappa'']=[\heta'\kappa]=[\theta \heta'\kappa]=[\theta\kappa']\oplus [\theta\kappa'']. $$  
If $[\theta\kappa']=[\kappa'']$, then the principal graph of $P\subset M$ would be $A_9$ and 
$\theta$ would commute with $\heta$, which is a contradiction. 
Thus $[\theta\kappa']=[\kappa']$ holds. 
However, this implies 
$$[\heta]\oplus [\heta']=[\kappa'\bkappa]=[\theta\kappa'\bkappa]=[\theta\heta]\oplus [\theta\heta']
=[\theta\heta]\oplus [\heta'],$$
and $[\theta\heta]=[\heta]$, which is a contradiction again. 
Therefore the claim is proven. 

Let $[\kappa''\bkappa]=[\heta']\oplus \bigoplus_{i=1}^nm_i[\heta_i]$ be the irreducible 
decomposition. 
We compute $\heta\theta\heta^2$ in two ways:
$$[\heta\theta\heta][\heta]=[\theta\heta\theta\heta]\oplus [\hzeta\heta]
=[\heta\theta]\oplus [\theta\hzeta]\oplus [\hzeta\heta]
=[\heta\theta]\oplus [\hzeta]\oplus [\hzeta\heta],$$
\begin{eqnarray*}[\heta\theta][\heta^2]
&=&[\heta\theta]\oplus [\heta\theta\heta]\oplus [\heta\theta\heta']
=[\heta\theta]\oplus [\theta\heta\theta]\oplus [\hzeta]\oplus [\heta\heta'\theta]\\
&=&2[\heta\theta]\oplus [\theta\heta\theta]\oplus [\heta'\theta]\oplus  [\hzeta] 
\oplus \bigoplus_{i=1}^nm_i[\heta_i\theta], 
\end{eqnarray*}
which implies 
$$[\hzeta\heta\theta]=[\heta]\oplus [\theta\heta]\oplus [\heta']\oplus 
\bigoplus_{i=1}^nm_i[\heta_i].$$
Since $[\theta][\hzeta\heta\theta]=[\hzeta\heta\theta]$ and $[\theta][\heta']\neq[\heta']$, 
we may assume that $m_1=1$, $[\theta\heta']=[\heta_1]$ and 
$$\bigoplus_{i=2}^nm_i[\theta\heta_i]=\bigoplus_{i=2}^nm_i[\heta_i].$$ 
By Frobenius reciprocity, the endomorphism $\heta'\theta\heta$ contains $\hzeta$ with 
multiplicity one. 
On the other hand, we have 
$$[\heta'\theta\heta]=[\theta\heta'\heta]=[\theta\heta]\oplus [\theta\heta']\oplus [\heta'] \oplus
\bigoplus_{i=2}^nm_i[\theta\heta_i]
=[\theta\heta]\oplus [\theta\heta']\oplus [\heta'] \oplus
\bigoplus_{i=2}^nm_i[\heta_i].$$
Corollary \ref{4corollary etaiota=iota'},(3) shows $[\hzeta]\neq [\theta\heta]$, 
and $[\theta\heta']\neq [\heta']$ implies $[\hzeta]\neq [\heta'], [\theta\heta']$ as 
$[\theta\hzeta]=[\hzeta]$ holds. 
Thus we may assume that $m_2=1$ and $[\hzeta]=[\heta_2]$. 
This in particular shows that $\fQ$ cannot be $(6,3)$-supertransitive. 

Now we have 
$$[\hzeta\heta]=[\heta\theta]\oplus [\theta\heta\theta]\oplus [\heta']\oplus [\theta\heta']
\oplus [\hzeta]\oplus \bigoplus_{i=3}^nm_i[\heta_i\theta],$$
and $[\hzeta]\neq [\heta_i\theta]$ for $i\geq 3$. 
Since the multiplicity of $\hzeta$ in $\hzeta\heta$ is one, there are exactly two vertices in the dual principal 
graph of $P\subset M$ connected to $\hzeta$, one of which is $\kappa''$. 
Therefore there exists an irreducible morphism $\kappa'''$ such that 
$[\hzeta\kappa]=[\kappa'']\oplus [\kappa''']$ with $[\kappa''']\neq [\kappa'']$. 
We claim $[\kappa''']=[\theta\heta\theta\kappa]=[\heta\theta\kappa]$. 
Indeed, computing $\dim(\theta\heta\theta\kappa,\theta\heta\theta\kappa)$, 
$\dim(\heta\theta\kappa,\heta\theta\kappa)$, and 
$\dim(\heta\theta\kappa,\theta\heta\theta\kappa)$, we can show that 
$\theta\heta\theta\kappa$ and $\heta\theta\kappa$ are irreducible and 
$[\theta\heta\theta\kappa]=[\heta\theta\kappa]$. 
Since $\theta\heta\theta\kappa\bkappa$ contains $\hzeta$ with multiplicity one, 
Frobenius reciprocity implies that $\hzeta\kappa$ contains $\theta\heta\theta\kappa$. 
If $[\theta\heta\theta\kappa]=[\kappa'']$ we would have $d(\heta)d(\kappa)=d(\kappa'')$, 
which implies $[M:P]=4$. 
Since $P\subset M$ is 5-supertransitive, the principal graph of $P\subset M$ should be 
$E_8^{(1)}$. 
However, since the $M-M$ sectors arising from $P\subset M$ is noncommutative, we get a contradiction. 
Therefore the claim is proven. 

The above computation shows that the dual principal graph is as follows. 
$$\hpic{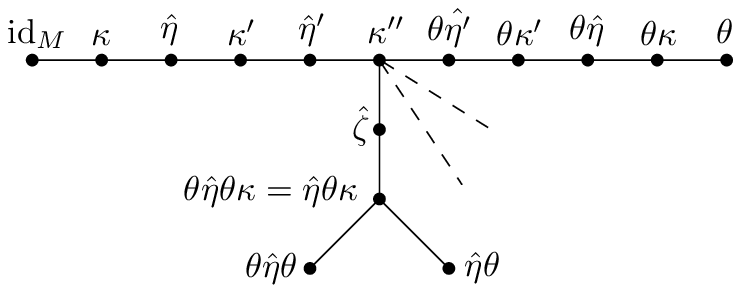}{1.4in}$$
which contains one of the principal graphs of the Asaeda-Haagerup subfactor \cite{AH}. 
\end{proof}

\begin{remark} 
We do not known whether there exists a (5,3)-supertransitive quadrilateral of 
factors in Class III with $[M:P]=(5+\sqrt{17})/2$. 
If such a quadrilateral of hyperfinite II$_1$ factors exists, it must be unique up to conjugacy. 
Indeed, the above argument shows that the inclusion $P\subset M$ must be Asaeda-Haagerup subfactor 
and $\xi=\bkappa\theta\kappa$ is determined by $P\subset M$. 
Since $N\subset P$ is 3-supsertransitive, a $Q$-system for $\id_P\oplus \xi$ is unique up to 
equivalence and $M$ is uniquely determined by $P\subset M$ up to inner conjugacy in $P$. 
Since $Q=\theta(P)$, where the representative $\theta$ of the sector $[\theta]$ is chosen so that 
the restriction $\theta|_N$ is trivial, the quadrilateral is uniquely determined by $P\subset M$. 
For existence, the same criterion using a $Q$-system as in the case of the Haagerup subfactor below 
(Theorem \ref{5theorem IV}) should work in principle, though it involves heavy computation in practice. 
\end{remark}

Theorem \ref{5theorem III1} and the propositions and lemmas in the previous subsection imply 
the following theorem:  

\begin{theorem} \label{5theorem III2} 
Let $\fQ=\quadri$ be a quadrilateral of factors in Class III such that 
$[M:P]\leq 4$. 
Then one of the following holds:
\begin{itemize}
\item [$(1)$] The principal graphs of $P\subset M$ and $Q\subset M$ are $A_5$ and those of 
$N\subset P$ and $N\subset Q$ are $E_7^{(1)}$. 
The quadrilateral is given by the construction in Example \ref{2example group} with 
$\fS_{\{1,2,3,4\}}$. 
\item [$(2)$] The principal graphs of $P\subset M$ and $Q\subset M$ are $E_6^{(1)}$. 
The quadrilateral is given by the construction in Example \ref{2example group} with $\fA_{\{1,2,3,4,5\}}$. 
\item [$(3)$] The principal graphs of $P\subset M$ and $Q\subset M$ are $E_7^{(1)}$. 
The quadrilateral is given by the construction in Example \ref{2example group} with $\fS_{\{1,2,3,4,5\}}$.  
\end{itemize}
In each case, such a quadrilateral of the hyperfinite II$_1$ factors exists and is unique 
up to conjugacy. 
\end{theorem}

\begin{corollary} \label{5corollary 6-3} 
Let $\fQ=\quadri$ be an irreducible noncommuting (6,3)-supertransitive 
quadrilateral of factors. 
Then one of the following holds: 
\begin{itemize}
\item [$(1)$] The principal graphs of $P\subset M$ and $Q\subset M$ are $A_3$ and those of 
$N\subset P$ and $N\subset Q$ are $A_5$. 
\item [$(2)$] The principal graphs of $P\subset M$ and $Q\subset M$ are $A_4$ and those of 
$N\subset P$ and $N\subset Q$ are  $D_6$. 
\item [$(3)$] The principal graphs of $P\subset M$ and $Q\subset M$ are $A_5$ and those of 
$N\subset P$ and $N\subset Q$ are $E_7^{(1)}$. 
\item [$(4)$] The principal graphs of $P\subset M$ and $Q\subset M$ are $A_7$ and those of 
$N\subset P$ and $N\subset Q$ are $A_7$. 
\end{itemize}
In each case, such a quadrilateral of the hyperfinite II$_1$ factors exists and is unique 
up to conjugacy. 
\end{corollary}

Thanks to Example \ref{2example group}, there are countably many mutually non-conjugate 
(3,3)-supertransitive quadrilaterals of the hyperfinite II$_1$ factors in Class III. 
\subsection{Class IV}

\begin{theorem} \label{5theorem IV} 
There exists a quadrilateral of factors $\fQ=\quadri$ in Class IV with 
$[M:P]=(3+\sqrt{13})/2$. 
Such a quadrilateral of the hyperfinite II$_1$ factors is unique up to flip conjugacy. 
\end{theorem}

\begin{proof} 
Let $P\subset M$ be the Haagerup subfactor whose principal graph and the dual principal 
graph are as follows.
$$\hpic{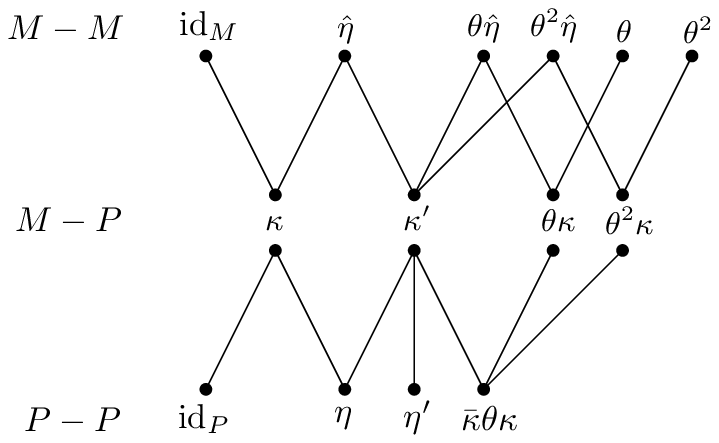}{2.0in}$$
It is known that $\theta$ and $\heta$ satisfies the following relation \cite{AH}: 
$$[\theta^3]=[\id_M],\quad [\theta^{-1}\heta]=[\heta\theta],\quad 
[\heta^2]=[\id_M]\oplus [\heta]\oplus [\theta\heta]\oplus [\theta^2\heta].$$
In the Appendix, we will show that there exists a unique $Q$-system for 
$\id_P\oplus \bkappa\theta\kappa$ up to equivalence. 
Therefore there exists a unique subfactor $N$ of $P$ up to inner conjugacy such that 
$[\iota\biota]=[\id_P]\oplus [\bkappa\theta\kappa]$, where 
$\iota:N\hookrightarrow P$ is the inclusion map. 
Since 
$$\dim(\kappa\iota,\kappa\iota)=\dim(\bkappa\kappa,\iota\biota)=
\dim(\id_P\oplus \eta,\id_P\oplus \bkappa\theta\kappa)=1,$$
the inclusion $N\subset M$ is irreducible. 
Since 
\begin{eqnarray*}
[\kappa\iota\biota\bkappa]&=&[\kappa\bkappa]\oplus [\kappa\bkappa\theta\kappa\bkappa]
=[\id_M]\oplus [\heta]\oplus [(\id_M\oplus \heta)\theta(\id_M\oplus \heta)]\\
&=&[\id_M]\oplus [\heta]\oplus [\theta]\oplus [\theta\heta]\oplus [\theta^2\heta]\oplus 
[\theta^2\heta^2]\\
&=&[\id_M]\oplus [\theta]\oplus [\theta^2]\oplus 
2[\heta]\oplus 2[\theta\heta]\oplus 2[\theta^2\heta], 
\end{eqnarray*}
we can take representative $\theta$ of the sector $[\theta]$ such that $\theta\kappa\iota=\kappa\iota$. 
Then $\theta^3$ is an inner automorphism satisfying $\theta^3\kappa\iota=\kappa\iota$ 
and so $\theta^3=\id_M$ as $N\subset M$ is irreducible. 
We set $Q=\theta(P)$.  
Using a similar argument as in the proof of Theorem \ref{5theorem I1}, we can show that 
$\quadri$ is the desired quadrilateral. 
Uniqueness up to flip conjugacy follows from uniqueness of the $Q$-system for 
$\id\oplus \bkappa\theta\kappa$. 
\end{proof}

We do not know if there exists a quadrilateral of the hyperfinite II$_1$ factors in Class IV 
different from the above example.

\section{Classification II}

\begin{theorem} \label{6theorem} Let $\fQ=\quadri$ be an irreducible noncommuting quadrilateral of factors such that 
the indices of all the elementary subfactors are less than or equal to 4. 
Then one of the following occurs: 
\begin{itemize}
\item [$(1)$] The principal graphs of $P\subset M$ and $Q\subset M$ are $A_3$ and those of 
$N\subset P$ and $N\subset Q$ are $A_5$. 
In this case $N$ is the fixed point algebra of an outer action of the symmetric  group $\fS_3$ on $M$. 
\item [$(2)$] The principal graphs of $P\subset M$ and $Q\subset M$ are $A_4$ and those of $N\subset P$ 
and $N\subset Q$ are $D_6$. 
\item [$(3)$] The principal graphs of all the elementary subfactors are $A_7$. 
\item [$(4)$] The principal graphs of $P\subset M$ and $Q\subset M$ are $A_3$ and those of 
$N\subset P$ and $N\subset Q$ are $D_6^{(1)}$. 
In this case $N$ is the fixed point algebra of an outer action of the dihedral group  $\frak{D}_8$ 
of order 8 on $M$. 
\item [$(5)$] The principal graphs of $P\subset M$ and $Q\subset M$ are $D_4$ and those of $N\subset P$ 
and $N\subset Q$ are $E_6^{(1)}$. 
In this case $N$ is the fixed point algebra of an outer action of the alternating  group $\fA_4$ on $M$. 
\item [$(6)$] The principal graphs of $P\subset M$ and $Q\subset M$ are $A_5$ and those of $N\subset P$ 
and $N\subset Q$ are $E_7^{(1)}$. 
In this case the quadrilateral $\fQ$ is given by the constructions in Example \ref{2example group} with 
$\fS_{\{1,2,3,4\}}$. 
\item [$(7)$] The principal graphs of all the elementary subfactors are $E_7^{(1)}$. 
\end{itemize}
In each case, such a quadrilateral of the hyperfinite II$_1$ factors is unique up to 
conjugacy. 
\end{theorem}

We use the symbols $\iota=\iota_{P,N}$, $\iota_Q=\iota_{Q,N}$, 
$\kappa=\iota_{M,P}$, $\kappa_Q=\iota_{M,Q}$. 
When $N\subset P$ is 2-supertransitive, we also use the symbols $\xi$ and $\hxi$ 
as in Section 4. 

\begin{lemma} \label{6lemma} 
Let $\fQ=\quadri$ be an irreducible noncommuting quadrilateral of factors such that 
the indices of all the elementary subfactors are less than or equal to 4.
Then 
\begin{itemize}
\item [$(1)$] The principal graph of $N\subset P$ is neither $A_3$, $D_4$ nor $D_4^{(1)}$. 
\item [$(2)$] If the principal graph of $N\subset P$ is $E_6^{(1)}$, then that of 
$P\subset M$ is $D_4$. 
There exists an outer action action of the alternating group $\fA_4$ on $M$ such that $N$ 
is the fixed point algebra of the action.  
\item [$(3)$] If the principal graph of $N\subset P$ is $D_n^{(1)}$, then $n=6$ and the principal graph of 
$P\subset M$ is $A_3$. 
There exists an outer action action of the dihedral group $\frak{D}_8$ of order 8 on $M$ such that $N$ 
is the fixed point algebra of the action.  
\item [$(4)$] If $N\subset P$ and $N\subset Q$ are 2-supertransitive, 
then the principal graph of $P\subset M$ is not $D_n^{(1)}$. 
\end{itemize}
\end{lemma}

\begin{proof} (1) follows from Remark \ref{3remark no-auto}. 

(2) Since the principal graph of $N\subset P$ is $E_6^{(1)}$, 
there exists an outer automorphism $\alpha\in \Aut(P)$ of order 3 such that 
$$[\xi^2]=[\id_P]\oplus [\alpha]\oplus [\alpha^2]\oplus 2[\xi].$$
The proof of Lemma \ref{4lemma xi-eta} shows that $\dim(\bkappa\kappa,\xi)=0$ and 
$$2\leq \dim(\bkappa\kappa,\xi^2)=\dim(\bkappa\kappa,\id_P\oplus \alpha\oplus\alpha^2).$$
Therefore $\bkappa\kappa$ contains $\alpha$ and so the principal graph of $P\subset M$ is 
$D_4$. 
The rest of the statement follows from \cite{Ho}, \cite{I5}. 

(3) Thanks to (1), the number $n$ is at least 5 and we have irreducible decomposition 
$$[\iota\biota]=[\id_P]\oplus [\alpha]\oplus [\xi],\quad [\biota\iota]\oplus [\tau]\oplus [\hxi],$$
where $d(\xi)=d(\hxi)=2$ and $\alpha\in \Aut(P)$ and $\tau\in \Aut(N)$ are outer automorphisms of 
order 2 satisfying $[\alpha\xi]=[\xi]$, $[\tau\hxi]=[\hxi]$. 
Since $[\kappa\iota\tau]=[\kappa\iota]$, the crossed product $N\rtimes_\tau \Z/2\Z$ is identified 
with the intermediate subfactor of $N\subset P$ generated by $N$ and a unitary $u\in M$ 
satisfying $uxu^*=\tau(x)$ for all $x\in N$. 
If $\biota_Q\iota_Q$ contained $\tau$ as well, the two inclusions $N\subset P$ and 
$N\subset Q$ would have a common intermediate subfactor $N\rtimes \tau \Z/2\Z$, 
which is a contradiction as we have $P\cap Q=N$. 
Since $\fQ$ is noncommuting $\biota_Q\iota_Q$ must contain $\hxi$ and the principal 
graph of $N\subset Q$ is either $A_5$ or $D_m^{(1)}$ with $m\geq 5$. 

Assume that the principal graph of $N\subset Q$ is $D_m^{(1)}$. 
Then we have the irreducible decomposition 
$[\biota_Q\iota_Q]=[\id_N]\oplus [\tau_Q]\oplus [\hxi]$ such that $[\tau_Q]\neq [\tau]$ and $\tau_Q$ 
is contained in $\hxi^2$. 
If $n$ were not equal to 6, the endomorphism $\hxi^2$ would contain exactly two automorphisms 
$\id_P$ and $\tau$, which would imply $[\tau]=[\tau_Q]$, a contradiction. 
Thus we get $n=m=6$ and $[\hxi^2]=[\id_N]\oplus [\tau]\oplus [\tau_Q]\oplus [\tau']$, 
where $\tau'$ is an automorphisms of $N$. 
It is known that there are 4 different subfactors with the principal graph $D_6^{(1)}$ \cite{IKa}. 
We show that the only one case among them is possible. 

Since the dual principal graph of $N\subset P$ is also $D_6^{(1)}$, 
we have $[\xi^2]=[\id]\oplus [\alpha]\oplus [\alpha_1]\oplus[\alpha_2]$ 
where $\alpha_1$ and $\alpha_2$ are automorphisms of $P$. 
The proof of Lemma \ref{4lemma xi-eta} shows 
$$2\leq \dim (\biota\bkappa\kappa\iota,\hxi)=1+\dim (\bkappa\kappa,\alpha_1\oplus \alpha_2).$$
If $\bkappa\kappa$ contained both $\alpha_1$ and $\alpha_2$, it would contain $\alpha$ too, 
which contradicts irreducibility of $N\subset M$. 
Thus we may assume that $\bkappa\kappa$ contains only $\alpha_1$, whose period must be two 
in consequence. 
Therefore either $[M:P]=2$ or the principal graph of $P\subset M$ is $D_k^{(1)}$. 
Let $L=P\rtimes_{\alpha_1}\Z/2\Z$, which is regarded as an intermediate subfactor of 
$P$ and $M$. 
Then it is known \cite[Chapter 8]{IKo1} that $N\subset L$ is given by the fixed point algebra of 
an outer action of either $\frak{D}_8$ or the Kac-Paljutkin algebra on $L$ and 
$$[\biota\bar{\sigma}\sigma\iota]=[\id_N]\oplus [\tau]\oplus [\tau_Q]\oplus 
[\tau']\oplus 2[\hxi],$$
where we set $\sigma=\iota_{L,P}$. 
We claim that $Q$ is an intermediate subfactor of $N\subset L$. 
Since $d(\hxi)=2$, Theorem \ref{2theorem decomposition},(1) shows that the multiplicity of 
$\hxi$ in $\biota\bkappa\kappa\iota$ is two and $\biota\bkappa\kappa\iota$ contains automorphisms with 
multiplicity at most one. 
On the other hand, the endomorphism $\biota_Q\iota_Q$ is contained in 
$\biota\bar{\sigma}\sigma\iota$, and so \cite[Theorem 3.9]{ILP} shows the claim. 
Since $\fQ$ is a quadrilateral we get $M=L$. 
Since $[M:P]=[M:Q]=2$, there exist period two automorphisms $\beta, \beta_1\in \Aut(M)$ such that 
$P=M^\beta$ and $Q=M^{\beta_1}$. 
If $N\subset M$ were given by the Kac-Paljutkin algebra, we would have 
$$[\kappa\iota\biota\bkappa]=[\id_M]\oplus [\beta_1]\oplus [\beta_2]\oplus [\beta_1\beta_2]\oplus 
2[\rho],$$
with $d(\rho)=2$, and the two automorphisms $\beta_1$ and $\beta_2$  
would generate a group $G$ of order 4 in $\Aut(M)$. 
However, this means $P\cap Q=M^G\neq N$, which is a contradiction. 
Therefore $N$ is the fixed point algebra of an outer action of $\frak{D}_8$. 

Now we assume that the principal graph of $N\subset P$ is $A_5$ and we show that $\fQ$ would be 
commuting, which contradicts the assumption. 
Since $[M:P][P:N]=[M:Q][Q:N]$ we have $[M:P]=3$ and $[M:Q]=4$. 
In this case, we have 
$$[\biota_Q\iota_Q]=[\id_N]\oplus [\hxi], \quad [\iota_Q\biota_Q]=[\id_Q]\oplus [\xi_Q],$$
$$[\hxi^2]=[\id_N]\oplus [\tau]\oplus [\hxi],\quad [\tau^2]=[\id_M],\quad [\tau\hxi]=[\hxi\tau]=[\hxi],$$
$$[\xi_Q^2]=[\id_Q]\oplus [\alpha_Q]\oplus [\xi_Q], \quad [\alpha_Q^2]=[\id_Q],\quad 
[\alpha_Q\xi_Q]=[\xi_Q\alpha_Q]=[\xi_Q].$$
As in \cite[Lemma 3.2]{I2}, we can choose representatives of $[\alpha_Q]$ and $[\tau]$ such that 
$\alpha_Q^2=\id_Q$, $\tau^2=\id_N$, and $\alpha_Q\iota_Q=\iota_Q\tau$ hold. 

We claim $[\kappa_Q\alpha_Q]=[\kappa_Q]$.  
Indeed, since $\kappa\iota=\kappa_Q\iota_Q$ and $[\iota\tau]=[\iota]$, we have 
$[\kappa_Q\alpha_Q\iota_Q]=[\kappa_Q\iota_Q\tau]=[\kappa_Q\iota_Q]$ and 
$$1=\dim(\kappa_Q\alpha_Q\iota_Q,\kappa_Q\iota_Q)=\dim (\kappa_Q\alpha_Q,\kappa_Q\iota_Q\biota_Q)
=\dim(\kappa_Q\alpha_Q,\kappa_Q\oplus \kappa_Q\xi_Q).$$
Since we have 
$$\dim(\kappa_Q\alpha_Q,\kappa_Q\xi_Q)=\dim(\kappa_Q,\kappa_Q\xi_Q\alpha_Q)=\dim 
(\kappa_Q,\kappa_Q\xi_Q)=\dim(\bkappa_Q\kappa_Q,\xi_Q)=0,$$
we get the claim. 

The claim implies that we can regard $L:=Q\rtimes_{\alpha_Q}\Z/2\Z$ as an intermediate subfactor 
of $Q\subset M$. 
On the other hand, we can regard $R:=N\rtimes_\tau \Z/2\Z$ as an intermediate subfactor of 
$N\subset P$. 
Since we have $\alpha_Q\iota_Q=\iota_Q\tau$, we have the inclusion $R\subset L$. 
Moreover \cite[Theorem 3.1]{I2} shows that $N$ is the fixed point algebra of an outer action of 
$\fS_3$ on $L$. 
In particular 
$\begin{array}{ccc}
R &\subset &L  \\
\cup & &\cup  \\
N &\subset  &Q
\end{array}$ 
is a commuting square as we have  $[L:R]=3$ and $[L:Q]=2$. 
By Galois correspondence, the subfactor $R$ is the fixed point subalgebra of $L$ under 
a subgroup of $\fS_3$ isomorphic to $\Z/3\Z$, and by duality, there exists an order three outer automorphism 
$\beta\in \Aut(R)$ such that $L=R\rtimes_\beta\Z/3\Z$. 

Since $[P:R]=2$, there exits an order two automorphism $\varphi\in \Aut(R)$ 
such that $P=R\rtimes_\varphi \Z/2\Z$. 
Since $\overline{\iota_{M,R}}\iota_{M,R}$ contains $\varphi$ and $\beta$ and $[M:R]=6$, 
we have $M=N\rtimes \Gamma$, where $\Gamma$ is the group generated by $[\varphi]$ and $[\beta]$, 
whose order must be 6. 
In particular 
$\begin{array}{ccc}
P &\subset &M  \\
\cup & &\cup  \\
R &\subset  &L
\end{array}$ 
is a commuting square. 
Therefore we have $E_QE_P=E_QE_LE_P=E_QE_R=E_N$, which shows that $\fQ$ is commuting, a contradiction. 
Thus the statement is proven. 

(4) Suppose that $N\subset P$ and $N\subset Q$ are 2-supertransitive and the principal graph of 
$P\subset M$ is $D_n^{(1)}$. 
If $\fQ$ were cocommuting, Corollary \ref{3corollary index} would imply $4=[M:P]\leq [P:N]-1$ which 
is a contradiction. 
If $\fQ$ were noncocommuting, the statement (3) applied to the dual quadrilateral $\hat{\fQ}$ would imply 
$M=N\rtimes \frak{D}_8$. 
However, this shows that $\fQ$ is commuting and we get a contradiction. 
Therefore the lemma is proven. 
\end{proof}

\begin{proof}[Proof of Theorem \ref{6theorem}] 
Assume first that $N\subset P$ and $N\subset Q$ are 
2-supertransitive. 
Then Lemma \ref{6lemma},(4) shows that $P\subset M$ is either 2-supertransitive 
or $M=P\rtimes G$ where $G$ is a finite abelian group of order 2,3, or 4. 
If the latter occurs, then the proof of Lemma \ref{4lemma xi-eta} implies that $\xi^2$ 
contains a non-trivial automorphism, and so the principal graph is 
either $A_5$, $E_6$, or $E_6^{(1)}$. 
Applying Lemma \ref{6lemma},(1) to the dual quadrilateral $\hat{\fQ}$, we see that 
$\fQ$ is cocommuting. 
Since the $E_6$ subfactor has trivial second cohomology, Corollary \ref{3corollary index} 
implies that $E_6$ never occurs. 
For the same reason, we have $[M:P]=[P:N]-1=2$ for $A_5$ and we get case (1) from 
Theorem \ref{2theorem no-extra}. 
For $E_6^{(1)}$, we get case (5) from Lemma \ref{6lemma},(2). 
If $P\subset M$ and $Q\subset M$ are 2-supertransitive, either Theorem 
\ref{4theorem (2,2)-supertransitive} or Lemma \ref{6lemma},(2) is applied and we get cases 
(1),(2),(3),(6), and (7) from Theorem \ref{5theorem I2}, \ref{5theorem II2}, and \ref{5theorem III2}. 

Assume now that $N\subset P$ is not 2-supertransitive. 
Then Lemma \ref{6lemma},(1) implies that the principal graph of $N\subset P$ is $D_n^{(1)}$ with 
$n\geq 5$, and so Lemma \ref{6lemma},(4) implies the case (4). 
\end{proof}

\begin{remark} Except for the case (4), we have $\Theta(P,Q)=\cos^{-1}1/([P:Q]-1)$ 
thanks to Lemma 2.5 and Corollary \ref{3corollary angle}. 
In the case (4), we have $\Theta(P,Q)=\pi/4$ thanks to \cite[Theorem 6.1]{SW}.  
\end{remark}

\section{$\alpha$-induction and GHJ pairs}

\subsection{$\alpha$-induction and angles}
Let $\cN$ be a properly infinite factor and $\{\lambda_i\}_{i\in I}$ 
be a finite system of irreducible endomorphisms in $\End_0(\cN)$.
We assume:\\
(1) there exists $0\in I$ such that $\lambda_0=\id_\cN$, \\
(2) for $i\neq j\in I$, we have $[\lambda_i]\neq[\lambda_j]$,\\
(3) for each $i\in I$, there exists $\bar{i}\in I$ such that $[\overline{\lambda_i}]=[\lambda_{\bar{i}}]$,\\
(4) there exist non-negative integers $N_{ij}^k$ such that 
$$[\lambda_i\lambda_j]=\bigoplus_{k\in I}N_{ij}^k[\lambda_k],$$ \\
(5) the system $\{\lambda_i\}_{i\in I}$ has a braiding $\{\varepsilon(\lambda_i,\lambda_j)\}_{i,j\in I}$ 
(see \cite[Definition 2.2]{BEK2} for the definition), \\
(6) $\cN\subset \cM$ is an irreducible inclusion of factors of finite index such that 
$$[\overline{\iota_{\cM,\cN}}\iota_{\cM,\cN}]=\bigoplus_{j\in J}n_j[\lambda_j],$$
where $J$ is a subset of $I$ and $n_j$ is a positive integer.

We naturally extend the braiding to the category generated by $\{\lambda_i\}_{i\in I}$ and use the same 
symbol $\varepsilon(\rho,\sigma)$ for the extension. 
We set $\varepsilon^{+}(\rho,\sigma)=\varepsilon(\rho,\sigma)$ and $\varepsilon^{-}(\rho,\sigma)=\varepsilon(\sigma,\rho)^*$. 
For simplicity, we use the following notation: $\nu=\iota_{\cM,\cN}$, 
$\gamma=\nu\bar{\nu}\in \End_0(\cM)$, $\hgamma=\bar{\nu}\nu\in \End_0(\cN)$. 
Then $\alpha$-induction $\alpha^\pm_{\lambda_i}\in \End_0(\cM)$ is defined by 
$$\alpha^\pm_{\lambda_i}=\bar{\nu}^{-1}\cdot\Ad \varepsilon^\pm(\lambda_i,\hgamma)\cdot 
\lambda_i\cdot\bar{\nu}.$$ (See \cite{LRe}, \cite{X}.)
Note that for $x\in \cN$ we have $\bar{\nu}(\nu(x))=\hgamma(x)$ and 
$\alpha^\pm_{\lambda_i}(x)=\lambda_i(x)$, 
which means $\alpha^\pm_{\lambda_i}\nu=\nu\lambda_i$. 

Let $\cH_j=(\nu,\nu\lambda_j)$. 
Then thanks to Theorem \ref{2theorem decomposition}, we have 
$$\cM=\bigoplus_{j\in J}\cN\cH_j$$
as a linear space. 
We choose an orthonormal basis $\{t(j)_k\}_{k=1}^{n_j}$ of $\cH_j$. 

\begin{lemma} \label{7lemma alpha}Let the notation be as above. 
Then for $t\in \cH_j$ we have 
$$\alpha^\pm_{\lambda_i}(t)=\varepsilon^\pm(\lambda_i,\lambda_j)^*t.$$
\end{lemma}

\begin{proof} Since $\bar{\nu}(t)\in (\hgamma,\hgamma\lambda_j)$, the statement follows from 
the equation 
$$\varepsilon(\lambda_i,\hgamma)\lambda_i(\bar{\nu}(t))\varepsilon(\lambda_i,\hgamma)^*
=\vpic{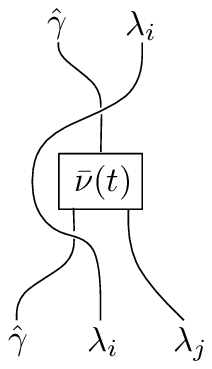}{0.8in}=\vpic{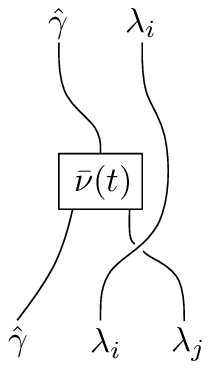}{0.8in}
=\hat{\gamma}(\varepsilon(\lambda_i,\lambda_j)^*)\bar{\nu}(t).$$
\end{proof}

From the above lemma, it is easy to show $(\rho,\sigma)\subset (\alpha^\pm_{\rho},\alpha^\pm_\sigma)$. 

We fix $i_0\in I$ and set $\lambda=\lambda_{i_0}$, 
$M=\cM$, $P=\alpha^+_{\lambda}(\cM)$, $Q=\alpha^-_{\lambda}(\cM)$, 
$N=P\cap Q$, and $R=\lambda(\cN)$. 
Since $\alpha^\pm_{\lambda}\nu=\nu\lambda$, we have $R\subset N$. 
In general these two factors do not coincide and we give a description of $N$ now. 

We set 
$$J_{i_0}=\{j\in J;\; \varepsilon(\lambda_j,\lambda)\varepsilon(\lambda,\lambda_j) 
\textrm{ is a scalar}\},$$
and set $m_j=\varepsilon(\lambda_j,\lambda)\varepsilon(\lambda,\lambda_j)$ for $j\in J_{i_0}$.
Recall that $\phi_{\rho}=r_{\rho}^*\bar{\rho}(\cdot)r_\rho$ is the standard left inverse of $\rho$. 
Since $\varepsilon(\lambda_j,\lambda)\varepsilon(\lambda,\lambda_j)\in 
(\lambda\lambda_j,\lambda\lambda_j)$, we have 
$\phi_{\lambda}(\varepsilon(\lambda_j,\lambda)\varepsilon(\lambda,\lambda_j))\in 
(\lambda_j,\lambda_j)$, which is a scalar. 
It is easy to show that $j\in J_{i_0}$ if and only if 
$$|\phi_{\lambda}(\varepsilon(\lambda_j,\lambda)\varepsilon(\lambda,\lambda_j))|=1.$$ 

\begin{lemma} \label{7lemma N} 
Let the notation be as above. 
Then 
$$N=\alpha^+_{\lambda}(\sum_{j\in J_{i_0}}\cN\cH_j)=\alpha^-_{\lambda}(\sum_{j\in J_{i_0}}\cN\cH_j).$$
In particular, the von Neumann algebra $N$ is a factor. 
\end{lemma}

\begin{proof} Let $x,y\in \cM$ satisfy $\alpha^+_{\lambda}(x)=\alpha^-_\lambda(y)$. 
Then Theorem \ref{2theorem decomposition} shows that we have expansion of $x$ and $y$  as
$$x=\sum_{j\in J}\sum_{k=1}^{n_j}a(j)_kt(j)_k,$$
$$y=\sum_{j\in J}\sum_{k=1}^{n_j}b(j)_kt(j)_k,$$
where $a(j)_k$ and $b(j)_k$ are elements in $\cN$ uniquely determined by $x$ and $y$ respectively. 
Thus we get 
$$\sum_{j\in J}\sum_{k=1}^{n_j}\lambda(a(j)_k)\varepsilon(\lambda,\lambda_j)^*t(j)_k=
\alpha^+_\lambda(x)=\alpha^-_\lambda(y)
=\sum_{j\in J}\sum_{k=1}^{n_j}\lambda(b(j)_k)\varepsilon(\lambda_j,\lambda)t(j)_k,$$
which implies 
$\lambda(a(j)_k)\varepsilon(\lambda,\lambda_j)^*=\lambda(b(j)_k)\varepsilon(\lambda_j,\lambda)$ 
for all $j\in J$ and $1\leq k\leq n_j$. 
Applying $\phi_{\lambda}$ to both 
$\lambda(a(j)_k)=\lambda(b(j)_k)\varepsilon(\lambda_j,\lambda)\varepsilon(\lambda,\lambda_j)$ 
and $\lambda(a(j)_k)(\varepsilon(\lambda_j,\lambda)\varepsilon(\lambda,\lambda_j))^*=b(j)_k$, 
we get 
$$a(j)_k=b(j)_k\phi_{\lambda}(\varepsilon(\lambda_j,\lambda)\varepsilon(\lambda,\lambda_j)),$$
$$a(j)_k\overline{\phi_{\lambda}(\varepsilon(\lambda_j,\lambda)\varepsilon(\lambda,\lambda_j))}=b(j)_k,$$
which shows that $a(k)_j\neq 0$ only if 
$|\phi_{\lambda_j}(\varepsilon(\lambda,\lambda_j)\varepsilon(\lambda_j,\lambda))|=1$, which is equivalent 
to $j\in J_{i_0}$. 
On the other hand, when $j\in J_{i_0}$ we have 
$\varepsilon(\lambda_j,\lambda)=m_j\varepsilon(\lambda,\lambda_j)^*$ and 
$$\alpha^+_\lambda(\sum_{j\in J_{i_0}}\sum_{k=1}^{n_j}a(j)_kt(j)_k)
=\alpha^-_\lambda(\sum_{j\in J_{i_0}}\sum_{k=1}^{n_j}m_j^{-1}a(j)_kt(j)_k),\quad \forall a(j)_k\in \cN,$$
which shows the first statement. 
Since $\sum_{j\in J_{i_{0}}}\cN\cH_j$ is an intermediate von Neumann algebra between 
$\cN$ and $\cM$, it must be a factor, and so is $N$. 
\end{proof}

Assume that $P\subset M$ has no non-trivial intermediate subfactor. 
Then $\fQ=\quadri$ is a quadrilateral though it is not irreducible in general. 

\begin{theorem} \label{7theorem angle}
Let the notation be as above. Then 
$$\Ang(P,Q)=\{\cos^{-1}|\phi_\lambda(\varepsilon(\lambda_j,\lambda)\varepsilon(\lambda,\lambda_j))|;\; 
j\in J\}\setminus \{0,\frac{\pi}{2}\}.$$
In particular, when the braiding is non-degenerate, we have $$\Ang(P,Q)=\{\cos^{-1}
\frac{|S_{00}S_{i_0j}|}{|S_{0i_0}S_{0j}|};\; j\in J\}\setminus \{0,\frac{\pi}{2}\},$$
where $S_{ij}$ is as in \cite[Proposition 2.4]{BEK2}.
\end{theorem}

\begin{proof} Since $\cM=\sum_{j\in J}\cN\cH_j$, we have 
$$P=\alpha^+_\lambda(\cM)=\sum_{j\in J}\lambda(\cN)\alpha^+_\lambda(\cH_j),$$
$$Q=\alpha^-_\lambda(\cM)=\sum_{j\in J}\lambda(\cN)\alpha^-_\lambda(\cH_j).$$
Thanks to Remark \ref{2remark angle2}, to compute the eigenvalues of $E_PE_QE_P$ it suffices to compute 
$E_PE_Q$ on $\alpha^+_\lambda(\cH_j)$. 
We claim that for every $t\in \cH_j$ the equalities 
$$E_Q(\alpha^+_\lambda(t))=\phi_\lambda(\varepsilon(\lambda,\lambda_j)^*\varepsilon(\lambda_j,\lambda)^*)
\alpha^-_\lambda(t),$$
$$E_P(\alpha^-_\lambda(t))=\phi_\lambda(\varepsilon(\lambda_j,\lambda)\varepsilon(\lambda,\lambda_j))
\alpha^+_\lambda(t),$$
hold. 
Note that we have 
$E_Q=\alpha^-_\lambda\phi_{\alpha^-_\lambda}$ and 
$\phi_{\alpha^-_\lambda}=r_\lambda^*\alpha^-_{\bar{\lambda}}(\cdot)r_\lambda.$ 
Thus 
\begin{eqnarray*}
\phi_{\alpha^-_\lambda}(\alpha^+_\lambda(t))&=&
r_\lambda^*\alpha^-_{\bar{\lambda}}(\alpha^+_\lambda(t))r_\lambda
=r_\lambda^*\alpha^-_{\bar{\lambda}}(\varepsilon(\lambda,\lambda_j)^*t)r_\lambda
=r_\lambda^* \bar{\lambda}(\varepsilon(\lambda,\lambda_j)^*)\varepsilon(\lambda_j,\lambda)tr_\lambda\\
&=&r_\lambda^* \bar{\lambda}(\varepsilon(\lambda,\lambda_j)^*)\varepsilon(\lambda_j,\lambda)
\lambda_j(r_\lambda)t.
\end{eqnarray*}
Since 
\begin{eqnarray*}r_\lambda^* \bar{\lambda}(\varepsilon(\lambda,\lambda_j)^*)
\varepsilon(\lambda_j,\lambda)\lambda_j(r_\lambda)
&=&\frac{1}{d(\lambda_j)}\hpic{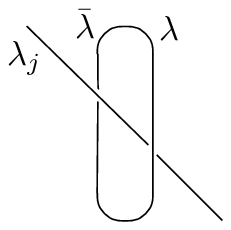}{1.0in}
=\frac{1}{d(\lambda_j)}\hpic{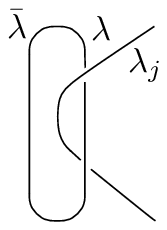}{1.0in}\\
&=&\phi_\lambda(\varepsilon(\lambda,\lambda_j)^*\varepsilon(\lambda_j,\lambda)^*),
\end{eqnarray*}
we get the first equation. 
The second one follows from a similar computation. 
Using the claim, we get 
$$E_PE_Q(\alpha^+_\lambda(t))=
|\phi_\lambda(\varepsilon(\lambda_j,\lambda)\varepsilon(\lambda,\lambda_j))|^2\alpha^+_\lambda(t),$$
which proves the statement. 
\end{proof}

\subsection{GHJ pairs}
We first recall the construction of the GHJ subfactors, which first appeared in the book of 
Goodman, de la Harpe and Jones \cite{GHJ}. 
Our presentation is based on \cite{EK}. 

Let $G$ be one of the Dynkin diagrams of type $A$, $D$, or $E$, with a distinguished vertex $*$ and  
let $\tA_n$ be the string algebra $\mathrm{String}^{(n)}G$ (see \cite[p.554]{EK} for the definition). 
For $i\geq 1$, we denote by $e_i$ the Jones projection in $\tA_{i+1}\cap \tA_{i-1}'$ 
defined by \cite[Definition 11.5]{EK}. 
Then the Jones projections $\{e_k\}_{k=1}^\infty$ satisfy the Temperley-Lieb relations:\\
(1) $e_i e_{i\pm1} e_i=\tau e_i$\\
(2) $e_i e_j=e_j e_i$ if $|i-j| \geq 2$\\
(3) $tr(e_i w)=\tau tr(w) $ if $w$ is a word on $e_1,...,e_{i-1}$.\\
Here $\tau=1/(4\cos^2 \pi/m)$, where $m$ is an integer associated to 
$G$ called \textit{the Coxeter number}. 
For each $n$, the Coxeter number of $A_n$ is $n+1$, the Coxeter number of $D_n$ is $2n-2$, and 
the Coxeter numbers of $E_6$, $E_7$, and $E_8$ are $12$, $18$, and $30$ respectively.

Let $\tB_i$ be the subalgebra of $\tA_i$ generated by 1 and $e_1,...,e_{i-1}$, for each $i=2,3,...$. 
Then the towers $\tB_i \subset \tA_i$ form commuting squares, and the von Neumann algebra $\tB$ 
generated by $\cup_{i=1}^{\infty}\tB_i $ is a subfactor of the factor $\tA$ 
generated by $\cup_{i=1}^{\infty}\tA_i $ in the GNS representation of the unique trace on 
$\cup_{i=1}^{\infty}\tA_i$. 
This is an irreducible subfactor with finite index, called the {\it GHJ subfactor} for $(G, *)$. 
The principal graphs of the GHJ subfactors were computed by Okamoto in \cite{Ok}.

Because of the Temperley-Lieb relations, there is a unitary representation of the braid group inside 
the algebras $\tB_i$, sending the usual braid group generators $\sigma_i$ to $g_i=(t+1)e_i-1$, 
where $t=e^{2\pi i/m}$, $m$ again being the Coxeter number of $G$.
Let $v_i=g_1g_2 \cdots g_{i-1}$  and $w_i=g_1^{-1}g_2^{-1}\cdots g_{i-1}^{-1}$.  
We set $\tC_i=v_i\tA_{i-1}v_i^*$  and $\tD_i=w_i\tA_{i-1}w_i^*$. 
Then the towers $\tC_i \subset \tA_i$  and $\tD_i\subset \tA_i$ 
form commuting squares, and the resulting subfactor 
$\tP \subset \tM$ and $\tQ\subset \tM$ are irreducible and $[\tM:\tP]=[\tM:\tQ]=4\cos^2\pi/k$. 
For $G=A_n$ the principal graphs of $\tP\subset \tM$ and $\tQ\subset \tM$ are $A_n$, 
for $G=D_{2n+1}$ they are $A_{4n-1}$, for $G=D_{2n}$ they are $D_{2n}$, 
for $G=E_7$ they are $A_{17}$ , and for $G=E_n$ with $n=6,8$ they are $E_n$. 
The pair $\tP$ and $\tQ$ is called the {\it GHJ pair} \cite{GJ}. 
(Though $*$ is assumed to be an endpoint in \cite[Definition 6.2.5]{GJ}, we don't pose this 
condition here.) 
We set $\tN=\tP\cap \tQ$ and set $\tR$ to be the subfactor generated by $\{e_i\}_{i=2}^\infty$. 
By construction, we have $\tR\subset \tN$ though these two factors do not necessarily 
coincide in general. 
For a special case, the angles between $\tP$ and $\tQ$ are computed in \cite[Theorem 6.3.2]{GJ}.

Now we show that $\tquadri$ is essentially the same object as $\quadri$ discussed in the previous 
subsection for an appropriate $\{\lambda_i\}_{i\in I}$ and $\cN\subset \cM$. 
Consequently, the angles between $\tP$ and $\tQ$ can be easily computed by Theorem \ref{7theorem angle}. 
Our argument below is inspired by \cite[Appendix]{BEK1}. 

Let $\cN$ be the AFD type III$_1$ factor and let $\{\lambda_i\}_{i=0}^k$ be a system of 
irreducible endomorphisms isomorphic to the irreducible sectors for the $SU(2)_k$ WZW model 
(see \cite{BEK1} for details). 
Such a system can be obtained either by the loop group construction \cite{Was} or by 
combination of \cite{Wen} and \cite{HY} using the quantum $SU(2)$ at a root of unity. 
Then the system $\{\lambda_i\}_{i=0}^k$ satisfies the assumptions (1)-(5) of the previous 
subsection. 
The principal graph of $\lambda_1(\cN)\subset \cN$ is $A_{k+1}$ and 
$$[\lambda_i\lambda_{j}]=\bigoplus_{0\leq 2l\leq \min\{|i+j|,k\}-|i-j|}[\lambda_{|i-j|+2l}],$$
$$d(\lambda_i)=\frac{\sin\frac{(i+1)\pi}{k+2}}{\sin\frac{\pi}{k+2}},$$
$$S_{ij}=\sqrt{\frac{2}{k+2}}\sin\Big(\frac{(i+1)(j+1)\pi}{k+2}\Big).$$

Assume that an inclusion of factors $\cN\subset \cM$ satisfies the assumption (6). 
We set $i_0=1$ and we use the same notation as in the previous subsection. 
It is known that such an inclusion is completely classified by a pair $(G,*)$ 
(up to graph automorphism) where $G$ is one of the Coxeter graphs of type $A$, $D$, or $E$ 
\cite[Theorem 2.1]{KLPR}, so that we can identify $\tA_n$ with $(\nu\lambda^n,\nu\lambda^n)$ 
and $\tB_n$ with $\nu((\lambda^n,\lambda^n))$. 
Since $\nu\lambda=\alpha^\pm_\lambda\nu$, we have the inclusion relation 
$\alpha^\pm_\lambda(\tA_n)\subset 
(\alpha^\pm_\lambda\nu\lambda^n,\alpha^\pm_\lambda\nu\lambda^n)
=(\nu\lambda^{n+1},\nu\lambda^{n+1})=\tA_{n+1}$. 
Recall that the trace $\tr$ on $\tA_n$ is given by $\phi_{\nu\lambda^n}$. 

\begin{lemma} \label{7lemma CD} 
Let the notation be as above. 
Then
\begin{itemize}
\item [$(1)$] For $x\in \tA_n$, we have  
$\alpha^\pm_\lambda(x)=\varepsilon^\pm(\lambda,\lambda^n)^*x
\varepsilon^\pm(\lambda,\lambda^n)$. 
In consequence, the restriction of $\alpha^\pm_\lambda$ to $\tA_n$, as a map from 
$\tA_n$ to $\tA_{n+1}$, is trace preserving. 
\item [$(2)$] For $x\in \tA_n$ we have $E_P(x)\subset \alpha^+_\lambda(\tA_n)$ and 
$E_Q(x)\subset \alpha^-_\lambda(\tA_n)$. 
The restrictions of $E_P$ and $E_Q$ to $\tA_n$ is trace preserving.  
\end{itemize}
\end{lemma}

\begin{proof}
(1) Let $x\in \tA_n$. 
Then $\bar{\nu}(x)\in (\hgamma\lambda^n,\hgamma\lambda^n)$ and 
\begin{eqnarray*}
\varepsilon(\lambda,\hgamma)\lambda(\bar{\nu}(x))\varepsilon(\lambda,\hgamma)^*
&=&\hpic{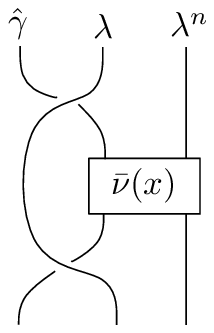}{1.2in}=\hpic{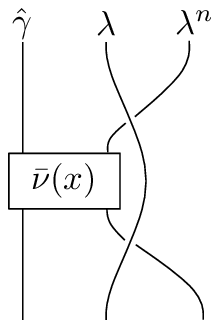}{1.2in}=\hpic{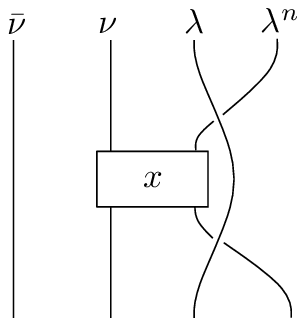}{1.2in}\\
&=&\bar{\nu}(\varepsilon(\lambda,\lambda^n)^*x\varepsilon(\lambda,\lambda^n)),
\end{eqnarray*}
which shows $\alpha^+_\lambda(x)=\varepsilon(\lambda,\lambda^n)^*x
\varepsilon(\lambda,\lambda^n)$. 
The second equality can be shown in the same way.

(2) Since $E_P(x)=\alpha^+_\lambda\phi_{\alpha^+_\lambda}(x)
=\alpha^+_\lambda(r_\lambda^*\alpha^+_\lambda(x)r_\lambda)$, it suffices to show 
$r_\lambda^*\alpha^+_\lambda(x)r_\lambda\in \tA_n$ for $x\in \tA_n$, 
which can be easily shown by using (1). 
This also shows 
\begin{eqnarray*}
\tr(E_P(x))&=&\tr(\alpha^+_\lambda\phi_{\alpha^+_\lambda}(x))=\tr(\phi_{\alpha^+_\lambda}(x))
=\phi_{\nu\lambda^n}\phi_{\alpha^+_\lambda}(x)=\phi_{\alpha^+_\lambda\nu\lambda^n}(x)\\
&=&\phi_{\nu\lambda^{n+1}}(x)=\tr(x).
\end{eqnarray*}
\end{proof}

Since $\varepsilon(\lambda,\lambda^n)^*
=\varepsilon(\lambda,\lambda)^*\lambda(\varepsilon(\lambda,\lambda)^*)\cdots 
\lambda^{n-1}(\varepsilon(\lambda,\lambda)^*)$, we can identify 
$\alpha^+_\lambda(\tA_{n-1})$ with $\tC_n$ and 
$\alpha^-_\lambda(\tA_{n-1})$ with $\tD_n$ (for an appropriate 
choice of the braiding). 

We now show that $\quadri$ and $\tquadri$ are conjugate after tensor product with 
the AFD type III$_1$ factor. 

\begin{lemma} \label{7lemma rel-commu}
Let $\cL$ be a factor and $\sigma\in \Mor_0(\cN,\cL)$. 
We set $\cA_n=(\sigma\lambda^n,\sigma\lambda^n)$ and 
$\cB_n=\sigma((\lambda^n,\lambda^n))$.  
Let $\cA$ and $\cB$ be the II$_1$ factors generated by 
$\cup_n\cA_n$ and $\cup_n\cB_n$ respectively on the GNS Hilbert space 
of the unique trace on $\cup_n\cA_n$. 
Then $\cA\cap \cB'=(\sigma,\sigma)$.  
\end{lemma}

\begin{proof} Since $(\lambda^n,\lambda^n)$ is generated by Jones projections, the statement 
follows from the flatness of the Jones projections \cite[Chapter 12]{EK}. 
\end{proof}

\begin{theorem} \label{7theorem GHJ=alpha} 
Let the notation be as above. 
Then $
\begin{array}{ccc}
\tP\otimes R_\infty &\subset &\tM\otimes R_\infty\\
\cup & &\cup  \\
\tN\otimes R_\infty &\subset &\tQ\otimes R_\infty  
\end{array}$ 
and $\quadri$ are conjugate where $R_\infty$ is the AFD type III$_1$ factor.  
\end{theorem}

\begin{proof} First we show that $R\subset M$ and $\tR\subset \tM$ have the same standard 
invariant. 
We fix a positive integer $n$ larger than half of the depth of $R\subset M$, 
which is finite now. 
Let $\rho=\nu\lambda$. 
Then $R=\rho(\cN)$ and 
$$
\begin{array}{ccccccccc}
(\brho\rho)^{n-1}\brho((\rho,\rho)) &\subset &(\brho\rho)^{n-1}((\brho\rho,\brho\rho)) 
&\subset \cdots \subset &((\brho\rho)^n,(\brho\rho)^n)\\
\cup & &\cup &  &\cup \\
\C &\subset &(\brho\rho)^{n-1}((\brho,\brho)) & \subset \cdots \subset  
&((\brho\rho)^{n-1}\brho, (\brho\rho)^{n-1}\brho) 
\end{array}
$$
is the standard invariant of $R\subset M$ (or rather 
$(\brho\rho)^n(\cN)\subset (\brho\rho)^{n-1}\brho(\cM)$). 
We set 
$$\cA_{2j,m}=(\bar{\rho}\rho)^{n-j}(((\brho\rho)^j\lambda^m,(\brho\rho)^j\lambda^m)),$$ 
$$\cA_{2j+1,m}=
(\bar{\rho}\rho)^{n-j-1}\bar{\rho}((\rho(\brho\rho)^j\lambda^m,\rho(\brho\rho)^j\lambda^m)).$$ 
Note that $\brho\rho$ is a direct sum of irreducibles in $\{\lambda_i\}_{i=0}^k$. 
When $m$ is larger than the depth of $\lambda(\cN)\subset \cN$, 
the inclusion graph for $\cA_{j,m}\subset \cA_{j+1,m}$ is the reflection of that for 
$\cA_{j+1,m}\subset \cA_{j+2,m}$ thanks to Frobenius reciprocity. 
Therefore 
$$\cA_{j,m}\subset \cA_{j+1,m}\subset \cA_{j+2,m}$$ 
is the basic construction. 

Let $\cA_{j,\infty}$ be the factor generated by $\cup_{m}\cA_{j,m}$ in the 
GNS representation of the unique trace of $\cup_{m}\cA_{2n,m}$. 
Then we may identify $\tR\subset \tM$ with $\cA_{0,\infty}\subset \cA_{1,\infty}$ and
$$\cA_{0,\infty}\subset \cA_{1,\infty}\subset \cA_{2,\infty}\subset \cdots\subset 
\cA_{2n,\infty}$$
is the Jones tower. 
Lemma \ref{7lemma rel-commu} implies 
$$\cA_{2j,\infty}\cap \cA_{0,\infty}'=(\brho\rho)^{n-j}(((\bar{\rho}\rho)^j,(\bar{\rho}\rho)^j)),$$
$$\cA_{2j+1,\infty}\cap \cA_{0,\infty}'=(\brho\rho)^{n-j-1}\brho
((\rho(\bar{\rho}\rho)^j,\rho(\bar{\rho}\rho)^j)),$$ 
which shows that $R\subset M$ and $\tR\subset \tM$ have the same principal graph. 
On the other hand, we have 
$$\cA_{2j,\infty}\cap \cA_{1,\infty}'\supset 
(\brho\rho)^{n-j}(((\bar{\rho}\rho)^{j-1}\brho,(\bar{\rho}\rho)^{j-1}\brho)),$$
$$\cA_{2j+1,\infty}\cap \cA_{1,\infty}'\supset 
(\brho\rho)^{n-j-1}\brho((\rho\brho)^j,(\rho\brho)^j)).$$
Since $[M:R]=[\tM:\tR]$, the equality holds in the above inclusions 
and the two subfactors $R\subset M$ and $\tR\subset \tM$ have the same standard invariant. 
Therefore we conclude that the two inclusions $R\subset M$ and 
$\tR\otimes R_\infty\subset \tM\otimes R_\infty$ are conjugate thanks to Popa's result \cite{P2} 
with \cite{I3} (see \cite{M2} too). 

To finish the proof, it suffices to show that $e_P=e_{\tP}$ and $e_Q=e_{\tQ}$ hold in the above 
identification of the standard invariants, which can be done by using Lemma \ref{7lemma CD},(2). 
\end{proof}

\begin{remark} \label{7remark crossed-product}
Since 
$$|\phi_\lambda(\varepsilon(\lambda_j,\lambda)\varepsilon(\lambda,\lambda_j))|
=\frac{|S_{1j}|}{d(\lambda_1)d(\lambda_j)|S_{00}|}
=\frac{|\sin\frac{\pi}{k+2}\sin\frac{2(j+1)\pi}{k+2}|}{|\sin\frac{2\pi}{k+2}\sin\frac{(j+1)\pi}{k+2}|}
=\frac{|\cos\frac{(j+1)\pi}{k+2}|}{|\cos\frac{\pi}{k+2}|},$$
the operator $\varepsilon(\lambda_j,\lambda)\varepsilon(\lambda,\lambda_j)$ is a scalar if and only if 
$j=0,k$. 
Thus we have $J_{i_0}=J\cap\{0,k\}$. 
It is known that $k\in J$ occurs only if $k$ is even. 
In this case, we may and do assume $\lambda_k^2=\id_\cN$ by choosing a representative $\lambda_k$ 
satisfying $\lambda_k\lambda_{k/2}=\lambda_{k/2}$, which 
implies that there exists a unitary $u_k\in \cH_k$ satisfying $u_k^2=1$. 
Therefore we may regard $\cN+\cN\cH_k$ as the crossed product 
$\cN\rtimes_{\lambda_k}\Z/2\Z$. 
We make this assumption and identification in what follows. 
\end{remark}

The following corollary answers the question raised in 
\cite[comment after Proposition 6.2.4]{GJ} in the negative. 
(The subfactor $\tR$ is denoted by $TL2$ in \cite{GJ}.)

\begin{corollary} \label{7corollary tNtR}
Let the notation be as above. 
Then $\tR=\tP\cap \tQ$ if and only if $k\notin J$. 
If $k\in J$, the index of $\tR \subset \tP\cap \tQ$ is 2. 
\end{corollary}

\begin{proof} Thanks to Theorem \ref{7theorem GHJ=alpha}, we may work on $R$ and $N$ 
instead of $\tR$ and $\tN$, which together with Lemma \ref{7lemma N} show the first statement. 
When $k\in J$, Lemma \ref{7lemma N} shows $N=\alpha^\pm_\lambda(\cN+\cN\cH_k)$ and 
the second statement follows from Remark \ref{7remark crossed-product}. 
\end{proof}

The following corollary is a generalization of \cite[Theorem 6.3.2]{GJ}. 

\begin{corollary} \label{7corollary angle}
Let the notation be as above. 
Then 
$$\Ang(\tP,\tQ)=\{\cos^{-1}\frac{|\cos\frac{(j+1)\pi}{k+2}|}{|\cos\frac{\pi}{k+2}|};\;j\in J\}\setminus 
\{0,\frac{\pi}{2}\}.$$
\end{corollary}

\begin{example} \label{7example E6} 
Let $\cN\subset \cM$ be the inclusion of the AFD type III$_1$ factors associated with 
the conformal embedding $SU(2)_{10}\subset SO(5)_1$ (see \cite[p.381]{X}). 
Then the corresponding GHJ subfactor is given by $E_6$ with $*$ equal to the vertex 
of the minimum value of the Perron-Frobenius eigenvector, and so its principal graph 
is as in Theorem \ref{5theorem II2},(2) (see \cite[p.726]{EK},\cite{Ok}). 
Since $[\hgamma]=[\lambda_0]\oplus [\lambda_6]$, we have $\tR=\tN$, and so $\tN\subset \tP$ is 
conjugate to the GHJ subfactor, while $\tP\subset \tM$ is the $E_6$ subfactor. 
Since 
$$\dim(\nu\lambda,\nu\lambda)=\dim(\hgamma,\lambda^2)
=\dim(\lambda_0\oplus \lambda_6,\lambda_0\oplus \lambda_2)=1,$$
the inclusion $\tN\subset \tM$ is irreducible. 
The angle between $\tP$ and $\tQ$ is computed as $\Theta(\tP,\tQ)=\cos^{-1}(2-\sqrt{3})$ 
by using either of \cite[Theorem 6.3.2]{GJ}, Theorem \ref{4theorem (2,2)-supertransitive}, or 
Corollary \ref{7corollary angle}. 
\end{example}

It is observed in \cite[Proposition 6.2.6]{GJ} that $\tR\subset \tM$ is irreducible if and only if 
$*$ is an endpoint. 
However, even when $*$ is not an endpoint, the inclusion $\tN\subset \tM$ may be irreducible 
(though it is very subtle in general to decide whether $\tN\subset \tM$ is irreducible or not 
as we will see below). 
We use this observation to compute the opposite angles of quadrilaterals 
constructed in \cite{G}. 
The following theorem is shown in \cite{G}. 

\begin{theorem} \label{7theorem fTL}
Let 
$\begin{array}{ccc}
B &\subset&A  \\
\cup& &\cup  \\
D &\subset &C 
\end{array}$
be an irreducible noncommuting quadrilateral of the hyperfinite II$_1$ factors such that 
$D\subset B$ and $D\subset C$ are supertransitive and 
$[B:D]$ and $[C:D]$ are less than 4. 
Then the principal graphs of $D\subset B$ and $D\subset C$ are $A_{2n+1}$ with an integer 
$n\geq 2$. 
For each integer $n\geq 2$, such a quadrilateral exists and is unique up to conjugacy. 
\end{theorem}

We now compute the ``opposite angles'' of the above quadrilateral. 

\begin{theorem} Let 
$\begin{array}{ccc}
\hat{C} &\subset&\hat{D}  \\
\cup& &\cup  \\
A &\subset &\hat{B} 
\end{array}$
be the dual quadrilateral of the quadrilateral in the previous theorem. 
Then 
$$\Ang(\hat{B},\hat{C})=\{\cos^{-1}\frac{\cos\frac{(2j+1)\pi}{2n+2}}{\cos\frac{\pi}{2n+2}};\;j=1,2,\cdots, 
[\frac{n-1}{2}]\},$$
where $[(n-1)/2]$ is the integer part of $(n-1)/2$. 
\end{theorem}

\begin{proof} We show that the above quadrilateral is given by the GHJ pair of $A_{2n+1}$ with $*$ equal to the 
midpoint. 
Thanks to Theorem \ref{7theorem GHJ=alpha}, we can use $\alpha$-induction instead of 
the GHJ construction for this purpose. 
Let $\cN$ and $\{\lambda_i\}_{i=0}^k$ be as before with $k=2n$. 
Let $(\cM,\{\tlambda_i\}_{i=0}^k)$ be a copy of $(\cN,\{\lambda_i\}_{i=0}^k)$. 
We embed $\cN$ into $\cM$ by 
$$\nu: \cN\ni x\mapsto \lambda_n(x)\in \cM.$$ 
Then we have $\bar{\nu}=\tlambda_n$ regarded as a map from $\cM$ into $\cN$, and so
$\hgamma=\lambda_n^2$ and $\gamma=\tlambda_n^2$. 
Since 
$$[\lambda_n^2]=\bigoplus_{i=0}^n[\lambda_{2i}],$$
we have $J=\{0,2,4,\cdots, k\}$ and $J_{i_0}=\{0,k\}$. 
Therefore thanks to Corollary \ref{7corollary tNtR}, 
$$N=\alpha^\pm_\lambda(\cN\rtimes_{\lambda_k}\Z/2\Z)
=\alpha^\pm_\lambda(\nu(\cN)+\nu(\cN)u_k),$$
where $u_k\in (\nu,\nu\lambda_k)=(\tlambda_n,\tlambda_n\tlambda_k)$. 
As before, we set $M=\cM$, $P=\alpha^+_\lambda(\cM)$, $Q=\alpha^-_\lambda(\cM)$, 
$N=P\cap Q$, and $\fQ=\quadri$. 
Then the principal graphs of $P\subset M$ and $Q\subset M$ are $A_{2n+1}$. 
To prove the theorem, it suffices to show that $N\subset M$ is irreducible and 
$\fQ$ is not cocommuting thanks to Corollary \ref{7corollary angle} and 
Theorem \ref{7theorem fTL}. 

We claim $\alpha^\pm_\lambda=\Ad\varepsilon^\pm(\lambda,\lambda_n)\tlambda$. 
Indeed, since 
$\varepsilon^\pm(\lambda,\hgamma)=\lambda_n(\varepsilon^\pm(\lambda,\lambda_n))\varepsilon^\pm(\lambda,\lambda_n)$ and 
$\varepsilon^\pm(\lambda,\lambda_n)\in (\lambda\lambda_n,\lambda_n\lambda)=(\lambda\bar{\nu},\bar{\nu}\tlambda)$, 
we have 
$$\Ad(\varepsilon^\pm(\lambda,\hgamma))\lambda\bar{\nu}(x)
=\Ad \lambda_n(\varepsilon^\pm(\lambda,\lambda_n))\Ad\varepsilon^\pm(\lambda,\lambda_n)\lambda\bar{\nu}(x)
=\Ad \bar{\nu}(\varepsilon^\pm(\lambda,\lambda_n))\bar{\nu}(\tlambda(x))$$
for all $x\in \cM$, which shows the claim. 
In particular, the two subfactors $P$ and $Q$ are inner conjugate in $M$. 
Therefore if $N\subset M$ is irreducible, Theorem \ref{3theorem angle} applied to the dual 
quadrilateral $\hat{\fQ}$ shows that $\fQ$ is not cocommuting. 

Now our only task is to show $M\cap N'=\C$. 
Note that $N$ is generated by $\nu(\lambda(\cN))$ and 
$\nu(\varepsilon(\lambda,\lambda_k)^*)u_k$ with $u_k\in (\nu,\nu\lambda_k)$, and so 
$$M\cap N'=(\nu\lambda,\nu\lambda)\cap\{\nu(\varepsilon(\lambda,\lambda_k)^*)u_k\}'.$$ 
Therefore it suffices to show 
$$(\lambda_n\lambda,\lambda_n\lambda)\cap\{\lambda_n(\varepsilon(\lambda,\lambda_k)^*)u\}'=\C.$$
where $u$ is a unitary in $(\lambda_n,\lambda_n\lambda_k)$. 
By Frobenius reciprocity, we have 
$$\dim(\lambda_n\lambda,\lambda_n\lambda)=\dim (\lambda_n^2,\lambda^2)
=\dim (\lambda_n^2,\lambda_0\oplus \lambda_2)=2,$$
and we can choose a basis $\{1,s\}$ of $(\lambda_n\lambda,\lambda_n\lambda)$ with 
$$s=\hpic{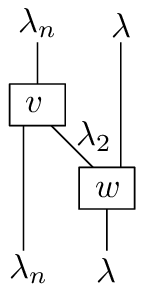}{1.2in},$$
where $v\in (\lambda_n,\lambda_n\lambda_2)$ and $w\in (\lambda_2\lambda,\lambda)$. 
On one hand, we have  
$$s\lambda_n(\varepsilon(\lambda,\lambda_k)^*)u=\hpic{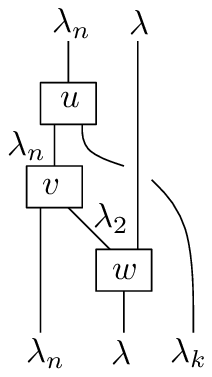}{1.5in}=\hpic{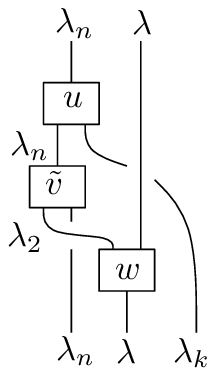}{1.5in},$$
where $\tilde{v}=\varepsilon (\lambda_2,\lambda_n)^*v$. 
On the other hand, 
$$\lambda_n(\varepsilon(\lambda,\lambda_k)^*)us=\hpic{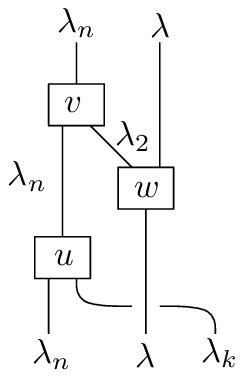}{1.5in}=\hpic{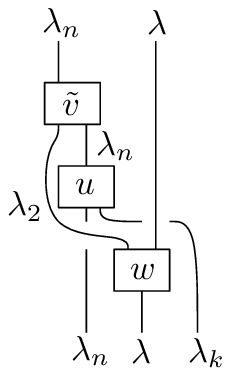}{1.5in}.$$
Since $\dim (\lambda_n,\lambda_2\lambda_n\lambda_k)=1$, 
there exists a scalar $c$ such that 
$$\hpic{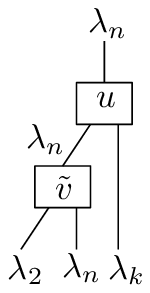}{1.2in}=c\hpic{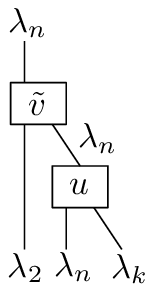}{1.2in},$$
which implies  
$s\lambda_n(\varepsilon(\lambda,\lambda_k)^*)u=c\lambda_n(\varepsilon(\lambda,\lambda_k)^*)us$. 
The number $c$ is given by Kirillov-Reshetikhin's quantum $6j$-symbol with $q=e^{\pi i/(n+1)}$ 
(see \cite[(5.4)]{KR}).  
Using the notation in \cite{KR} (irreducible representations are parameterized by half integers there), 
we get 
$$c=\left\{
\begin{array}{ccc}
n & \frac{n}{2}&\frac{n}{2}  \\
1 &\frac{n}{2} &\frac{n}{2} 
\end{array}
\right\}_q.$$
Note that although an explicit formula of the $6j$-symbols depends on the choice of intertwiners in general, 
the above one does not because the same intertwiners $u$ and $\tilde{v}$ 
appear in the both sides. 
This can be computed by \cite[Theorem 5.1]{KR} and it turns out that we have $c=-1$. 
Therefore $N\subset M$ is irreducible. 
\end{proof}

\subsection{Asymptotic inclusions and angles}
Another rich source of noncommuting quadrilaterals is asymptotic inclusions. 
We show that a technique similar to the one we used in the previous subsection works 
in this case too. 

Let $\cN\subset \cM$ be an irreducible inclusion of the hyperfinite II$_1$ factors with finite index and 
of finite depth and let 
$$\cN\subset \cM\subset \cM_1\subset \cdots \subset \cM_n\subset \cdots$$
be the Jones tower. 
The algebra $\cup_{n}\cM_n$ has a unique trace and we denote by $\cM_\infty$ the factor generated by 
$\cup_{n}\cM_n$ in the GNS representation of the unique trace of $\cup_{n}\cM_n$. 
Then the inclusion $\cM\vee (\cM_\infty\cap \cM')\subset \cM_\infty$ is called the asymptotic inclusion of 
$\cN\subset \cM$ (see \cite[Chapter 12.6]{EK} for details). 
It is easy to show that 
$\begin{array}{ccc}
\cN\vee (\cM_\infty\cap \cN') &\subset &  \cM_\infty\\
\cup &&\cup\\
\cN\vee(\cM_\infty\cap \cM') & \subset & \cM\vee (\cM_\infty\cap \cM')
\end{array}$ 
is an irreducible commuting quadrilateral of factors. 
We show how to compute the ``opposite angles" of this quadrilateral.  

\begin{example}
When the principal graph of $\cN\subset \cM$ is $A_n$, 
there exists a sequence of Jones projection $\{e_i\}_{i=1}^\infty$ such that 
$\cN={\{e_i\}_{i=2}^\infty}''\subset \cM={\{e_i\}_{i=1}^\infty}''$. 
We set $e_0=e_N$, $e_{-1}=e_{\cM}$ and $e_{-(i+1)}=e_{\cM_{i}}$ for $i\geq 1$. 
Then $\cM_\infty={\{e_i\}_{i=-\infty}^\infty}''$, $\cM\vee (\cM_\infty\cap \cM')=\{e_i\}_{i\neq 0}''$, 
$\cN\vee (\cM_\infty\cap \cN')=\{e_i\}_{i\neq 1}''$, and 
$\cN\vee(\cM_\infty\cap \cM')=\{e_i\}_{i\neq 0,1}''$. 
The index of the asymptotic inclusion in this case is computed in \cite{Ch} and the 
principal graphs are obtained in \cite{EK1}. 
\end{example}

It is well-known that $\cM_\infty\cap \cM'$ is naturally identified with the opposite algebra 
$\cM^\opp$ of $\cM$ and 
$\cM\vee(\cM_\infty \cap \cM')$ is naturally identified with $\cM\otimes \cM^\opp$, 
and so we identify $\cN\vee(\cM_\infty\cap \cM')$ with $\cN\otimes \cM^\opp$ too. 
However, we never identify $\cN\vee(\cM_\infty\cap \cN')$ with $\cN\otimes \cN^\opp$ as 
we cannot take common factorizations $\cM\otimes \cM^\opp$ and $\cN\otimes \cN^\opp$ for 
$\cM\vee(\cM_\infty\cap\cM')$ and $\cN\vee(\cM_\infty\cap\cN')$. 
By cutting the Hilbert space $L^2(\cM_\infty)$ by an appropriate projection in $\cM_\infty'$, 
we may assume that $\cM_\infty$ acts on $L^2(\cM\otimes \cM^\opp)=L^2(\cM)\otimes L^2(\cM^\opp)$. 
Since $e_N$ may be confused with $e_N\otimes 1$ in this representation, we use the symbol $e$ 
for $e_N$ in $\cM_\infty$. 
Let $J=J_\cM\otimes J_{\cM^\opp}$ be the canonical conjugation of $\cM\otimes \cM^\opp$. 
We set 
$$M:=J\big(\cN\vee(\cM_\infty\cap \cM')\big)'J=J(\cN\otimes \cM^\opp)'J=\cM_1\otimes \cM^\opp,$$ 
$$P:=J\big(\cM\vee(\cM_\infty\cap \cM')\big)'J=J(\cM\otimes \cM^\opp)'J=\cM\otimes \cM^\opp,$$ 
$$Q:=J\big(\cN\vee(\cM_\infty\cap \cN')\big)'J, \quad  N:=J\cM_\infty'J.$$ 

Note that $Q$ is not equal to $\cM_1\otimes \cN^\opp$ as we formally (just formally) broke the symmetry 
between $\cM\vee(\cM_\infty\cap \cM')$ and $\cN\vee(\cM_\infty\cap \cN')$ by representing 
$\cM_\infty$ on $L^2(\cM\otimes \cM^\opp)$. 
However, we still have the following: 

\begin{lemma} \label{7lemma Q} 
Let the notation be as above. 
Then $Q$ and $\cM_1\otimes \cN^\opp$ are inner conjugate in $M$. 
In particular, we have $M\otimes_QM\cong \cM_1\otimes \cM_1^\opp$ as 
$M-M$ bimodules. 
\end{lemma}

\begin{proof} Note that 
$$\cN\otimes \cN^\opp\subset \cN\otimes \cM^\opp\subset \cN\vee(\cM_\infty\cap\cN')
=(\cN\otimes \cM^\opp)\vee\{e\}$$
is the basic construction with the Jones projection $e$. 
Thus 
$$Q\subset \cM_1\otimes \cM^\opp \subset J(\cN\otimes \cN^\opp)'J=\cM_1\otimes \cM_1^\opp$$
is the basic construction too, which shows that $Q$ and $\cM_1\otimes \cN^\opp$ are inner conjugate 
in $M$. 
\end{proof}

Our purpose of this subsection is to show the following theorem: 

\begin{theorem} \label{7theorem asymptotic}
Let the notation be as above. 
If the principal graph of $\cN\subset \cM$ is $A_n$, then 
$$\Ang(P,Q)=\{\cos^{-1}\frac{\cos\frac{(j+1)\pi}{n+1}}{\cos\frac{\pi}{n+1}};\;j=1,2,\cdots [\frac{n-2}{2}]\}.$$
\end{theorem}

We go back to the general (not necessarily $A_n$) case. 
Let $\fQ=\quadri$ be as above. 
To compute $\Ang(P,Q)$ we may replace $M$ with 
$$M\otimes \B(\ell^2)\otimes \B(\ell^2)\cong \cM\otimes\B(\ell^2)\otimes (\cM\otimes\B(\ell^2))^\opp,$$
and so we may and do assume that $\cN\subset \cM$ is an inclusion of properly infinite factors from 
now on. 
In this case, the inclusion $N\subset P$ is described as the Longo-Rehren inclusion \cite{M1} 
whose structure is well-studied \cite{I6}. 
Therefore the structure of $N\subset P\subset M$ is completely understood. 
It is a little tricky to decide the position of $Q$, but we can handle it by using 
Lemma \ref{7lemma Q}. 
We use the symbols $\iota$, $\iota_Q$, $\kappa$, and $\kappa_Q$ as in Section 4. 

We recall the notation in \cite{I6} to describe the Longo-Rehren inclusion $N\subset P$. 
Let $\nu:\cM\hookrightarrow \cM_1$ be the inclusion map and let $\Delta=\{\rho_\xi\}_{\xi\in \Delta_0}$ 
be the set of irreducibles contained in $\cup_n(\bnu\nu)^n$. 
We may assume that $\rho_e=\id_\cM$ with $e\in \Delta_0$ and $[\overline{\rho_{\xi}}]=[\rho_{\bar{\xi}}]$. 
We identify $\cM^\opp$ with $\cM'=J_{\cM}\cM J_{\cM}$ and define $j:\cM\rightarrow \cM^\opp$ by 
$j(x)=J_\cM x J_\cM$ for $x\in \cM$. 
For $x\in \cM$, the corresponding element $x^\opp\in \cM^\opp$ is identified with $j(x^*)$. 
We denote by $\rho_\xi^\opp$ the endomorphism of $\cM^\opp$ corresponding to $\rho_\xi$. 
In our situation, we have $\rho_\xi^\opp=j\rho_\xi j^{-1}$. 
It is known that we have 
$[\iota\biota]=\bigoplus_{\xi\in \Delta_0}[\hrho_\xi],$
where $\hrho_\xi=\rho_\xi\otimes \rho_\xi^\opp$, and so $d(\iota)^2=\sum_{\xi\in\Delta_0}d(\xi)^2$. 
We choose an isometry $V_\xi\in (\rho_\xi\otimes \rho_\xi^\opp,\iota\biota)$ 
such that $r_{\biota}=V_e$, which will be denoted by $V$ for simplicity. 
We denote $r_\iota \in N$ by $W$, which is described as follows. 
Let $\{T(_{\xi,\eta}^\zeta)_i\}_{i=1}^{N_{\xi,\eta}^\zeta}$ be an 
orthonormal basis of $(\rho_\zeta,\rho_\xi\rho_\eta)$, and set 
$$T_{\xi,\eta}^\zeta=\sum_{i=1}^{N_{\xi,\eta}^\zeta}T(_{\xi,\eta}^\zeta)_i
\otimes j(T(_{\xi,\eta}^\zeta)_i),$$ 
which does not depend on the choice of the basis. 
Then we have 
$$W=\sum_{\xi,\eta,\zeta}\frac{\sqrt{d(\xi)d(\eta)}}{d(\iota)\sqrt{ d(\zeta)}}
V_\xi\hrho_\xi(V_\eta)T_{\xi,\eta}^\zeta V_\zeta^*.$$

To describe the conditional expectation $E_Q$, we use the following lemma:

\begin{lemma} \label{7lemma expectation}
Let $\cB\subset \cA$ be an irreducible inclusion of property infinite factors with 
finite index and let $\mu: \cB\hookrightarrow \cA$ be the inclusion map. 
Let $[\mu\bar{\mu}]=\bigoplus_{i=0}^mn_i[\sigma_i]$ be the irreducible decomposition and 
let $\{s(i)_l\}_{l=1}^{n_i}$ be an orthonormal basis of $(\mu,\sigma_i\mu)$. 
Then 
$$E_\cB(x)=\frac{1}{[\cA:\cB]}\sum_{i=0}^m\sum_{l=1}^{n_i}d(\sigma_i)s(i)_l^*\sigma_i(x)s(i)_l, 
\quad \forall x\in \cA.$$
\end{lemma}

\begin{proof} Let $\tilde{s}(i)_l=\sqrt{d(\sigma_i)}s(i)_l^*\sigma_i(r_{\bar{\mu}})$. 
Then $\{\tilde{s}(i)_l\}_{l=1}^{n_i}$ is an orthonormal basis of $(\sigma_i,\mu\bar{\mu})$ 
thanks to Frobenius reciprocity. 
Thus 
$$E_\cB(x) =\mu(r_\mu^*\bar{\mu}(x)r_\mu)=\mu(r_\mu^*)\mu\bar{\mu}(x)\mu(r_\mu) 
=\sum_{i=0}^m\sum_{l=1}^{n_i}\mu(r_\mu^*)\tilde{s}(i)_l\sigma_i(x)\tilde{s}(i)_l^*\mu(r_\mu).$$
On the other hand, 
$$\mu(r_\mu^*)\tilde{s}(i)_l=\sqrt{d(\sigma_i)}\mu(r_\mu^*)s(i)_l^*\sigma_i(r_{\bar{\mu}})
=\sqrt{d(\sigma_i)}s(i)_l^*\sigma_i(\mu(r_\mu^*)r_{\bar{\mu}})=\frac{\sqrt{d(\sigma_i)}}{d(\mu)}s(i)_l^*,$$
which shows the statement. 
\end{proof}

Recall that to compute $\Ang(P,Q)$, it suffices to determine the eigenvalues of $E_PE_Q$ restricted to 
$P$. 

\begin{lemma} \label{7lemma EPEQ}
Let $[\bnu\nu]=\bigoplus_{\xi \in \Delta_1}n_\xi[\rho_\xi]$ be the irreducible decomposition. 
Then for $x\in P$, we have 
$$E_PE_Q(x)=\frac{d(\iota)^2}{[\cM:\cN]}\sum_{\xi\in \Delta_1}\frac{n_\xi}{d(\xi)}W^*V_\xi\hrho_\xi(x)
V_\xi^*W.$$
\end{lemma}

\begin{proof}
Thanks to Lemma \ref{7lemma Q}, we have 
$[\kappa_Q\bkappa_Q]=\bigoplus_{\xi\in  \Delta_1}n_\xi[\id_{\cM_1}\otimes \rho_\xi^\opp]$.  
Thus to compute $E_Q$ using Lemma \ref{7lemma expectation}, it suffices to obtain 
$(\kappa_Q,(\id_{\cM_1\otimes \rho_\xi^\opp})\kappa_Q)$ for $\xi\in \Delta_1$.
Although it is not so easy to capture this space directly, we have  
\begin{eqnarray*}
(\kappa_Q,(\id_{\cM_1}\otimes \rho_\xi^\opp)\kappa_Q)&\subset&
(\kappa_Q\iota_Q,(\id_{\cM_1}\otimes \rho_\xi^\opp)\kappa_Q\iota_Q)\\
&=&((\nu\otimes \id_{\cM^\opp})\iota,(\nu\otimes \rho_\xi^\opp)\iota)\\
&\cong&(\iota\biota,\bnu\nu\otimes \rho_\xi^\opp).
\end{eqnarray*}
Comparing the dimensions of the both sides above, we get equality for the above inclusion.  

Let $\{t(\xi)_i\}_{i=1}^{n_\xi}$ be an orthonormal basis of $(\nu,\nu\rho_\xi)$.
We claim that 
$$\{\frac{d(\iota)}{\sqrt{d(\xi)}}(t(\xi)_i^*\otimes 1)V_\xi^* W\}_{i=1}^{n_\xi}$$ 
is an orthonormal basis of $((\nu\otimes \id_{\cM^\opp})\iota,(\nu\otimes \rho_\xi^\opp)\iota)$. 
Indeed, it is easy to show that $(t(\xi)_i^*\otimes 1)V_\xi^* W$ belongs to the intertwiner space  
$((\nu\otimes \id_{\cM^\opp})\iota,(\nu\otimes \rho_\xi^\opp)\iota)$, and so 
the operator $W^*V_\xi (t(\xi)_lt(\xi)_i^*\otimes 1)V_\xi^* W$ is already a scalar, 
which is equal to 
\begin{eqnarray*}
\lefteqn{
E_NE_P(W^*V_\xi (t(\xi)_lt(\xi)_i^*\otimes 1)V_\xi^* W)=
E_N(W^*V_\xi E_P(t(\xi)_lt(\xi)_i^*\otimes 1)V_\xi^* W)}\\
&=&\frac{\inpr{t(\xi)_l}{t(\xi)_i}}{d(\xi)}E_N(W^*V_\xi V_\xi^*W)
=\frac{\delta_{i,l}}{d(\xi)}W^*E_N(V_\xi V_\xi^*)W
=\frac{d(\xi)\delta_{i,l}}{d(\iota)^2}W^*W\\
&=&\frac{d(\xi)\delta_{i,l}}{d(\iota)^2}.
\end{eqnarray*}
This shows the claim. 
Thus Lemma \ref{7lemma expectation} shows that for $x\in P$ we have 
\begin{eqnarray*}
E_Q(x) &=& \frac{1}{[\cM:\cN]}\sum_{\xi\in \Delta_1}\sum_{i=1}^{n_\xi}d(\iota)^2 
W^*V_\xi(t(\xi)_i\otimes 1)(\id_\cM\otimes \rho_\xi^\opp)(x)(t(\xi)_i^*\otimes 1)V_\xi^*W\\
&=&\frac{d(\iota)^2}{[\cM:\cN]}\sum_{\xi\in \Delta_1}\sum_{i=1}^{n_\xi} 
W^*V_\xi\hrho_\xi(x)(t(\xi)_it(\xi)_i^*\otimes 1)V_\xi^*W.
\end{eqnarray*}
Therefore we get \begin{eqnarray*}
E_PE_Q(x) &=&\frac{d(\iota)^2}{[\cM:\cN]}\sum_{\xi\in \Delta_1}\sum_{i=1}^{n_\xi} 
W^*V_\xi\hrho_\xi(x)E_P(t(\xi)_it(\xi)_i^*\otimes 1)V_\xi^*W \\
 &=&\frac{d(\iota)^2}{[\cM:\cN]}\sum_{\xi\in \Delta_1}\frac{n_\xi}{d(\xi)} 
W^*V_\xi\hrho_\xi(x)V_\xi^*W.
\end{eqnarray*}
\end{proof}

To compute the eigenvalues of $E_PE_Q|_P$, we need the crossed product-like decomposition as in 
Theorem \ref{2theorem decomposition} based on the irreducible decomposition of $\biota\iota$. 
It is known that every sector contained in a power of $\biota\iota$ is of the form 
$\widetilde{\sigma}^\alpha$, where $\sigma$ is a finite direct sum of endomorphisms in $\Delta$ 
and $\alpha$ is a parameter of the corresponding half-braiding 
$\varepsilon_\sigma^\alpha(\xi)\in (\sigma\rho_\xi,\rho_\xi\sigma)$ 
(see \cite[Section 4]{I6} for details). 
Let 
$$U^\alpha(\sigma)=\sum_{\xi\in \Delta_0}V_\xi(\varepsilon_\sigma^\alpha(\xi)\otimes 1)
(\sigma\otimes \id_{\cM^\opp})(V_\xi^*).$$
Then $\widetilde{\sigma}^\alpha$ is the restriction of 
$\Ad U^\alpha(\sigma)(\sigma\otimes \id_{\cM^\opp})$ to $N$. 
Irreducible $\widetilde{\sigma}^\alpha$ is contained in $\biota\iota$ 
if and only if $\sigma$ contains $\id_\cM$. 
In this case, we have 
$$(\iota,\iota\widetilde{\sigma}^\alpha)=U^\alpha(\sigma)\big((\id_\cM,\sigma)\otimes \C\big).$$
Therefore to compute $\Ang(P,Q)$, it suffices to compute $E_PE_Q$ on this space thanks to 
Remark \ref{2remark angle}. 

\begin{lemma} \label{7lemma EPEQII}
Let $\widetilde{\sigma}^\alpha$ be an irreducible endomorphism of $N$ contained in 
$\biota\iota$. 
Then for $v\in (\id_\cM,\sigma)$, 
$$E_PE_Q(U^\alpha(\sigma)(v\otimes 1))=U^\alpha(\sigma)
(\phi_{\bnu\nu}(\varepsilon_\sigma^\alpha(\bnu\nu)v)\otimes 1).$$
The map $(\id_\cM,\sigma)\ni v\mapsto \phi_{\bnu\nu}(\varepsilon_\sigma^\alpha(\bnu\nu)v)$ depends 
only on the sector of $\bnu\nu$ and 
$$ \phi_{\bnu\nu}(\varepsilon_\sigma^\alpha(\bnu\nu)v)=\sum_{\xi\in \Delta_1}
\frac{n_\xi d(\xi)}{[\cM:\cN]} \phi_{\rho_\xi}(\varepsilon_\sigma^\alpha(\xi)v).$$
\end{lemma}

\begin{proof} It is routine work to show the second statement. 
Since we have 
$$U^\alpha(\sigma)(v\otimes 1)
=\sum_{\eta\in \Delta_0}V_\eta(\varepsilon_\sigma^\alpha(\eta)v\otimes 1)V_\eta^*,$$
Lemma \ref{7lemma EPEQ} implies 
\begin{eqnarray*}\lefteqn{
E_PE_Q(U^\alpha(\sigma)(v\otimes 1))}\\ 
&=&\frac{d(\iota)^2}{[\cM:\cN]}
\sum_{\xi\in \Delta_1}\sum_{\eta\in \Delta_0}\frac{n_\xi}{d(\xi)}W^*V_\xi
\hrho_\xi(V_\eta(\varepsilon_\sigma^\alpha(\eta)v\otimes 1)V_\eta^*)V_\xi^*W \\
&=&\frac{1}{[\cM:\cN]}\sum_{\eta,\zeta\in \Delta_0}\sum_{\xi\in \Delta_1}\sum_{i=1}^{N_{\xi,\eta}^\zeta}
\frac{n_\xi d(\eta)}{d(\zeta)}V_\zeta\big(
T(_{\xi,\eta}^\zeta)_i^* \rho_\xi(\varepsilon_\sigma^\alpha(\eta)v)T(_{\xi,\eta}^\zeta)_i
\otimes 1\big)V_\zeta^*.
\end{eqnarray*}
Thanks to the half-braiding relation, we have 
$$T(_{\xi,\eta}^\zeta)_i^* \rho_\xi(\varepsilon_\sigma^\alpha(\eta)v)
=\varepsilon_\sigma^\alpha(\zeta)\sigma(T(_{\xi,\eta}^\zeta))_i^*)\varepsilon_\sigma^\alpha(\xi)^*
\rho_\xi(v).$$
Thus Frobenius reciprocity implies
\begin{eqnarray*}\lefteqn{
\sum_{\eta\in \Delta_0}\sum_{i=1}^{N_{\xi,\eta}^\zeta}d(\eta)
T(_{\xi,\eta}^\zeta)_i^* \rho_\xi(\varepsilon_\sigma^\alpha(\eta)v)T(_{\xi,\eta}^\zeta)_i}\\ 
&=&\varepsilon_\sigma^\alpha(\zeta)\sum_{\eta\in \Delta_0}\sum_{i=1}^{N_{\xi,\eta}^\zeta}d(\eta)
\sigma(T(_{\xi,\eta}^\zeta))_i^*)\varepsilon_\sigma^\alpha(\xi)^*
\rho_\xi(v)T(_{\xi,\eta}^\zeta)_i\\
&=&d(\xi)d(\zeta)\varepsilon_\sigma^\alpha(\zeta)\sum_{\eta\in \Delta_0}\sum_{i=1}^{N_{\xi,\eta}^\zeta}
\sigma(r_{\rho_{\bar{\xi}}}^*\rho_\xi(T(_{\bar{\xi},\zeta}^\eta))_i))\varepsilon_\sigma^\alpha(\xi)^*
\rho_\xi(vT(_{\bar{\xi},\zeta}^\eta)_i^*)r_{\rho_{\bar{\xi}}}\\
&=&d(\xi)d(\zeta)\varepsilon_\sigma^\alpha(\zeta)\sum_{\eta\in \Delta_0}\sum_{i=1}^{N_{\xi,\eta}^\zeta}
\sigma(r_{\rho_{\bar{\xi}}}^*\rho_\xi(T(_{\bar{\xi},\zeta}^\eta))_i))\varepsilon_\sigma^\alpha(\xi)^*
\rho_\xi(\sigma(T(_{\bar{\xi},\zeta}^\eta)_i^*)v)r_{\rho_{\bar{\xi}}}\\
&=&d(\xi)d(\zeta)\varepsilon_\sigma^\alpha(\zeta)\sum_{\eta\in \Delta_0}\sum_{i=1}^{N_{\xi,\eta}^\zeta}
\sigma(r_{\rho_{\bar{\xi}}}^*\rho_\xi(T(_{\bar{\xi},\zeta}^\eta)_iT(_{\bar{\xi},\zeta}^\eta)_i^*))
\varepsilon_\sigma^\alpha(\xi)^*\rho_\xi(v)r_{\rho_{\bar{\xi}}}\\
&=&d(\xi)d(\zeta)\varepsilon_\sigma^\alpha(\zeta)\sigma(r_{\rho_{\bar{\xi}}}^*)
\varepsilon_\sigma^\alpha(\xi)^*\rho_\xi(v)r_{\rho_{\bar{\xi}}}.
\end{eqnarray*}
Using the half-braiding relation again, we see that this is equal to 
$$d(\xi)d(\zeta)\varepsilon_\sigma^\alpha(\zeta)
r_{\rho_{\bar{\xi}}}^*\rho_\xi(\varepsilon_\sigma^\alpha(\bar{\xi})v)r_{\rho_{\bar{\xi}}}
=d(\xi)d(\zeta)\varepsilon_\sigma^\alpha(\zeta)
\phi_{\rho_{\bar{\xi}}}(\varepsilon_\sigma^\alpha(\bar{\xi})v),$$
which finishes the proof.
\end{proof}

\begin{proof}[Proof of Theorem \ref{7theorem asymptotic}] 
Assume that the principal graph of $\cN \subset \cM$ is $A_n$. 
The structure of $N-N$ sectors associated with the inclusion $N\subset P$ is described in 
\cite[Section 7]{I6}. 
We may assume that there exists a system of endomorphism $\{\lambda_i\}_{i=0}^{n-1}$ 
of $\cM$ isomorphic to the irreducible sectors of the $SU(2)_{n-1}$ WZW model such that 
$\cN=\lambda_1(\cM)$. 
Then we have $[\bnu\nu]=[\lambda_1^2]$ and  $\Delta=\{\lambda_{2j}\}_{j=0}^{[(n-1)/2]}$. 
When $i-j$ is even, we set  $\sigma_{i,j}=\lambda_{i}\lambda_j$ and 
$\varepsilon_{\sigma_{i,j}}(\lambda_l)=\varepsilon^+(\lambda_i,\lambda_l)
\lambda_i(\varepsilon^-(\lambda_j,\lambda_l))$. 
Then $\{\varepsilon_{\sigma_{i,j}}(\lambda_l)\}_l$ is a half-braiding for $\sigma_{i,j}$ and 
we denote by $\widetilde{\sigma_{i,j}}$ the corresponding endomorphism of $N$. 

Assume first that $n$ is even. 
In this case, the endomorphism $\widetilde{\sigma_{2i,2i}}$ is irreducible and 
$$[\biota\iota]=\bigoplus_{i=0}^{n/2-1} [\widetilde{\sigma_{2i,2i}}].$$
Since $(\id_\cM,\sigma_{2i,2i})=\C r_{\lambda_{2i}}$, we get 
$$\Ang(P,Q)=\{
\cos^{-1}\sqrt{r_{\lambda_{2i}}^*\phi_{\lambda_1^2}(\varepsilon_{\sigma_{2i,2i}}(\lambda_1^2)
r_{\lambda_{2i}})}
;\; i=0,1,\cdots, \frac{n-2}{2}\}
\setminus \{0,\frac{\pi}{2}\}.$$
It is easy to show $r_{\lambda_{2i}}^*\phi_{\lambda_1^2}(\varepsilon_{\sigma_{2i,2i}}(\lambda_1^2)
r_{\lambda_{2i}})=\phi_{\lambda_1}(\varepsilon(\lambda_{2i},\lambda_1)\varepsilon(\lambda_1,\lambda_{2i}))^2$ 
and we get the statement. 

Assume now that $n=2s+1$ is odd. 
Then $\widetilde{\sigma_{i,i}}$ is irreducible for $i=0,1,2,\cdots s-1$ while 
$\widetilde{\sigma_{s,s}}$ is decomposed into two sectors, say $\mu_0$ and $\mu_1$ 
such that 
$$[\biota\iota]=\bigoplus_{i=0}^{s-1}[\widetilde{\sigma_{i,i}}]\oplus [\mu_0].$$
Then a similar computation as above shows 
$$\Ang(P,Q)=\{
\cos^{-1}|\phi_{\lambda_1}(\varepsilon(\lambda_{i},\lambda_1)\varepsilon(\lambda_1,\lambda_{i}))|
;\; i=1,2,\cdots, s\},$$
(since $E_PE_Q((\iota,\iota\widetilde{\sigma_{s,s}}))=0$ we have $E_PE_Q((\iota,\iota\mu_0))=0$ too) 
which shows the statement. 
\end{proof}
\section{Appendix}
In the proof of Theorem \ref{5theorem IV} we use the fact that there exists a unique 
$Q$-system for $\id_P\oplus \bkappa\theta\kappa$ up to equivalence. 
We give a proof of this statement here. 

The next lemma follows from Frobenius reciprocity \cite[Section 2.3]{I4}, which says that 
the right normalization of the norm of an element $v\in (\tau,\rho\sigma)$ is 
$\big(\frac{d(\rho)d(\sigma)}{d(\tau)}\big)^{1/4}$. 
 
\begin{lemma} \label{Alemma normalization}
Let $\cL$, $\cM$, and $\cN$ be properly infinite factors and 
$\rho\in \Mor_0(\cN,\cM)$, $\sigma\in \Mor_0(\cL,\cN)$, and $\tau\in (\cL,\cM)$ be 
irreducible morphisms. 
If $v\in (\tau,\rho\sigma)$ satisfies $||v||=\big(\frac{d(\rho)d(\sigma)}{d(\tau)}\big)^{1/4}$, then 
$$||\hpic{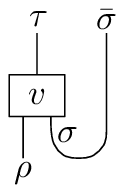}{1.0in}||=\big(\frac{d(\tau)d(\sigma)}{d(\rho)}\big)^{1/4},$$
$$||\hpic{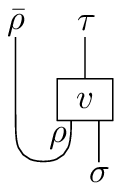}{1.0in}||=\big(\frac{d(\tau)d(\rho)}{d(\sigma)}\big)^{1/4}.$$
\end{lemma}

We recall the construction of the Haagerup subfactor in \cite[Section 7]{I7}. 
Let $\cO_4$ be the Cuntz algebra generated by isometries $\{S_0,T_0,T_1,T_2\}$ and 
let $d=(3+\sqrt{13})/2$ (which will be $d(\rho)$).  
We introduce an endomorphism $\rho$ and a period 3 automorphism $\alpha$ of $\cO_4$ by setting  
\begin{equation}
\alpha(S_0)=S_0,\quad \alpha(T_i)=T_{i+2},\end{equation}
\begin{equation}
\rho(S_0)=\frac{1}{d}S_0+\frac{1}{\sqrt{d}}\sum_{i\in \Z/3\Z}T_iT_i,\end{equation}
\begin{equation}
\rho(T_i) =\frac{1}{\sqrt{d}}S_0T_{-i}^*+T_{-i}S_0S_0^* 
 +\sum_{j,k\in \Z/3\Z}A(i+j,i+k)T_jT_{i+j+k}T_k^*,
 \end{equation}
where $i$ is understood as an element of $\Z/3\Z$ and 
$$A(0,0)=1-\frac{1}{d-1},$$
$$A(0,1)=A(0,2)=A(1,0)=A(2,0)=A(1,1)=A(2,2)=\frac{-1}{d-1},$$
$$A(1,2)=\overline{A(2,1)}=\frac{1+\sqrt{4d-1}\sqrt{-1}}{2(d-1)}.$$
Then $\rho$ and $\alpha$ extend to an irreducible 
endomorphism and an outer automorphism respectively 
of the weak closure $M$ of $\cO_4$ in the GNS representation of some KMS state, 
which will be denoted by the same symbols. 
Moreover we have $S_0\in (\id_M,\rho^2)$, $T_0\in (\rho,\rho^2)$, $T_1\in (\alpha\rho, \rho^2)$, and 
$T_2\in (\alpha^2\rho,\rho^2)$ and $\alpha\rho=\rho\alpha^2$. 
It is easy to show that $\rho$ satisfies the condition of Lemma \ref{3lemma Q-system}, namely 
\begin{equation}\rho(T_0)S_0=T_0S_0,\end{equation}
\begin{equation}\sqrt{d}S_0+(d-1)T_0^2=\sqrt{d}\rho(S_0)+(d-1)\rho(T_0)T_0.
\end{equation}
Therefore there exist a $Q$-system for $\id_M\oplus \rho$ and  
a subfactor $P\subset M$ such that 
$[\kappa\bkappa]=[\id_M]\oplus [\rho]$, where $\kappa:P\hookrightarrow M$ is the inclusion map. 
Note that $\rho=\heta$ and $\alpha=\theta$ in the notation of Theorem \ref{5theorem IV}.  
Lemma \ref{3lemma Q-system} shows that we can choose isometries $v=r_{\bkappa}$, $w=\kappa(r_\kappa)$, and 
$v_1\in (\rho,\kappa\bkappa)$ such that (3.3) holds: 
\begin{equation}
w=\frac{v}{\sqrt{d+1}}+\frac{v_1\rho(v)v_1^*}{\sqrt{d+1}}
+\sqrt{\frac{d}{d+1}}v_1\rho(v_1)S_0 v^*
+\sqrt{\frac{d-1}{d+1}}v_1\rho(v_1)T_0 v_1^*.
\end{equation}

To finish the proof of Theorem \ref{5theorem IV}, it suffices to show that there exists a unique 
$Q$-system for $\id_P\oplus \bkappa\alpha\kappa$ up to equivalence. 
For this purpose we solve the equations (3.1) and (3.2) with $\sigma:=\bkappa\alpha\kappa$. 
For this, we use the following notation: 
$$\hpic{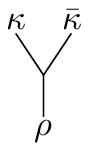}{0.8in}:=\big(\frac{d(\kappa)^2}{d}\big)^{1/4}v_1^*
=\big(\frac{d+1}{d}\big)^{1/4}v_1^*,\quad \hpic{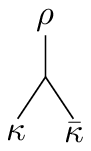}{0.8in}:=\big(\frac{d+1}{d}\big)^{1/4}v_1,$$
$$\hpic{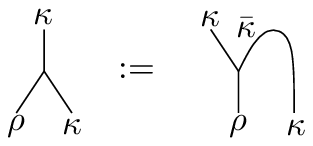}{0.8in},\quad \hpic{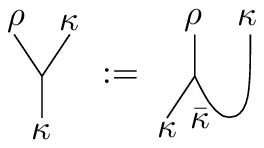}{0.8in},$$
$$\hpic{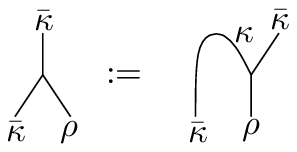}{0.8in},\quad \hpic{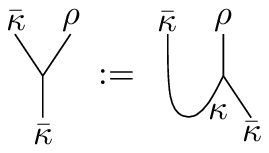}{0.8in}.$$
Every vertex expressing an element in $(\alpha\rho,\rho\alpha^2)$, $(\rho,\alpha\rho\alpha)$ etc. 
such as 
$$\hpic{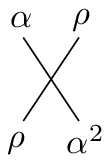}{0.8in},\quad \hpic{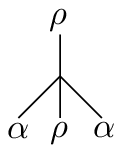}{0.8in}$$
will always mean 1. 
We choose $S_0$ as $r_\rho=\bar{r_\rho}$ and  set 
$$\hpic{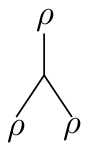}{0.8in}:=d^{1/4}T_0,\quad \hpic{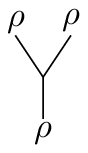}{0.8in}=d^{1/4}T_0^*.$$
Then 
$$\hpic{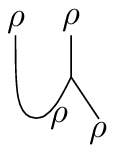}{0.8in}=d^{1/2}S_0^*\rho(d^{1/4}T_0)=d^{1/4}T_0^*=\hpic{A12}{0.8in},$$
$$\hpic{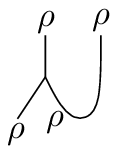}{0.8in}=\rho(d^{1/2}S_0^*)d^{1/4}T_0=d^{1/4}T_0^*=\hpic{A12}{0.8in}.$$
In the same way, we have
$$\hpic{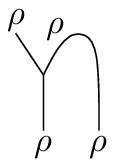}{0.8in}=\hpic{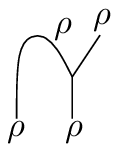}{0.8in}=\hpic{A11}{0.8in}.$$
Equation (8.4) means 
$$\hpic{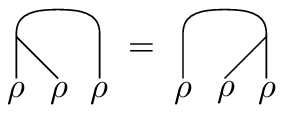}{0.5in}.$$
We use a similar expression for $d^{1/4}T_1\in (\alpha\rho,\rho^2)=(\rho\alpha^2,\rho^2)$ and 
$d^{1/4}T_2\in (\alpha^2\rho,\rho^2)=(\rho\alpha,\rho^2)$. 

Since we have 
$$\hpic{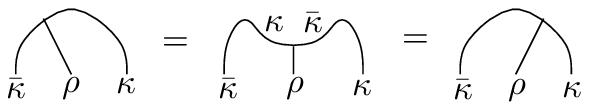}{0.5in},$$
we simply express this intertwiner by
\hpic{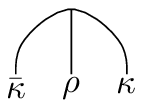}{0.5in}. 
In a similar way, we have 
$$\hpic{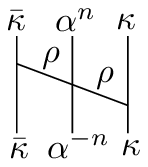}{0.8in}=\hpic{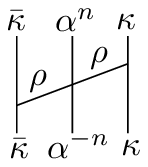}{0.8in},$$
and we simply express it by 
\hpic{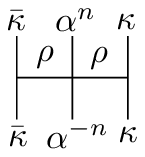}{0.7in}. 
The diagram \hpic{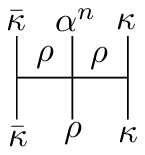}{0.7in} is also interpreted in the same way. 

We set 
$$R=\hpic{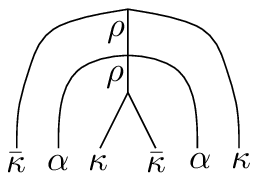}{0.9in}\in (\id_P,\sigma^2).$$
Then thanks to Lemma \ref{Alemma normalization}, we have 
$$||R||=\big(d(\kappa)^2\frac{d(\kappa)d(\rho)}{d(\kappa)}\frac{d(\kappa)^2}{d(\rho)}\big)^{1/4}
=d(\kappa)=\sqrt{d+1}.$$ 
All we have to show is the following: 

\begin{theorem} \label{Atheorem} 
Let the notation be as above. 
Then there exist exactly two elements $S\in (\sigma,\sigma^2)$ satisfying 
$||S||=d(\sigma)^{1/4}=(d+1)^{1/4}$ and
\begin{equation}\sigma(S)R=SR,
\end{equation}
\begin{equation}\frac{\sqrt{d+1}}{d}(R-\sigma(R))=\sigma(S)S-S^2.
\end{equation}
(Note that if $S$ satisfies the above condition, so does $-S$.)
\end{theorem}

To prove the theorem, we choose a basis of $(\sigma,\sigma^2)$. 
Note that we have 
\begin{eqnarray*}
\dim(\sigma,\sigma^2)&=&\dim(\bkappa\alpha\kappa,\bkappa\alpha\kappa\bkappa\alpha\bkappa)
=\dim(\kappa\bkappa\alpha\kappa\bkappa,\alpha\kappa\bkappa\alpha)\\
&=&\dim(\alpha\oplus \alpha\rho\oplus \alpha^2\rho\oplus \alpha^2\rho^2,\alpha^2\oplus \rho)=2.
\end{eqnarray*}
We set 
$$S_1=\hpic{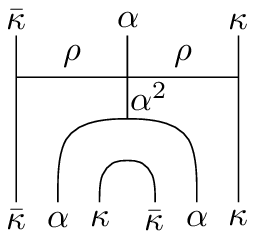}{1.0in}\in (\sigma,\sigma^2),$$
$$S_2=\hpic{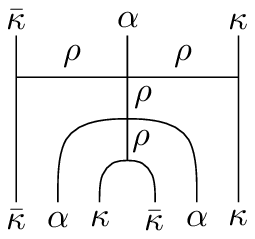}{1.0in}\in (\sigma,\sigma^2).$$
Then since $r_{\bkappa}$ and $v_1$ are orthogonal, so are $S_1$ and $S_2$. 

\begin{lemma} Let the notation be as above. 
Then $||S_1||=||S_2||=d(\sigma)^{1/4}=(d+1)^{1/4}$. 
\end{lemma}

\begin{proof} For $S_1$, we have 
$$||S_1||^2=d(\kappa)\hpic{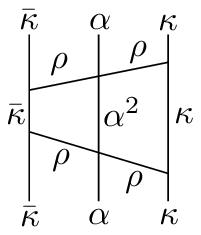}{1.0in}.$$
Since this is already a scalar, it is equal to
$$\hpic{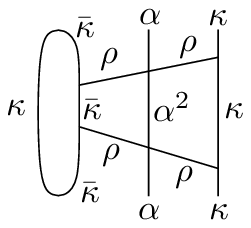}{1.0in}=\hpic{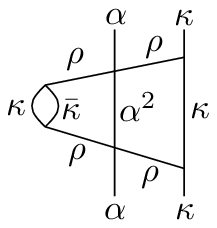}{1.0in}=\frac{d(\kappa)}{\sqrt{d}}\hpic{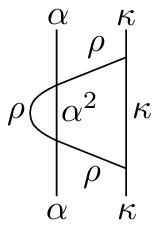}{1.0in}=d(\kappa).$$
For $S_2$, 
$$||S_2||^2=\frac{d(\kappa)}{\sqrt{d}}\hpic{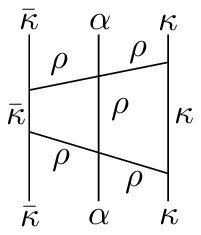}{1.0in}=\frac{1}{\sqrt{d}}\hpic{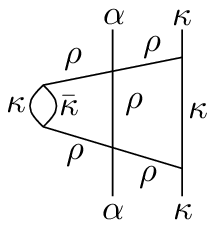}{1.0in}
=\frac{d(\kappa)}{d}\hpic{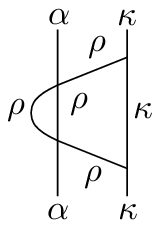}{1.0in}=d(\kappa).$$
\end{proof}

The above lemma shows that every element $S\in (\sigma,\sigma^2)$ satisfying $||S||=d(\sigma)^{1/4}$ is uniquely 
expressed as $S=aS_1+bS_2$, where $a$ and $b$ are complex numbers satisfying $|a|^2+|b|^2=1$. 
Therefore (8.7) and (8.8) are equivalent to the following equations respectively: 
\begin{equation*}aS_1R+bS_2R=a\sigma(S_1)R+b\sigma(S_2)R,
\end{equation*}
\begin{eqnarray*}\frac{\sqrt{d+1}}{d}(R-\sigma(R))&=&a^2(\sigma(S_1)S_1-S_1^2)
+ab(\sigma(S_1)S_2-S_1S_2+\sigma(S_2)S_1-S_2S_1)\\
&+&b^2(\sigma(S_2)S_2-S_2^2).
\end{eqnarray*}

The following lemma will be frequently used in what follows. 

\begin{lemma} \label{Alemma decomp}Let the notation be as above. Then 
$$\hpic{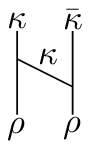}{0.6in}=\hpic{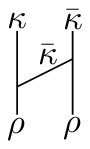}{0.6in}=\frac{1}{\sqrt{d}}\hpic{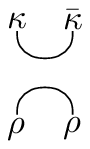}{0.6in}
+\sqrt{\frac{d-1}{d}}\hpic{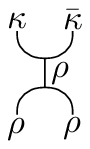}{0.6in}.$$
\end{lemma}

\begin{proof} The first equality is easy. 
For the second, we have 
$$\hpic{A33}{0.6in}=
\frac{1}{d(\kappa)}\hpic{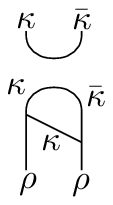}{0.8in}+\frac{\sqrt{d}}{d(\kappa)}\hpic{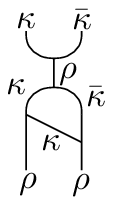}{0.8in}
=\frac{1}{d(\kappa)}\hpic{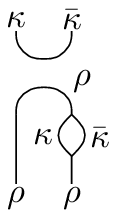}{1.0in}+\frac{\sqrt{d}}{d(\kappa)}\hpic{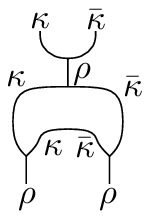}{1.0in}.$$
Thanks to Equation (8.6), we have  
$$\hpic{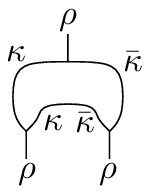}{0.8in}=\big(\frac{d(\kappa)^2}{d}\big)^{3/4}d(\kappa)^{1/2}v_1^*\rho(v_1^*)wv_1
=\frac{\sqrt{d^2-1}}{d}d^{1/4}T_0,$$
which shows the statement. 
\end{proof}

\begin{lemma} Let the notation be as above. 
Then we have 
$$S_1R=\sigma(S_1)R,\quad S_2R=\sigma(S_2)R.$$  
In particular, (8.7) always holds. 
\end{lemma}

\begin{proof} We introduce a linear isomorphism $F:(\id_P,\sigma^3)\rightarrow 
(\kappa\bkappa,\alpha\kappa\bkappa\alpha\kappa\bkappa\alpha)$ by
$$F(x)=\hpic{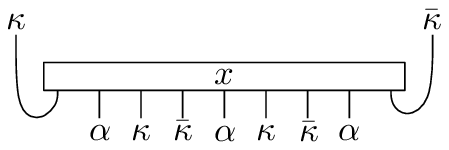}{0.6in}.$$
Since every intertwiner above belongs to the space $(\id_P,\sigma^3)$, it suffices to show the equalities 
after applying $F$ to the intertwiners. 

For $F(S_1R)$, we have
$$F(S_1R)=\hpic{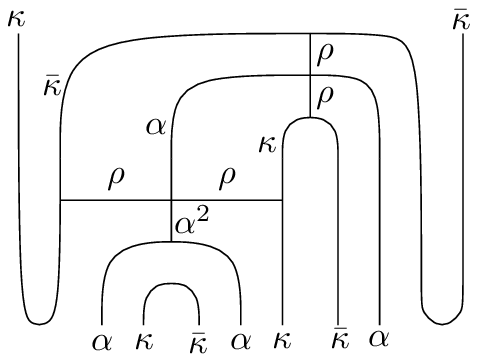}{1.2in}=\hpic{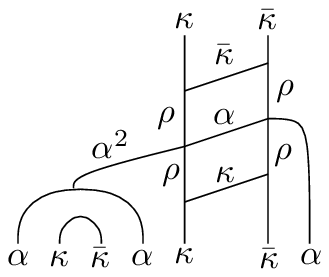}{1.2in}.$$
Using Lemma \ref{Alemma decomp} with consideration of $[\alpha^2(\id_M\oplus \rho)\alpha]=[\id_M]\oplus [\alpha\rho]$, 
we get 
$$F(S_1R)=\frac{1}{d}\hpic{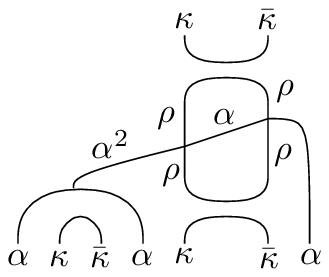}{1.2in}.$$
In a similar way, we have 
$$F(\sigma(S_1)R)=\frac{1}{d}\hpic{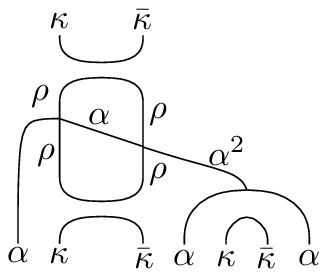}{1.2in}.$$
Since we have 
$$\hpic{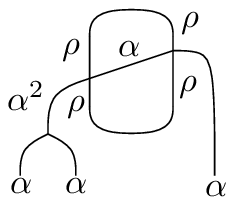}{1.0in}=d\alpha^2(S_0^*)S_0=d,$$
$$\hpic{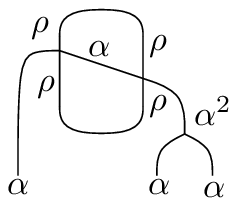}{1.0in}=d\alpha(S_0^*)S_0=d,$$
we get $F(S_1R)=F(\sigma(S_1)R)$. 

For $F(S_2R)$, we have
$$F(S_2R)=\hpic{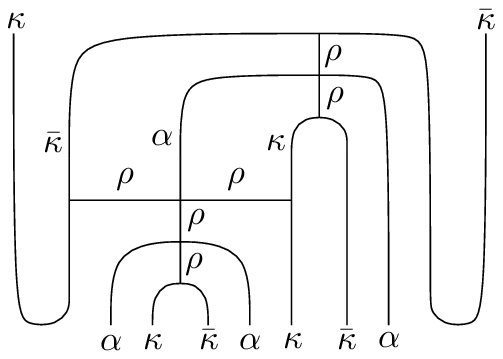}{1.2in}=\hpic{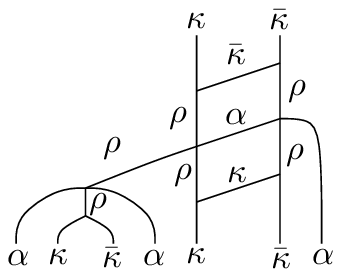}{1.2in}.$$
Using Lemma \ref{Alemma decomp} with consideration of $[\rho\;\id_M\alpha]=[\alpha\rho]$ and 
$[\rho\rho\alpha]=[\alpha]\oplus [\rho]\oplus [\alpha\rho]\oplus [\alpha^2\rho]$, we get 
$$F(S_2R)=\frac{d-1}{d}\hpic{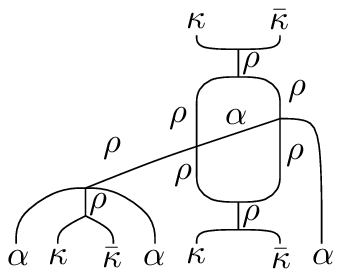}{1.2in}.$$
In a similar way, 
$$F(\sigma(S_2)R)=\frac{d-1}{d}\hpic{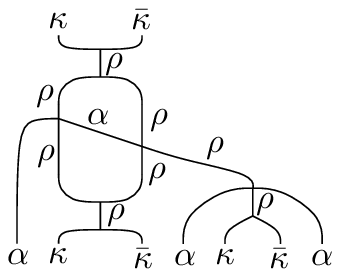}{1.2in}.$$
Since we have 
$$\hpic{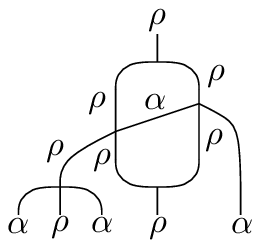}{1.2in}=d^{3/4}\rho(T_0^*)T_2T_0=d^{3/4}\overline{A(2,1)}T_1=d^{3/4}A(1,2)T_1,$$
$$\hpic{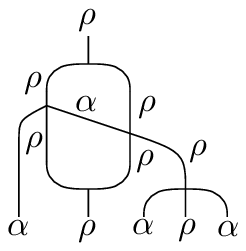}{1.2in}=d^{3/4}\alpha(T_0^*)\alpha\rho(T_1)T_0=d^{3/4}T_2^*\rho(T_2)T_0
=d^{3/4}A(1,2)T_1,$$
we get $F(S_2R)=F(\sigma(S_2)R)$. 
\end{proof}

\begin{lemma} Let $S=aS_1+bS_2$ with $a,b\in \C$. 
Then Equation (8.8) is equivalent to 
\begin{equation}
\frac{1}{\sqrt{d}}=a^2+\sqrt{d-1}ab,
\end{equation}
\begin{equation}
\sqrt{\frac{d-1}{d}}=ab-\frac{b^2}{\sqrt{d-1}}
\end{equation}
\begin{equation}0=a^2-\frac{ab}{\sqrt{d-1}}-A(1,2)b^2,
\end{equation}
\begin{equation}0=ab+\frac{1+(d-1)A(1,2)}{(d-1)^{3/2}}b^2. 
\end{equation}
\end{lemma}

\begin{proof} Let $G:(\sigma,\sigma^3)\rightarrow 
(\kappa\bkappa\alpha\kappa\bkappa,\alpha\kappa\bkappa\alpha\kappa\bkappa\alpha)$ be
the linear isomorphism defined in a similar way as in the proof of the previous lemma. 
We first compute $G(R),$ $G(\sigma(R))$, $G(S_1^2)$ etc. 

For $G(R)$, we have 
\begin{eqnarray*}G(R)&=&\hpic{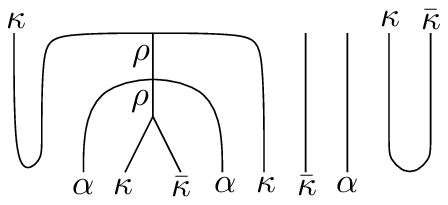}{0.8in}\\
&=&\frac{1}{d(\kappa)}\hpic{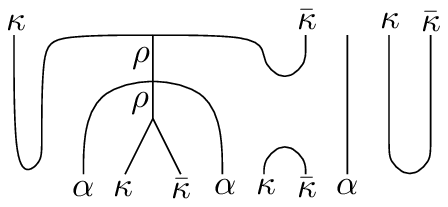}{0.8in}+\frac{\sqrt{d}}{d(\kappa)}\hpic{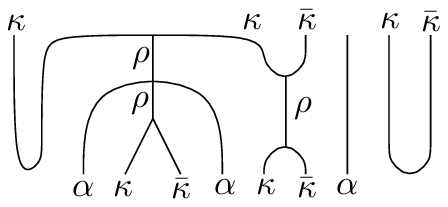}{0.8in}.
\end{eqnarray*}
Applying Lemma \ref{Alemma decomp} to the second term, we get
\begin{eqnarray*}
G(R)&=&\frac{1}{\sqrt{d+1}}\hpic{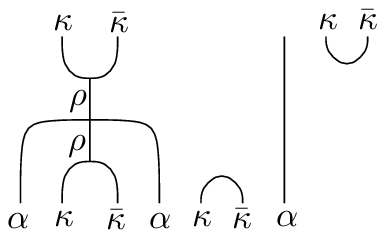}{0.8in} \\
 &+&\frac{1}{\sqrt{d+1}}\hpic{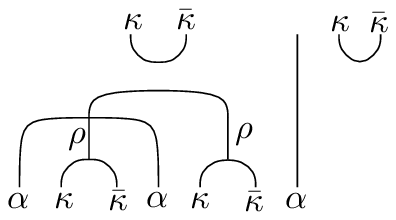}{0.8in}+\sqrt{\frac{d-1}{d+1}}\hpic{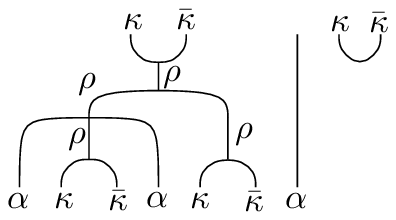}{0.8in}. 
\end{eqnarray*}
The intertwiner $G(\sigma(R))$ is the mirror image of $G(R)$. 
Since 
$$\hpic{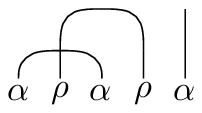}{0.4in}=\sqrt{d}S_0=\hpic{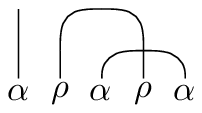}{0.4in},$$
we get 
\begin{eqnarray*}\lefteqn{\frac{\sqrt{d+1}}{d}G(R-\sigma(R))}\\
 &=& \frac{1}{d}\hpic{A59}{0.8in} -\frac{1}{d}\hpic{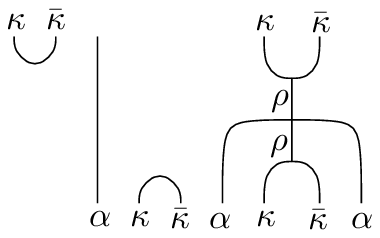}{0.8in} \\
 &+&\frac{\sqrt{d-1}}{d}\hpic{A61}{0.8in}-\frac{\sqrt{d-1}}{d}\hpic{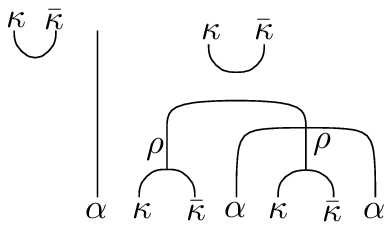}{0.8in}.
\end{eqnarray*}

For $G(S_1^2)$, Lemma \ref{Alemma decomp} implies  
\begin{eqnarray*}
G(S_1^2)&=&\hpic{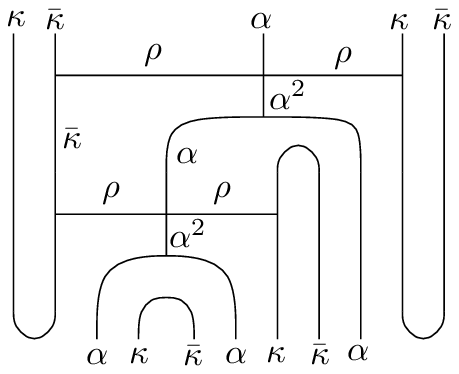}{1.2in}=\hpic{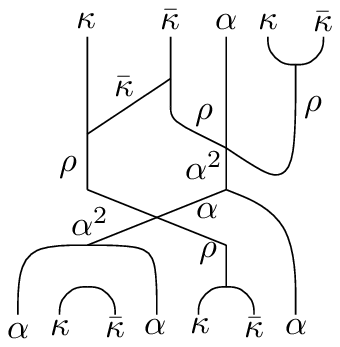}{1.2in}\\ 
&=&\frac{1}{\sqrt{d}}\hpic{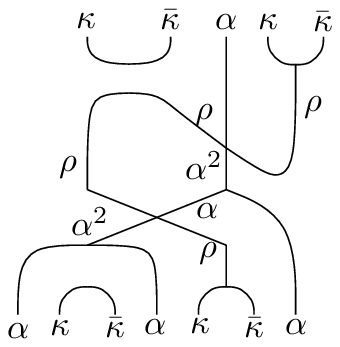}{1.4in}+\sqrt{\frac{d-1}{d}}\hpic{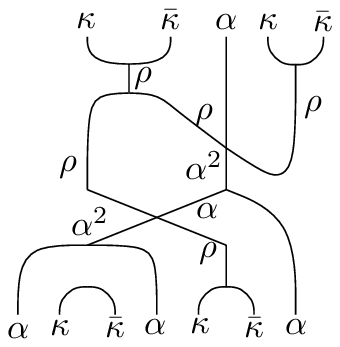}{1.4in}\\
&=&\frac{1}{\sqrt{d}}\hpic{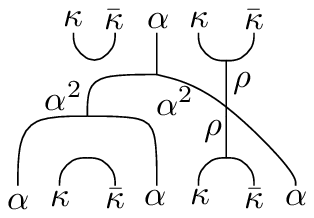}{1.4in}+\sqrt{\frac{d-1}{d}}\hpic{A69}{1.4in},\\
\end{eqnarray*}
where we use $\alpha(S_0)=S_0$. 
The intertwiner $G(\sigma(S_1)S_1)$ is the mirror image of $G(S_1^2)$. 
A similar argument shows 
$$G(S_2S_1)=\frac{1}{\sqrt{d}}\hpic{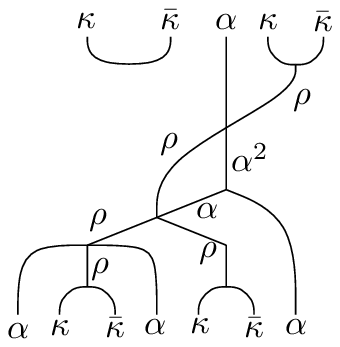}{1.4in}+\sqrt{\frac{d-1}{d}}\hpic{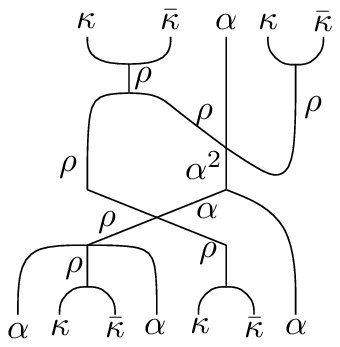}{1.4in},$$
and $G(\sigma(S_2)S_1)$ is its mirror image. 

For $G(S_1S_2)$, we have 
\begin{eqnarray*}
G(S_1S_2) &=& \hpic{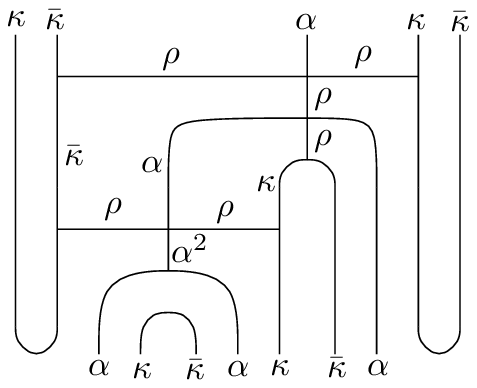}{1.4in}=\hpic{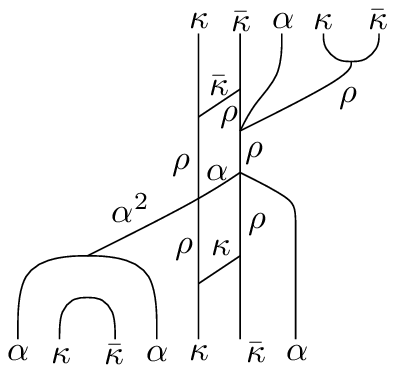}{1.4in}\\
&=&\frac{1}{d}\hpic{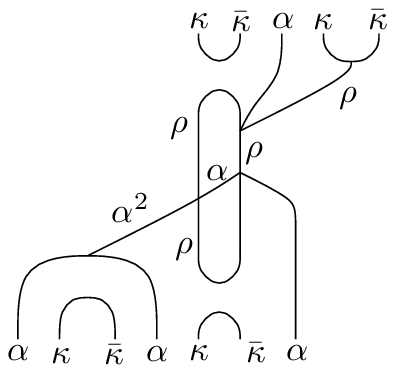}{1.4in}+\frac{\sqrt{d-1}}{d}\hpic{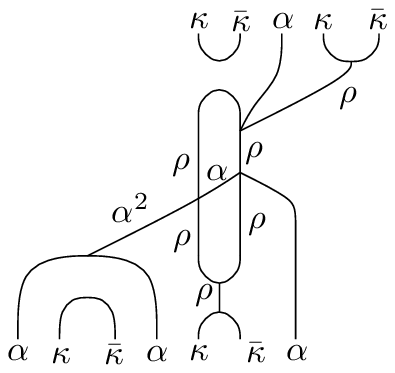}{1.4in}\\
&+&\frac{\sqrt{d-1}}{d}\hpic{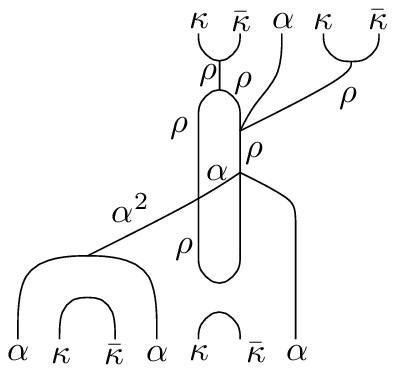}{1.4in}+\frac{d-1}{d}\hpic{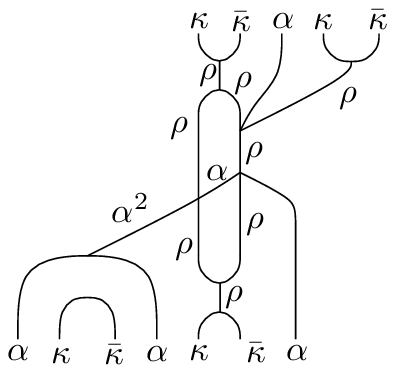}{1.4in}.
\end{eqnarray*}
Since $(\alpha^2\alpha,\alpha\rho)=(\alpha^2\alpha,\rho\alpha\rho)=0$, we get 
$$G(S_1S_2)=\frac{\sqrt{d-1}}{d}\hpic{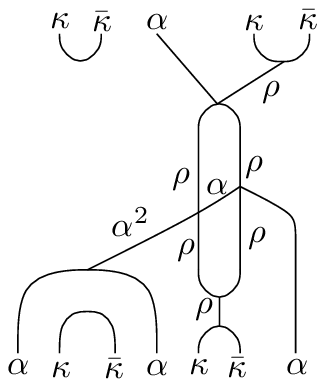}{1.4in}+\frac{d-1}{d}\hpic{A78}{1.4in}.$$
The intertwiner $G(\sigma(S_1)S_2)$ is the mirror image of $G(S_1S_2)$. 
A similar argument shows 
\begin{eqnarray*}
G(S_2^2) &=&\frac{\sqrt{d-1}}{d}\hpic{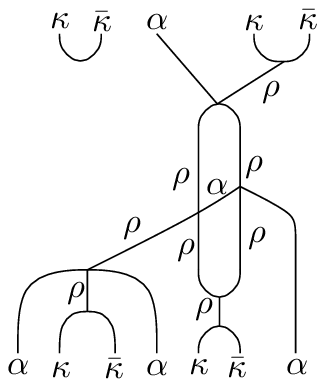}{1.4in}+\frac{\sqrt{d-1}}{d}\hpic{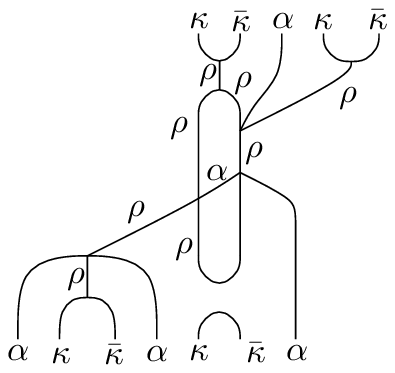}{1.4in}\\
&+&\frac{d-1}{d}\hpic{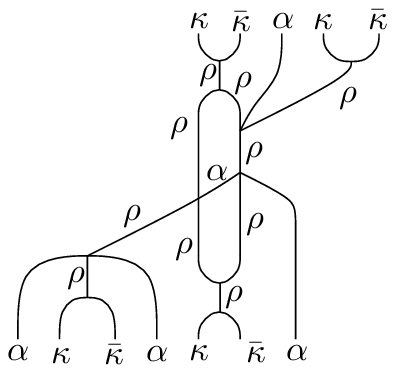}{1.4in}
\end{eqnarray*}
and $G(\sigma(S_2)S_2)$ is the mirror image of $G(S_2^2)$. 

Now (8.8) is equivalent to 
$$\frac{1}{d}\hpic{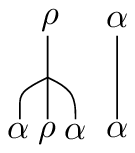}{0.5in}=\frac{a^2}{\sqrt{d}}\hpic{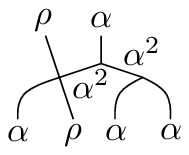}{0.6in}
+\frac{\sqrt{d-1}}{d}ab\hpic{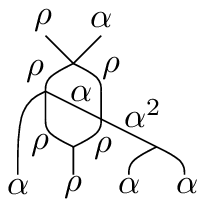}{0.7in},$$
$$\frac{1}{d}\hpic{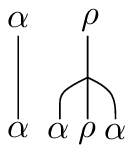}{0.5in}=\frac{a^2}{\sqrt{d}}\hpic{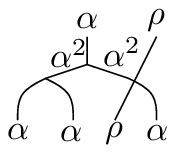}{0.6in}
+\frac{\sqrt{d-1}}{d}ab\hpic{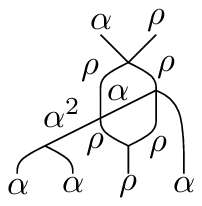}{0.7in},$$
$$\frac{\sqrt{d-1}}{d}\hpic{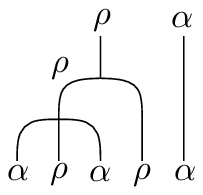}{0.7in}=\frac{ab}{\sqrt{d}}\hpic{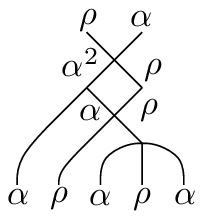}{0.7in}
+\frac{\sqrt{d-1}}{d}b^2\hpic{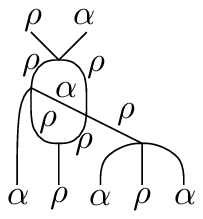}{0.7in},$$
$$\frac{\sqrt{d-1}}{d}\hpic{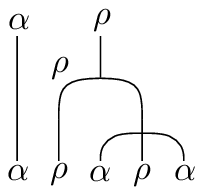}{0.7in}=\frac{ab}{\sqrt{d}}\hpic{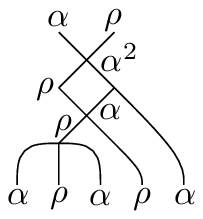}{0.7in}
+\frac{\sqrt{d-1}}{d}b^2\hpic{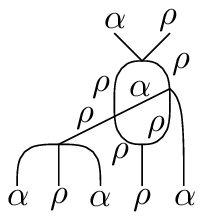}{0.7in},$$
\begin{eqnarray*}
0 &=&\sqrt{\frac{d-1}{d}}a^2\hpic{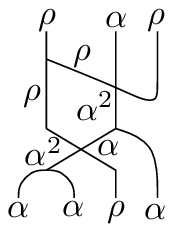}{1.0in} \\
 &+&\frac{d-1}{d}ab\hpic{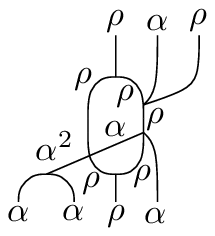}{1.0in}-\frac{\sqrt{d-1}}{d}b^2\hpic{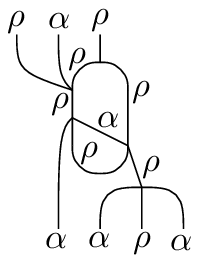}{1.0in},
\end{eqnarray*}
\begin{eqnarray*}
0 &=&\sqrt{\frac{d-1}{d}}a^2\hpic{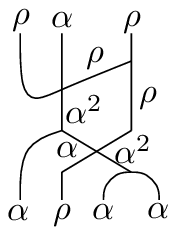}{1.0in} \\
 &+&\frac{d-1}{d}ab\hpic{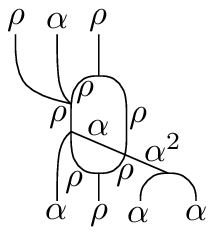}{1.0in}-\frac{\sqrt{d-1}}{d}b^2\hpic{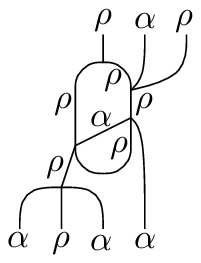}{1.0in},
\end{eqnarray*}
\begin{eqnarray*}
0&=&\sqrt{\frac{d-1}{d}}ab \hpic{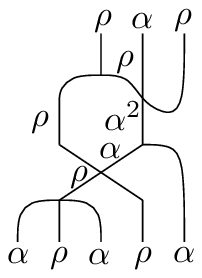}{1.2in}
-\sqrt{\frac{d-1}{d}}ab \hpic{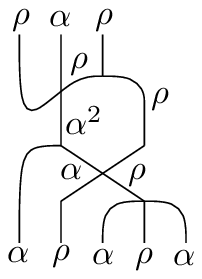}{1.2in} \\
&+&\frac{d-1}{d}b^2\hpic{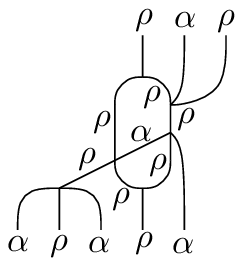}{1.2in} -\frac{d-1}{d}b^2\hpic{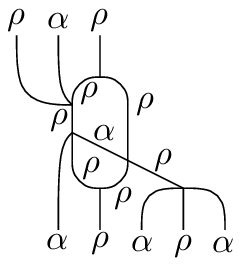}{1.2in}. 
\end{eqnarray*}

It is easy to show that the first and the second equations are equivalent to (8.9),  
the third and fourth are equivalent to (8.10), and the fifth and sixth are equivalent to 
(8.11). 
For the last equation, we have \begin{eqnarray*}
\hpic{A101}{1.2in}-\hpic{A102}{1.2in}
&=&dT_2\rho\alpha^2(S_0^*)T_0-d\alpha(T_1)S_0^*\rho\alpha(T_0)\\
&=&dT_2\rho(S_0^*)T_0-dT_0S_0\rho(T_2)=\sqrt{d}(T_2T_0^*-T_0T_1^*),
\end{eqnarray*}
\begin{eqnarray*}\lefteqn{
\hpic{A103}{1.2in} -\hpic{A104}{1.2in}}\\
&=&d\rho(T_0^*)T_2\rho(T_2^*)T_0 -d\alpha(T_0^*)\alpha\rho(T_1)T_2^*\rho\alpha(T_0)\\
&=&d\rho(T_0^*)T_2\rho(T_2^*)T_0 -dT_2^*\rho(T_2)T_2^*\rho(T_2)\\
&=&d\big(\sum_{l\in \Z/3\Z}A(2,l)T_{2+l}T_l^*\big)^*
 \big(\sum_{k\in \Z/3\Z}A(2,2+k)T_{2+k}T_k^*\big)^*\\
&-&d\big(\sum_{l\in \Z/3\Z}A(1,l+2)T_{1+l}T_l^*\big)\big(\sum_{k\in \Z/3\Z}A(1,k+2)T_{1+k}T_k^*\big)\\
&=&d\sum_{l\in \Z/3\Z}\overline{A(2,l)A(2,l+1)}T_lT_{l+1}^*
-d\sum_{k\in \Z/3\Z}A(1,k)A(1,k+2)T_{k-1}T_k^*\\
&=&d\big(\frac{1}{(d-1)^2}+\frac{A(1,2)}{d-1}\big)(T_2T_0^*-T_0T_1^*).
\end{eqnarray*}
Therefore the last equation is equivalent to (8.12).
\end{proof}

\begin{proof}[Proof of Theorem \ref{Atheorem}] Equation (8.11) shows $b\neq 0$ and so 
(8.12) implies 
\begin{equation}a=-\frac{B+1}{(d-1)^{3/2}}b.\end{equation}
where $B=(d-1)A(1,2)$. 
Using the fact that $B$ satisfies $B^2-B+d=0$, we can see 
that this is compatible with (8.11). 
Iterating (8.13) into (8.9), we get 
\begin{eqnarray*}
\frac{1}{\sqrt{d}} &=&\frac{(B+1)^2}{(d-1)^3}b^2-\frac{B+1}{d-1}b^2
=\frac{1}{(d-1)^3}\big(B^2+2B+1-(d-1)^2(B+1)\big)b^2 \\
 &=&\frac{1}{(d-1)^3}\big(3B-d+1-(d+2)(B+1)\big)b^2\\ 
 &=&\frac{1}{(d-1)^3}\big((1-d)B+d-d^2\big)b^2\\ 
 &=&-\frac{1}{(d-1)^2}\big(B+d)b^2,\\ 
\end{eqnarray*}
which shows 
\begin{equation}b^2=-\frac{(d-1)^2}{(B+d)\sqrt{d}}.
\end{equation}
We get the same equation from (8.10). 

Now all we have to show is that (8.13) and (8.14) are compatible with $|a|^2+|b|^2=1$.  
The equation (8.13) and $|a|^2+|b|^2=1$ imply 
$$\frac{1}{|b|^2}=1+\frac{1}{(d-1)^3}|B+1|^2=1+\frac{d+2}{(d-1)^3}=1+\frac{1}{d-1}
=\frac{d}{d-1}.$$
Since $|B+d|^2=d^2+2d=d(d-1)^2$, this is compatible with (8.14). 
\end{proof}
\bibliographystyle{amsplain}

\bigskip

\flushleft{Pinhas Grossman, Department of Mathematics, 1326 Stevenson Center, Vanderbilt University, Nashville, TN, 37203 \\
{\it e-mail}: pinhas.grossman@vanderbilt.edu}

\flushleft{Masaki Izumi, Department of Mathematics, Graduate
School of Science, Kyoto University, Sakyo-ku, Kyoto 606-8502,
Japan\\
{\it e-mail}: izumi@math.kyoto-u.ac.jp}

\end{document}